\documentclass[12pt,letter]{amsart}
\usepackage[centering,text={15.5cm,23cm}]{geometry} 

\usepackage{amscd}
\usepackage{amssymb}
\usepackage{amsfonts,amssymb, amscd, latexsym, graphicx, psfrag, color,float, tikz-cd}
\allowdisplaybreaks[1]

\usepackage{caption}
\usepackage{comment}
\usepackage{wrapfig}
\usepackage[all]{xy}
\usepackage{mathrsfs}
\usepackage{mathtools}
\usepackage{marvosym}
\usepackage{stmaryrd}
\usepackage{srcltx}
\usepackage{upgreek}
\usepackage[inline]{enumitem}
\usepackage{booktabs}
\usepackage{subcaption}
\usepackage{graphicx}

\usepackage[bookmarks=true, bookmarksopen=true,%
bookmarksdepth=3,bookmarksopenlevel=2,%
colorlinks=true,%
linkcolor=blue,%
citecolor=blue,%
filecolor=blue,%
menucolor=blue,%
urlcolor=blue]{hyperref}

\usepackage{tikz}
\usetikzlibrary{matrix,shapes,arrows,arrows.meta,calc,topaths,intersections,hobby,positioning,decorations.pathreplacing,decorations.pathmorphing,fit,patterns}
\tikzset{cross/.style={cross out, draw=black, fill=none, minimum size=2*(#1-\pgflinewidth), inner sep=0pt, outer sep=0pt}, cross/.default={2pt}}

\usepackage{braket}

\DeclareFontFamily{U}{mathx}{}
\DeclareFontShape{U}{mathx}{m}{n}{ <-> mathx10 }{}
\DeclareSymbolFont{mathx}{U}{mathx}{m}{n}
\DeclareFontSubstitution{U}{mathx}{m}{n}
\DeclareMathAccent{\widecheck}{0}{mathx}{"71}

\definecolor{mygray}{gray}{0.75} 
\definecolor{shadecolor}{rgb}{1,0.9,0.7}

\newtheorem{theorem}{Theorem}[section]
\newtheorem{lemma}[theorem]{Lemma}
\newtheorem{lemma-definition}[theorem]{Lemma-Definition}
\newtheorem{proposition}[theorem]{Proposition}
\newtheorem{corollary}[theorem]{Corollary}
\newtheorem{conjecture}[theorem]{Conjecture}

\newtheorem*{theoremx}{Theorem}

\newtheorem*{corollaryx}{Corollary}

\theoremstyle{definition}

\newtheorem{definition}[theorem]{Definition}

\newtheorem{convention}[theorem]{Convention}
\newtheorem{example}[theorem]{Example}

\newtheorem*{definitionx}{Definition}

\newtheorem{remark}[theorem]{Remark}

\numberwithin{equation}{section}
\numberwithin{figure}{section}

\newcommand {\lfor} {\llbracket}
\newcommand {\rfor} {\rrbracket}

\newcommand{\barM}{\overline{\mathcal{M}}}

\newcommand{\ZZ} {\mathbb{Z}}
\newcommand{\QQ} {\mathbb{Q}}

\newcommand{\CC} {\mathbb{C}}
\newcommand{\FF} {\mathbb{F}}
\newcommand{\PP} {\mathbb{P}}

\newcommand{\bC}{\mathbb{C}}

\newcommand{\bP}{\mathbb{P}}

\newcommand{\bZ}{\mathbb{Z}}

\newcommand{\cD}{\mathcal{D}}

\newcommand{\cI}{\mathcal{I}}

\newcommand{\cN}{\mathscr{N}}

\newcommand{\cO}{\mathcal{O}}

\newcommand{\cV}{\mathcal{V}}

\newcommand {\ua} {\underline{a}}

\newcommand {\im}  {\operatorname{im}}

\renewcommand {\ker } {\operatorname{ker}}

\newcommand {\liminv} {\varprojlim}

\newcommand {\NE}  {\operatorname{NE}}

\newcommand {\ord}  {\operatorname{ord}}

\newcommand {\qh} {q^{\frac{1}{2}}}
\newcommand {\qmh} {q^{-\frac{1}{2}}}
\newcommand{\q} {q}

\newcommand {\res}  {\operatorname{Res}}

\def\const{{\sf Const}}

\def\mydate{\ifcase\month \or January\or February\or March\or
April\or May\or June\or July\or August\or September\or October\or 
November\or December\fi \space\number\day,\space\number\year}

\makeatletter
\let\oldcite\cite
\renewcommand{\cite}{\@ifnextchar[{\@newcite}{\oldcite}}
\def\@newcite[#1]#2{\oldcite{#2},\,#1}
\makeatother

\title{Enumerative Geometry of Quantum Periods}

\author{Tim Gr\"afnitz}
\author{Helge Ruddat}
\author{Eric Zaslow}
\author{Benjamin Zhou}

\begin{document}

\begin{abstract}
We interpret the $q$-refined theta function $\vartheta_1$ of a log Calabi-Yau surface $(\bP,E)$ as a natural $q$-refinement of the open mirror map, defined by quantum periods of mirror curves for outer Aganagic-Vafa branes on the local Calabi-Yau $K_\bP$. The series coefficients are
all-genus logarithmic
two-point invariants, directly
extending the genus-zero relation found in \cite{GRZ}.  
We show that the winding-$1$ open-closed correspondence of Chan and Lau-Leung-Wu  \cite{Cha,LLW} generalizes to higher genus in the toric situation by using the Topological Vertex.  We conjecture a more general, higher-winding open-closed relation involving LMOV invariants.
Yet we find an explicit discrepancy with the log invariants of $(\bP,E)$, expressible in terms of stationary Gromov-Witten invariants of an elliptic curve.
\end{abstract}

\maketitle

{\tiny
\setcounter{tocdepth}{1}
\tableofcontents
}

\section{Introduction}

In \cite{GRZ}, the first three authors found new enumerative interpretations involving
periods of the mirror curve of a toric
Calabi-Yau threefold, $Y$.\footnote{
Other enumerative interpretations had previously
been found \cite{CLL,CLT,CCLT,CLLT,LLW}.  See 
in particular Theorem 1.1 and Example 4.6 of \cite{LLW} and Example 6.5 (3) of \cite{CCLT}. The conjecture in \cite{GRZ} was partly proved in \cite{You}.}  
The Picard-Fuchs equations determine
the periods, and in particular the open mirror map,
and its coefficients have two equal
interpretations:
(1)  two-point relative invariants used to define the first
theta function, or proper
Landau-Ginzburg potential;
(2) open invariants of an Aganagic-Vafa brane
on $Y.$
In this paper, we extend (1) to higher
genus by relating the $q$-refined
quantum periods of the quantum mirror
curve.  We find a discrepancy in (2) at
higher genus, and give an explicit expression
for the correction by analyzing the degeneration
formula.
Here, by a $q$-refinement we mean the following.

\begin{definition}
\label{def:q-refine}
A \emph{q-refinement} of $f\in R\llbracket t\rrbracket$ is $g\in R[q^{\pm\frac{1}{2}}]\llbracket t\rrbracket$ such that $g|_{q=1}=f$.
\end{definition}

\begin{example}
\label{ex:introex}
For example, when $Y = K_{\bP^2},$ the periods of (classical) local mirror symmetry are determined by
an ODE whose logarithmic solution has holomorphic
part $$F(z) = \sum_{k>0} (-1)^k\frac{(3k)!}{k(k!)^3}z^k$$ --- see, e.g.,
\cite[Section~2.1]{GRZ}.
The quantum period (see Example \ref{ex:ptf})
is $$a(z,q) = -\left(\qh + \qmh\right)z - \left(q^2 + \frac{7}{2}q + 6 + \frac{7}{2}q^{-1} + q^{-2}\right)z^2 + \cdots,$$
which $q$-refines $-\frac{1}{3}F(-z) = -2z - 15z^2 - \cdots.$ 
The non-refined period $-\frac{1}{3}F(-z)$ leads to open Gromov-Witten predictions for the Aganagic-Vafa (AV) brane in framing zero, which for genus zero and winding one in low degrees gives
$2, 5, 32, \cdots$.  These same numbers were also found in 
\cite{GRZ} to  be two-pointed logarithmic
invariants in $\bP^2$ in genus zero.  The
same procedure with the quantum period,
gives a $q$-refinement of these numbers --- specifically
$$\left(\qh + \qmh\right), \left(q^2+q+1+q^{-1}+q^{-2}\right),$$
$$\left( q^{\frac{9}{2}} + 3q^{\frac{7}{2}} + 4q^{\frac{5}{2}}+ 4q^{\frac{3}{2}}+ 4q^{\frac{1}{2}}+4q^{-\frac{1}{2}}+ 4q^{-\frac{3}{2}}+ 4q^{-\frac{5}{2}}+ 3q^{-\frac{7}{2}}+ q^{-\frac{9}{2}}\right), \cdots$$ --- see Example \ref{ex:qM}.
However, the $q$-dependence of the open Gromov-Witten invariants
for the AV brane is computed via the topological
vertex \cite[Equation (9.10)]{AKMV} to be
$$2, 5, (9q + 14 + 9q^{-1}),\cdots,$$ quite different.
\end{example}

\vskip0.01in

The aim of this paper is first to prove
that the quantum periods computed via the mirror symmetry procedure determine the all-genus logarithmic
two-point invariants, then to explain
the discrepancy from the open Gromov-Witten invariants.

\vskip0.17in

We now describe our results.
Let $\bP$ be a smooth projective Fano surface\footnote{The combinatorial Laurent series methods used in Section~\ref{S:agrees} work in any dimension, but for the treatment of higher genus invariants with $\lambda$-insertions, we have to restrict to surfaces.} and let $E$ be a smooth anticanonical divisor, so $(\bP,E)$ is a log Calabi-Yau pair. 
Section~\ref{sec:notation} contains our full notation conventions and Section~\ref{S:period} relevant definitions. 

\subsection{Quantum periods}

\begin{definitionx}[Definition \ref{def:aper}]
Let $p(X,Y)$ be the \emph{quantum mirror curve} defined by $\bP$ and write $p=1-f$. We define the \emph{quantum A-period} by
\[ a(z,q) := \const_{X,Y}(\log p(X,Y)) = - \sum_{n>0} \frac{1}{n}\const_{X,Y}(f^n). \]
In \cite{ACDKV} the quantum A-period is defined as
\[ \ua(z,q) = \const_X\left(\frac{\Psi(qX)}{\Psi(X)}\right), \]
where the \emph{wavefunction} $\Psi(X)$ is a solution to the difference equation $p\cdot \Psi=0$ defined by the quantum mirror curve, see Section~\ref{sec:wfns-pers}.
\end{definitionx}

\begin{theoremx}[Theorem \ref{thm:per}]
The definitions give identical series, $a(z,q) = \ua(z,q)$.
\end{theoremx}

To compute $\ua(z,q)$ from \cite{ACDKV}, one can write $Y(X):=\Psi(qX)/\Psi(X)$ as an ansatz $Y(X)=\sum_{k\geq 0}Y_k(X)s^k$, with $s=\qh z$ and then solve order-by-order in $s$, see Section~\ref{sec:wfns-pers}. Theorem \ref{thm:per} states that one can work with the alternative definition $a(z,q)$, which can be computed combinatorially.

\begin{definitionx}[Definition \ref{defi:desc}]
Let $q=e^{i\hbar}$.
The \emph{$q$-refined regularized GW period} of $\bP$ is
\[ \widehat{G}(z,q) = 1 + \sum_{\beta\in \textup{NE}(\bP)}\sum_{g\geq 0} (\beta\cdot E)! D_{g,1}(\bP,\beta)\hbar^{2g}z^\beta, \]
with descendant Gromov-Witten invariants
\[ D_{g,1}(\bP,\beta) = \int_{\mathcal{M}_{g,1}(\bP,\beta)} (-1)^g\lambda_g\psi^{\beta\cdot{E}-2}\textup{ev}^\star[\textup{pt}], \]
where $\lambda_g=c_g(\mathbb{E})$ are the top Chern classes of the Hodge bundle.
\end{definitionx}

\begin{theoremx}[Theorem \ref{thm:descendant}]
Let $D$ be defined as in Section \ref{sec:notation}.
\[ -Da(z,q)=\widehat{G}(z,q). \]
\end{theoremx}

The theorem is a $q$-refinement of the mirror description in \cite{CCGGK}, Definition 4.9, which states that the regularized quantum period of $\bP$ equals the classical period of its mirror potential. Theorem \ref{thm:descendant} is proved using the tropical correspondence for refined descendant invariants of \cite{KSU}.

\subsection{Natural \texorpdfstring{$q$}{q}-refinements of mirror maps}

Having a quantum period allows us to use the relation between periods and mirror maps to define $q$-refinements.

\begin{definitionx}[Definition \ref{def:mirmap}]
Put $U = xe^{a(-z,q)},$ a relation between
two formal parameters
$x$ and $U$ via the quantum period.  
We define $Q^\beta = z^\beta e^{-(\beta\cdot E)a(-z,q)}$. By Proposition \ref{prop:autom} there exists an inverse expression $z^\beta(Q,q)$.
We define the series
\[ M(Q,q) := e^{-a(-z(Q,q),q)} = \frac{U}{x} = \left(\frac{z^\beta(Q,q)}{Q^\beta}\right)^{-1/(\beta\cdot E)} \]
where the last equality holds for every non-zero class $\beta\in\NE(\PP)$. 
\end{definitionx}

\begin{remark}
The definition for $M(Q,q)$ is a natural $q$-refinement of the open mirror map $M(Q)$ in the sense that it uses the quantum period of the mirror curve of outer Aganagic-Vafa branes, analogous to how $M(Q)$ is defined. 
It is not known to us whether $M(Q,q)$ is analytic in a suitable domain. However, it specializes under $q=1$ to the analytic function $M(Q)$.
\end{remark}

\begin{definitionx}[Section \ref{S:theta}]
Let $\vartheta_1(y,z,q)_\infty$ be the $q$-refined proper Landau-Ginzburg potential (or primitive theta function at infinity) of $(\bP,{E})$. This is defined as the series enumerating $q$-refined broken lines
in a scattering diagram defined by a toric degeneration of $(\bP,{E})$ in a chamber asymptotically close to $E$. By \cite{GRZ}, Theorems 4.14, 4.8 and 5.4, we can equivalently define $\vartheta_1(y,z,q)_\infty$ as a generating function of $2$-marked logarithmic Gromov-Witten invariants, with $q=e^{i\hbar}$,
\[ \vartheta_1(y,z,q)_\infty = y + \sum_{\beta\in \textup{NE}(\bP)} \sum_{g\geq 0} \frac{1}{\beta\cdot{E}-1}R_{g,(1,\beta\cdot{E}-1)}(\bP(\log E),\beta)\hbar^{2g}z^\beta y^{-(\beta\cdot{E}-1)}. \]
\end{definitionx}

\begin{theoremx}[Theorem \ref{thm:agrees}]
Under the change of variables $Q^\beta=z^\beta\cdot (-1/y)^{\beta\cdot E}$ we have
\[ \vartheta_1(y,z,q)_\infty = yM(Q,q). \]
\end{theoremx}

\begin{corollaryx}[Corollary \ref{cor:agrees}]
Let $q=e^{i\hbar}$. Then
\[ M(Q,q) = 1+\sum_{\beta\in \textup{NE}(\bP)}\sum_{g\geq 0} \frac{(-1)^{\beta\cdot E}}{\beta\cdot{E}-1}R_{g,(1,\beta\cdot{E}-1)}(\mathbb{P}(\log E),\beta)\hbar^{2g}Q^\beta. \]
\end{corollaryx}

\begin{example}
\label{ex:gwex}
For $\bP^2$, the quantum period is (see Example \ref{ex:ptf}),
\[ a(z,q) = -\left(\qh + \qmh\right)z - \left(q^2 + \frac{7}{2}q + 6 + \frac{7}{2}q^{-1} + q^{-2}\right)z^2 + \cdots. \]
The natural $q$-refinement of the closed mirror map is
\[ Q = ze^{-3a(-z,q)} = z - \left(3\qh+3\qmh\right)z^2 + \left(3q^2+15q+27+15q^{-1}+3q^{-2}\right)z^3 + \cdots, \]
with inverse
\[ z(Q,q) = Q + \left(3\qh+3\qmh\right)Q^2 + \left(-3q^2+3q+9+3q^{-1}-3q^{-2}\right)Q^3 + \cdots. \]
The definition gives
$$
\begin{array}{l}
M(Q,q) = \left(\frac{z(Q,q)}{Q}\right)^{-1/3} = 1 - \left(\qh + \qmh\right)Q + \left(q^2+q+1+q^{-1}+q^{-2}\right)Q^2 \\[5pt]
\ \ - \left( q^{\frac{9}{2}} + 3q^{\frac{7}{2}} + 4q^{\frac{5}{2}}+ 4q^{\frac{3}{2}}+ 4q^{\frac{1}{2}}+ 4q^{-\frac{1}{2}}+ 4q^{-\frac{3}{2}}+ 4q^{-\frac{5}{2}}+ 3q^{-\frac{7}{2}}+ q^{-\frac{9}{2}}\right)Q^3 + \cdots.
\end{array}
$$
\end{example}
Theorem \ref{thm:agrees} and Corollary \ref{cor:agrees} extend the results of \cite{GRZ} to the $q$-refined setting. We prove Theorem \ref{thm:agrees} using combinatorial results that are proved in Section \ref{sec:combinatorics}. These include a Lagrange inversion formula for formal Laurent series and an exponentiation formula for classical periods of these. The latter follows from a relation for partial ordinary Bell polynomials proved in \cite{BGW}. Then Corollary \ref{cor:agrees} follows from the tropical correspondence results in \cite{GRZ}. The combinatorial arguments of Section \ref{S:agrees} work in any dimension. 
For the $q$-refined tropical correspondence of Section \ref{S:3} we need surfaces.

\subsection{Open-log correspondence}
In Section \ref{S:openlog}, we prove an all-genus, open-log correspondence between two-pointed log invariants $R_{g, (1, \beta\cdot E-1)}(\bP(\log E), \beta)$ and certain open invariants of $K_{\bP}.$
\subsubsection{Log-local correspondence and correction terms}
\label{S:14}
Let $\pi:\widehat{\bP} \rightarrow \bP$ be the blow up at a point with exceptional curve $C$, and let $n_g(K_{\widehat{\bP}}, \pi^*\beta - C)$ be the genus-$g$ Gopakumar-Vafa invariant of $K_{\widehat{\bP}}$ in curve class $\pi^*\beta - C$ defined by multiple cover formulas \cite{GV1} \cite{GV2}. In Section \ref{sec:log_GV}, we first show the following theorem relating $R_{g, (1, \beta\cdot E-1)}(\bP(\log E), \beta)$ and $n_g(K_{\widehat{\bP}}, \pi^*\beta - C)$.

\begin{theoremx}[Theorem \ref{thm:tropical_local}]

Let $(\mathbb{P}, E)$ be a smooth log Calabi-Yau pair. Then,

\begin{multline*}
    \sum_{\substack{g \geq 0, \\ \beta \in \textup{NE}(\bP)}}\frac{1}{(\beta\cdot E-1)}R_{g, (1, \beta\cdot E-1)}(\bP(\log E), \beta) \hbar^{2g}Q^{\beta} =\\
    \sum_{\substack{g \geq 0, \\ \beta \in \textup{NE}(\bP)}}\left[(-1)^{\beta\cdot E}n_g(K_{\widehat{\bP}}, \pi^*\beta-C)
    \left(2\sin \frac{\hbar}{2}\right)^{2g-2}Q^{\beta}\right] - \Delta  
\end{multline*}
where $\Delta$, defined in Equation \ref{eq:delta_ol}, is a discrepancy expressed by the stationary Gromov-Witten theory of $E$ and two-pointed log invariants of $\bP(\log E)$.
\end{theoremx}

We prove Theorem \ref{thm:tropical_local} by using higher genus analogues of arguments in \cite{GRZ}, Theorem 1.1. The discrepancy term $\Delta$ comes from application of the $g\ge 0$ log-local principle of \cite{BFGW} and a blow up formula for log invariants in \cite{GRZ}, Corollary 6.6 and is expressed by the stationary Gromov-Witten theory of $E$ and, after fixing a genus-$g$, the two-pointed log invariants $R_{h, (1, \beta\cdot E-1)}(\bP(\log E), \beta)$ for $h < g.$ 

\subsubsection{Open-closed correspondence}

\label{sec:intro_winding}

In Section \ref{sec:winding-w}, we additionally assume that $\bP$ is toric and $\pi: \widehat{\bP}\to \bP$ is a toric blow up, so that $K_{\bP}, K_{\widehat{\bP}}$ are toric Calabi-Yau 3-folds. Let $L \subset K_{\bP}$ be an outer Aganagic-Vafa brane. Denote $N^{\mathrm{LMOV}}_{g, \nu}(K_{\bP}/L, \beta)$ to be the genus-$g$, LMOV invariant of $L$ in representation $\nu$, and $n_g^{open}(K_{\bP}/L, \beta+w\beta_0, \vec{k})$ be the genus-$g$, winding-$w$, open-BPS invariant of $L$ in winding profile $\vec{k}$, curve class $\beta + w\beta_0 \in H_2(K_{\bP}, L)$ defined by multiple cover formulas \cite{MV} (see beginning of Section \ref{S:openlog} and Section \ref{sec:winding-w} for more details of these two invariants). The LMOV invariants $N^{\mathrm{LMOV}}_{g, \nu}$ are related to the open-BPS invariants $n^{open}_g$ by a linear transformation (Equation \ref{eq:change_of_basis}). In Theorem \ref{thm:winding}, for fixed $w \in \mathbb{Z}_{>0}$, we prove an equality relating LMOV invariants of $L \subset K_{\bP}$ in representations of length $w$ to the closed partition function of $K_{\widehat{\bP}}$ in curve classes $\pi^*\beta - wC$. Specializing Theorem \ref{thm:winding} to $w = 1, 2$ gives the following correspondence of enumerative invariants.

\begin{theoremx}[Corollaries \ref{cor:winding-1}, \ref{cor:winding-2}]
Suppose $\bP$ is a toric Fano surface, $\pi:\widehat{\bP} \rightarrow \bP$ is a toric blow up at a point with exceptional curve $C$, and let $K_{\bP}, K_{\widehat{\bP}}$ be the toric canonical bundles. Then for $\beta \in NE(\bP), g \geq 0, w = 1, 2$, we have,
\[ n_g(K_{\widehat{\bP}}, \pi^*\beta-wC) = (-1)^{g}N^{LMOV}_{g, (w)}(K_{\bP}/L, \beta) \]
where $(w)$ denotes the representation given by a Young Tableau of a single row of length $w$.
\end{theoremx}
In genus-0, Corollary \ref{cor:winding-1} is equivalent to \cite{LLW}, Theorem 1.1 after converting from LMOV invariants to open Gromov-Witten invariants using Equation \ref{eq:change_of_basis} and the open multiple cover formula \cite{MV}. In Section \ref{S:verification}, we provide a computational verification for Corollary \ref{cor:winding-1} for $\bP = \bP^2$ in low degrees and all-genus.  That the theorem holds for all positive windings $w$ is the content of Conjecture \ref{conj:GV_LMOV}.  We provide partial verification with LMOV and Gopakumar-Vafa invariants listed in Appendix \ref{sec:LMOV_GV}.

By substituting Corollary \ref{cor:winding-1} into Theorem \ref{thm:tropical_local}, we show an all-genus open-log correspondence in Theorem \ref{thm:ol_correspondence} relating two-pointed log invariants $R_{g, (1, \beta\cdot E-1)}(\bP(\log E), \beta)$ to open-BPS invariants $n^{open}_g(K_\bP/L, \beta+\beta_0, (1))$ for smooth log Calabi-Yau pairs. In genus-$0$, Theorem \ref{thm:ol_correspondence} reduces to the open-log correspondence established in \cite{GRZ}. We give an explicit example of Theorem \ref{thm:ol_correspondence} in Example \ref{ex:explicitdelta}.

\subsection{Notation}
\label{sec:notation}
Let $\bP$ be a smooth projective Fano surface and let $E$ be a smooth anticanonical divisor, so $(\bP,E)$ is a log Calabi-Yau pair. Let $\textup{NE}(\bP)$ be the monoid of effective curve classes. For elements of $\CC\llbracket \textup{NE}(\bP)\rrbracket $ we write $z^\beta$ or $Q^\beta$, depending on the context. We regard $z$ as the complex structure modulus and $Q$ as the symplectic structure modulus. We define an operation $D$ on $\CC\llbracket \textup{NE}(\bP)\rrbracket $ by $Dz^\beta=(\beta\cdot E)z^\beta$. Note that this operation resembles a log derivative.

The mirror curve is a moduli space for A-branes on the local Calabi-Yau threefold $K_\bP$, see Section \ref{S:mircurve}. We denote a local coordinate on the mirror curve by $x$. Mirror dual to $x$ is a symplectic area modulus of B-branes, which we denote by $U$. 

Let $R$ be a ring and let $f \in R\llbracket y_1^\pm, ..., y_s^\pm\rrbracket $ be a formal Laurent series with coefficients in $R,$ so $f = \sum_{m\in \bZ^s} f_m y^m.$  We define $\const_{y_1,\ldots,y_s}(f):= f_0 \in R$ to be the constant term.
For a Laurent series $f=\sum_{k\in\mathbb{Z}}a_ky^k$ and $n\in\mathbb{Z}$, we write $[y^n]f=a_n$ for the $n$-th coefficient of $f$. In particular, $\const_y(f)=[y^0]f$.

We write $R_{g,(p,q)}(\bP(\log E),\beta)$ for the logarithmic Gromov-Witten invariant counting curves on $\bP$ of genus-$g$ and class $\beta$ intersecting the divisor $E$ in two points, a prescribed point with multiplicity $p$ and a non-prescribed point with multiplicity $q$. Let $\beta_0 \in H_2(K_{\bP}, L) $ be a disc class with boundary on an outer Aganagic-Vafa brane $L$. We write $O_g(K_{\bP}/L, \beta+w\beta_0, \vec{k})$ for the genus-$g$, winding-$w$, open Gromov-Witten invariant of an outer Aganagic-Vafa brane $L \subset K_{\bP}$ in framing-0 and winding profile $\vec{k}$, and $n^{open}_g(K_\bP/L, \beta+w\beta_0, \vec{k})$ to be its corresponding open-BPS invariant. We write $N^{LMOV}_{g, \nu}(K_{\bP}/L, \beta)$ to be the genus-$g$, LMOV invariant of $L \subset K_{\bP}$ in representation $\nu$. We write $N_g(K_{\widehat{\bP}}, \pi^*\beta-wC)$ to be the genus-$g$ local Gromov Witten invariant of $\widehat{\bP}$ in curve class $\pi^*\beta-wC$ and $n_g(K_{\widehat{\bP}}, \pi^*\beta-wC)$ to be its corresponding Gopakumar-Vafa invariant.

\subsection{Acknowledgments}
We would like to thank Kwokwai Chan, Babak Haghighat, Melissa Liu, Daniel Krefl, Ezra Getzler and Jörg Teschner for helpful exchanges.
E.Z.~has been supported by NSF grants DMS-2104087 and DMS-2505652. H.R. is supported by the NFR Fripro grant Shape2030.

\section{Quantum Mirror Curves, Periods and Wavefunctions}
\label{S:period}

In this section, we give two definitions of 
quantum periods and show they are equivalent.
Local mirror symmetry determines
a mirror curve $\Sigma = \{p(X,Y)=0\} \subset (\bC^\times)^2$.  In the quantum version,
the $\bC$ coordinates $\log X$ and $\log Y$ are conjugate variables, meaning $YX = qXY$,
and the
defining equation $p=0$ defines a left ideal for this algebra.  A cyclic vector $\Psi$ obeying $p\cdot \Psi=0$ is called a wavefunction.
From here there are two routes to defining quantum periods.  On the one hand, the classical period of the
mirror curve (see \cite{CCGGK}) can be quantized in a natural way.
On the other hand, a quantum A-period can be constructed from the wavefunction, following \cite{ACDKV}. 
In the following subsections, we review these constructions,
then prove in Theorem \ref{thm:per} that they
are equal.

\subsection{Fano surfaces}
\label{S:fano}

Let $\bP$ be a smooth surface with very ample anticanonical bundle, i.e., one of the $8$ del Pezzo surfaces given by $\bP^1\times\bP^1$ or the blow up of $\bP^2$ in $0\leq k \leq 6$ points. Let $\bP_{\Sigma'}$ be a Gorenstein toric Fano surface to which $\bP$ admits a $\QQ$-Gorenstein degeneration. If $\bP$ is toric, we can simply take $\bP_{\Sigma'}=\bP$. Let $\Delta^\star$ be the reflexive polygon given by the convex hull of the primitive integral ray generators of $\Sigma'$. The ray generators are exactly the vertices of $\Delta^\star$. Any mutation of $\Delta^\star$ gives another Gorenstein toric Fano surface to which $\bP$ admits a $\QQ$-Gorenstein degeneration \cite{Ilten}. Let $\Sigma$ be the subdivision of $\Sigma'$ whose ray generators are all integral points on the boundary of $\Delta^\star$, not just the vertices. Then $\bP_\Sigma$ is a smooth semi-Fano toric surface to which $\bP$ admits a $\QQ$-Gorenstein degeneration. Again, if $\bP$ is toric we can take $\bP_\Sigma=\bP$. There are exactly 16 possibilities for $\bP_{\Sigma'}$, hence for $\bP_\Sigma$, corresponding to the 16 reflexive polygons. See Section \ref{S:theta} and \cite{CPS} for more details and \cite{Gra1}, Figure 1.3, for a complete list of cases.

\subsection{Definition of the mirror curve}
\label{S:mircurve}

The mirror of the canonical bundle of a Fano surface $K_\bP$ is a threefold $Z\subset \bC^2\times (\bC^\times)^2$ described by an equation $uv = p(x,y)$. 
The Riemann surface $\{p = 0\}\subset (\bC^\times)^2$ is called the (classical) mirror curve.
It is the moduli space of holomorphic B-branes supported on $\{u=0\}\cong \bC$, mirror to an Aganagic-Vafa 
A-brane $L\subset K_\bP$.

We view the mirror curve as being fibered over $\text{Spec }\mathbb{C}[\textup{NE}(\bP)]$ via the closed moduli $z^\beta$, and we write $p(x,y,z)$. 
Write $p(x,y,z)=1-f(x,y,z)$. We will give an explicit definition of the mirror curve by specifying the \emph{potential} $f$.

Linear relations between the ray generators of $\Sigma$ correspond to curve classes of $\bP_\Sigma$. Choose a cyclic numbering of the rays $\rho_1,\ldots,\rho_{r+2}$. Let $m_k=(a_k,b_k)$ be the primitive integral generator of $\rho_k$, i.e., $\rho_k=\mathbb{R}_{\geq 0}(a_k,b_k)$ with $a_k,b_k\in\mathbb{Z}$ and $\gcd(a_k,b_k)=1$. A set of independent linear relations
\[ \sum_{i=1}^{r+2} Q_{ik}m_k = 0, \quad k=1,\ldots,r, \]
can be packed into the \emph{charge matrix} $Q=(Q_{ik})_{i,k}$. Using Gauss elimination we can transform the first $r$ columns of $Q$ into the identity matrix, so that $Q_{ik}=\delta_{ik}$ for $i,k\leq r$. Since $\bP_\Sigma$ is smooth, the ray generators of $\rho_{r+1}$ and $\rho_{r+2}$ span the lattice, so all entries of $Q$ are integers. The rows of $Q$ define generators $\beta_1,\ldots,\beta_r$ of the group $H_2(\bP_\Sigma,\mathbb{Z})$ by intersections with toric divisors
\[ D_{\rho_i} \cdot \beta_k = Q_{ik}. \]
If $\Delta^\star$ has only edges of lattice length $1$, then $\bP_{\Sigma'}=\bP_\Sigma$ is a smooth toric Fano surface, hence isomorphic to $\bP^1\times\bP^1$ or the blow up of $\bP^2$ in $0\leq k \leq 3$ points. In this case, we define $f$ as the Hori-Vafa potential
\[ f(x,y,z) = \sum_{k=1}^{r+2} z^{\beta_k}x^{a_k}y^{b_k}. \]
Now consider the case when $\Delta^\star$ has an edge $E$ of length $\ell>1$. Let $(a_k,b_k)$ be a lattice point on $E$.
Let $m_k,n_k\geq 0$ be the lattice distances of $(a_k,b_k)$ from the two vertices of $E$. 
(If $(a_k,b_k)$ is a vertex, then it is contained in two different edges, but the following construction will not depend on this choice.)
For a lattice point on $E$ indexed by $k$, let $\tilde{\beta}_k$ be the curve class defined by the charge vector $(0,\ldots,0,1,-2,1,0,\ldots,0)$, with the entry $-2$ at the $k$-th position, i.e., by 
\[ D_{\rho_k}\cdot\tilde{\beta}_k=-2, \quad D_{\rho_{k\pm 1}}\cdot\tilde{\beta}_k=1, \quad D_{\rho_i}\cdot\tilde{\beta}_k=0 \ \forall i\notin\{k-1,k,k+1\}. \]
Here we consider the indices modulo $r+2$. For an integer $l\geq 0$ and ordered tuples of integers $0\leq i_1<\ldots<i_l\leq m_k$ and $0\leq j_1<\ldots<j_l\leq n_k$ define
\[ \beta_{k,(i_1,\ldots,i_l),(j_1,\ldots,j_l)} = \beta_k+\sum_{p=1}^l\sum_{q=0}^{i_p-1}\tilde{\beta}_{k-q}+\sum_{p=1}^l\sum_{q=1}^{j_p-1}\tilde{\beta}_{k+q}. \]
Note that if $l=0$, then this is equal to $\beta_k$.

\begin{definition}
\label{def:f}
Define the \emph{(toric) potential} of $\bP_\Sigma$, resp. $\bP$, by
\[ f(x,y,z) = \sum_{k=1}^{r+2} \left(\sum_{l=0}^{\min(m_k,n_k)}\sum_{\substack{0< i_1<\ldots<i_l\leq m_k \\ 0< j_1<\ldots<j_l\leq n_k}} z^{\beta_{k,(i_1,\ldots,i_l),(j_1,\ldots,j_l)}}\right) x^{a_k}y^{b_k}. \]
\end{definition}

\begin{remark}
If $(a_k,b_k)$ is a vertex of $\Delta^\star$, then $m_k=n_k=0$ and the coefficient of $x^{a_k}y^{b_k}$ is $z^{\beta_k}$. In particular, it is independent of the choice of edge on which $(a_k,b_k)$ lies. Moreover, this shows that $f$ is a generalization of the Hori-Vafa potential.
\end{remark}

\begin{lemma}
The number of summands in the $x^{a_k}y^{b_k}$-coefficient of $f(x,y,z)$ is $\binom{\ell}{m_k}=\binom{\ell}{n_k}$.
\end{lemma}

\begin{proof}
Note that $m_k+n_k=\ell$, so $\binom{\ell}{m_k}=\binom{\ell}{n_k}$. Without loss of generality assume $n_k\leq m_k$. The number of terms in the $x^{a_k}y^{b_k}$-coefficient of $f(x,y,z)$ is given by
\[ \sum_{l=0}^{n_k}\sum_{\substack{0< i_1<\ldots<i_l\leq m_k \\ 0< j_1<\ldots<j_l\leq n_k}}1 = \sum_{l=0}^{n_k}\binom{m_k}{l}\binom{n_k}{l}. \]
The Vandermonde identity states that for all $m,n,r\in\mathbb{N}$ one has
\[ \binom{m+n}{r} = \sum_{l=0}^{r} \binom{m}{l}\binom{n}{r-l}. \]
For $m=m_k$, $n=n_k$ and $r=n_k$ and with $\ell=m_k+n_k$ the above equation becomes
\[ \binom{\ell}{m_k} = \sum_{l=0}^{n_k} \binom{m_k}{l}\binom{n_k}{n_k-l} = \sum_{l=0}^{n_k} \binom{m_k}{l}\binom{n_k}{l}. \]
This equation agrees with the above expression for the number of terms.
\end{proof}

\begin{remark}
\begin{enumerate}[wide]
\item We will see in Lemma \ref{lem:thetaper} that $f$ agrees with the primitive theta function $\vartheta_1$ in the central chamber of the scattering diagram defined by $\bP_\Sigma$. 
\item The relation obtained from specializing the closed moduli $z^\beta$ by pairing with the polarization was already known by \cite{CPS}, Theorem 8.2, which amounts to saying that $\binom{\ell}{m_k}=\binom{\ell}{n_k}$ counts the summands in the $x^{a_k}y^{b_k}$-coefficient of $\vartheta_1$. Lemma \ref{lem:thetaper} extends the agreement in part (1) of this remark by giving the correct closed moduli. 
\item Moreover, without closed moduli it was shown in \cite{CKPT}, Proposition 3.9, that $f$ is the unique maximally mutable Laurent polynomial with Newton polytope $\Delta^\star$. It would be interesting to study mutations of $f$ respecting the closed moduli.
\item For non-Fano varieties, the potential $f$ is defined to coincide with the theta function $\vartheta_1$ in the central chamber. Compared to Definition \ref{def:f} the potential will receive corrections from internal scattering, see \cite{BGL} and \cite[\S8.2]{CPS}.
\end{enumerate}
\end{remark}

\begin{convention}
\label{convention:fano}
By a lattice isomorphism, we can assume that $\Delta^\star$ contains the points $(1,0)$ and $(0,1)$ and contains no integral points with second coordinate $>1$. These conventions are satisfied by the cases in \cite{Gra1}, Figure 1.3, after changing variables $x\leftrightarrow y$, apart from case (3a) corresponding to a cubic surface, where we have to additionally reflect the picture along the $x$-axis.
\end{convention}

From now on we assume Convention \ref{convention:fano}. If $\rho_j$ has primitive integral generator $(a_j,b_j)$, then the $j$-th row of $Q$ becomes $(0,\ldots,0,1,0,\ldots,0,-a_j,-b_j)$, where the entry $1$ is at the $j$-th position. This row corresponds to a curve class $\beta_j$ defined by
\[ D_{\rho_{r+1}} \cdot \beta_j = -a_j, \quad D_{\rho_{r+2}} \cdot \beta_j = -b_j, \quad D_{\rho_j}\cdot \beta_j = 1, \quad D_{\rho_i}\cdot \beta_j = 0 \ \forall  i \notin \{j,r+1,r+2\}. \]

\begin{example}
Consider $\bP^2$. A fan $\Sigma$ for $\bP^2$ is generated by $(-1,-1)$, $(1,0)$ and $(0,1)$. All ray generators are vertices of $\Delta^\star=\text{Conv}\{(1,0),(0,1),(-1,-1)\}$. The construction above gives
\[ f = x+y+\frac{z}{xy}, \]
where $z$ is the unique closed modulus corresponding to the class of a line. All rays $\rho'$ of the fan correspond to toric divisors $D_{\rho'}$ whose class is the line class. The ray $\rho=\mathbb{R}_{\geq 0}(-1,-1)$ corresponds to the curve class $\beta_\rho$ defined by $\beta_\rho\cdot D_{\rho'}=1$ for all rays $\rho'$, which is equal to the class of a line.
\end{example}

\begin{example}
Consider $\bP^1\times \bP^1$. A fan $\Sigma$ for $\bP^1\times\bP^1$ is generated by $(-1,0)$, $(0,-1)$, $(1,0)$ and $(0,1)$. Again, all ray generators are vertices of $\Delta^\star$. The construction above gives
\[ f = x+y+\frac{z^{\beta_1}}{x}+\frac{z^{\beta_2}}{y}, \]
where $\beta_1$ and $\beta_2$ are the curve classes of the two $\bP^1$ factors. 
The class of the toric divisors $D_{\mathbb{R}_{\geq 0}(1,0)}$ and $D_{\mathbb{R}_{\geq 0}(-1,0)}$ is equal to $\beta_1$ while the class of $D_{\mathbb{R}_{\geq 0}(0,1)}$ and $D_{\mathbb{R}_{\geq 0}(0,-1)}$ is equal to $\beta_2$. 

Now, let's compute the curve classes $\beta_\rho$ (note that $\beta_\rho$ is not the class of the toric divisor $D_\rho$). For $\rho=\mathbb{R}_{\geq 0}(-1,0)$ the corresponding curve class $\beta_\rho$ is defined by $\beta_\rho\cdot\beta_1=1$ and $\beta_\rho\cdot\beta_2=0$, hence is equal to the class $\beta_2$. Analogously, $\beta_\rho$ for $\rho=\mathbb{R}_{\geq 0}(0,-1)$ is equal to the class $\beta_1$.  
\end{example}

\begin{example}
Consider the Hirzebruch surface $\bP_\Sigma=\FF_2$. It admits a $\QQ$-Gorenstein deformation to $\bP^1\times \bP^1$. The fan $\Sigma$ is generated by $(-1,0)$, $(-2,-1)$, $(1,0)$ and $(0,1)$. The construction above gives
\[ f = x+y+\frac{z^F+z^{F+E}}{x}+\frac{z^{2F+E}}{x^2y}, \]
where $F$ is the class of a fiber $E$ is the class of the negative section. This potential $f$ agrees with the one found in \cite{Aur}, Proposition 3.1. See also \cite{BGL}, \S13.2. Together, $F$ and $E$ generate $\textup{NE}(\mathbb{F}_2)$, with intersection relations $F^2=0$, $E^2=-2$ and $F\cdot E=1$. Under the deformation from $\FF_2$ to $\FF_0=\bP^1\times\bP^1$, the classes are mapped as $F\mapsto\beta_2$ and $E\mapsto\beta_1-\beta_2$.

The toric divisors of $(-1,0)$, $(-2,-1)$, $(1,0)$, $(0,1)$ are given by $E,F,2F+E,F$, in this order. For $\rho_2=\mathbb{R}_{\geq 0}(-2,-1)$ the class $\beta_2$ is defined by $\beta_2\cdot(2F+E)=2$, $\beta_2\cdot F=1$ and $\beta_2\cdot E=0$. Hence, $\beta_2=2F+E$. For $\rho_1=\mathbb{R}_{\geq 0}(-1,0)$, the class $\beta_1$ is defined by $\beta_1\cdot(2F+E)=1$, $\beta_1\cdot F=0$ and $\beta_1\cdot E=1$. Hence, $\beta_1=F$. The point $(-1,0)$ is not a vertex of $\Delta^\star$. The class $\tilde{\beta}_1$ is defined by $\tilde{\beta}_1\cdot E=-2$ and $\tilde{\beta}_1\cdot F=1$. Hence, $\tilde{\beta}_1=E$ and $\beta_{1,(0),(0)}=\beta_1+\tilde{\beta}_1=F+E$. Then the $x^{-1}$-coefficient of $f(x,y,z)$ is given by $z^{\beta_1}+z^{\beta_{1,(0),(0)}}=z^F+z^{F+E}$ as claimed.
\end{example}

\subsection{Quantum Mirror Curves and Quantum Periods}
\label{S:qmircurve}

The \emph{quantum mirror curve} arises from topological strings, in which $\log x$ and $\log y$ are conjugate variables
that get
quantized as $\log X$ and $\log Y$ obeying $[\log Y,\log X] = i\hbar$ ($\hbar$ is
also called variously $g_s$ and $\lambda$, the string
coupling constant).  
So $X$ and $Y$ generate a quantum torus $YX = qXY,$
with $q = e^{i\hbar}.$

\begin{definition}
\label{def:mircurve}
By \emph{quantum mirror curve} we mean the ideal in this quantum torus defined by $p(X,Y,z)=1-f(X,Y,z)$, for $f$ as in Definition \ref{def:f}. 
\end{definition}

\begin{remark}
We often suppress the $z$-dependence and write $p(X,Y)=1-f(X,Y)$ as elements in the quantum torus with coefficients in $\CC[\textup{NE}(\bP)]$.
\end{remark}

\begin{example}
The quantum mirror curve for $K_{\bP^2}$ is defined by
$$p\; = 1 - X - Y - q^{\frac{1}{2}}z X^{-1}Y^{-1}.$$
The classical mirror curve is recovered
in the
limit $\hbar \to 0,$ i.e.,~$q\to 1.$
\end{example}

\begin{remark}
A canonical choice of ordering of $p$ can be made as follows.
The non-constant monomials in $p$ are of the
form $z^m,$ where $m\in \bZ^2$ is a ray vector of the
fan $\Sigma$ of $\bP.$
These terms are quantized as $\hat{z}^m,$
obeying the commutation
relations $\hat{z}^m \hat{z}^{m'} = q^{-\frac{1}{2}\det(m|m')} \hat{z}^{m+m'}.$
We take $X = \hat{z}^{(1,0)}$ and $Y = \hat{z}^{(0,1)}.$
Note $\hat{z}^{(-1,-1)} = q^{\frac{1}{2}}\hat{z}^{(-1,0)}\hat{z}^{(0,-1)}=
q^{\frac{1}{2}}X^{-1}Y^{-1}$ explains the last
term in the quantum mirror curve of $\bP^2$.
\end{remark}

Let $\cD$ be the algebra of the quantum torus $YX = qXY$ and let $\cI = \cD {p}$ be the
left ideal defined by
${p}.$ 
Note that $\cV:= \bC[\textup{NE}(\bP)][q^{\pm\frac{1}{2}}][X^{\pm 1}]$ defines a standard representation of $\cD$ by
$(X\cdot f)(X) = Xf(X)$ and $(Y\cdot f)(X) = f(qX).$
This is the induced representation of $\cD$ defined by the
representation of the subalgebra $\cV\subset \cD$ given by $\cV$ itself.
A wavefunction $\Psi$ for the ideal $\cI$ is an element of $\cV$
which is a cyclic vector for the $\cD$-module $\cD/\cD{p},$
in plain terms a solution to the difference equation $p\cdot \Psi=0.$

\begin{remark}
\label{rmk:potential}
To leading order in $\hbar$, $\Psi\sim e^{W/\hbar}$
where $W$ is the potential of the (classical) mirror curve:
$p(x,y) = 0,$
i.e.,~ $\log y = \partial_{\log x} W.$
\end{remark}

\subsection{The quantum A-period}

According to \cite{CKYZ}, \S6, one (classical) period to consider is the integral of $\log(p)$
over the torus $|x| = |y| = 1.$  We call this the A-period. Note that with $p=-f(x,y,z)$ we have $\log(p) = -\sum_{n\geq 1} \tfrac{1}{n} f^n$, and each term $f^n$ is a Laurent polynomial in $x$ and $y$.
Recall from Section \ref{sec:notation} that for such Laurent series we write
$\const(F)$ for the constant term --- then the
integral picks up the constant term $\const(\log(p)) = -\sum_{n\geq 1}\tfrac{1}{n}{\const}(f^n)$.
We use this fact to generalize the definition of period
to the non-commutative case, with non-commuting variables $X$ and $Y$ obeying $YX = qXY.$

\begin{definition}
\label{def:aper}
The quantum A-period is 
\[ a(z,q) := \const_{X,Y}(\log p(X,Y)) = - \sum_{n>0} \frac{1}{n}\const_{X,Y}(f^n). \]
Note that when taking the powers $f^n$, we have to apply the commutation relation $YX=qXY$ repeatedly, leading to different powers of $q^{\tfrac{1}{2}}$.
\end{definition}

\begin{example}
\label{ex:ptf}
In the case of $\bP=\bP^2$, we have
$p = 1 - f$ with $f=X + Y +  \qh z X^{-1}Y^{-1}$.
Then the constant terms of $f^3$ come the $3!$
ways of ordering the product
$X\cdot Y\cdot \qh z X^{-1}Y^{-1}$.
Using $XY = q^{-1}YX$ and $YX^{-1} = q^{-1} X^{-1}y,$ 
the result is seen to be $\qh ( 3 + 3q^{-1}),$
This contributes to the coefficient of $ z$ in the quantum
period, which we call $a_1,$
giving $a_1(q) = -\qh - \qmh$.
The next terms are computed in a similar way, giving
$$
\begin{array}{l}
a(z,q) = \left({-\qh - \qmh}\right) z +
\left(-q^2 - \frac{7}{2} q - 6 - \frac{7}{2} q^{-1} - q^{-2}\right) z^{2} + \\[5pt]
\ \ \left(-q^{\frac{9}{2}} - 3 q^{\frac{7}{2}} - 12 q^{\frac{5}{2}} - \frac{88}{3} q^{\frac{3}{2}} - 48 \qh - 48 \qmh - \frac{88}{3} q^{-\frac{3}{2}} - 12 q^{-\frac{5}{2}} - 3 q^{-\frac{7}{2}} - {q^{-\frac{9}{2}}}\right) z^{3} + \ldots
\end{array}
$$
\end{example}

\begin{remark}
The coefficient of $z^n$ above is a Laurent polynomial
in $q$ which refines
the multinomial coefficient $\binom{3n}{n,n,n}$.
The coefficient of $q^k$ encodes the number of
closed paths of length $3n$ in the 2d triangular lattice
enclosing a fixed signed area, where only three
directions of motion \begin{tikzpicture}[scale = .3]
    \draw[thick,-stealth] (0,0)--(1,0);
    \draw[thick,-stealth] (0,0)--(-1/2,.866);
    \draw[thick,-stealth] (0,0)--(-1/2,-.866);
\end{tikzpicture} are allowed.
For $\bP^1\times \bP^1$ we'd be counting closed paths
of length $2n$ in the rectangular lattice
enclosing a fixed signed area, with
steps allowed in any direction
\begin{tikzpicture}[scale = .3]
    \draw[thick,-stealth] (0,0)--(1,0);
    \draw[thick,-stealth] (0,0)--(-1,0);
    \draw[thick,-stealth] (0,0)--(0,1);
    \draw[thick,-stealth] (0,0)--(0,-1);
\end{tikzpicture}, giving a $q$-analogue of $\binom{2n}{n}^2$.
These $q$-analogues seem to be hard to compute \cite{PS}.
\end{remark}

\subsection{Wavefunctions and the A-period}
\label{sec:wfns-pers}

An important insight of Aganagic-Vafa was that the superpotential $W$ is
the Abel-Jacobi map, i.e.,~the integral of a one-form on the mirror curve:  $W = \int \log Y d\log X$
(or equivalently $\log Y = \partial_{\log X} W,$
as in Remark \ref{rmk:potential} above).

It is shown in \cite[Section 2.2]{ACDKV} that the one-form
$\log Y d \log X$ acquires a pole at the point of brane insertion
with residue $\hbar$ (recall $\hbar = g_s = \lambda$ is the string coupling) -- equivalently $W$ changes by
$\hbar$ times a period.\footnote{
Quantum mirror curves
play a central role in the 
Topological Strings/Spectral Theory duality
of \cite{GHM}, where the spectrum of the
operator defining the mirror curve is studied. 
Further period
computations
can be found in many other works --- see, e.g., \cite{DH,HMMO,HK,HW,DL}, for example.}

By Remark \ref{rmk:potential}, since
$\Psi \sim e^{W/\hbar}$, then the wavefunction changes
by $e^{\textup{period}}$ to leading order.
This reasoning leads \cite{ACDKV} to define the \emph{quantum} A-period
as the constant term  of $\log(\Psi(x+\hbar)/\Psi(x))$ with respect to the variable $X$:

\begin{definition}
\label{defi:reseq}
Define
\[ \ua(z,q) = \const_X \left( \log\frac{\Psi(qX)}{\Psi(X)}\right). \]
\end{definition}

We will show in Theorem \ref{thm:per} that $\ua(z,q)$ agrees with the quantum A-period of Definition \ref{def:aper}.

\begin{definition}
Consider the rings $\CC[q^{\pm\frac{1}{2}}]\llbracket \textup{NE}(\bP)\rrbracket $ and $\CC[q^{\pm\frac{1}{2}}]\llbracket X\rrbracket [X^{-1}]\llbracket \textup{NE}(\bP)\rrbracket $ with grading given by $\deg(z^\beta)=\beta\cdot E$ and denote the pieces of degree at least one by $(\CC[q^{\pm\frac{1}{2}}]\llbracket \textup{NE}(\bP)\rrbracket )_{>0}$ and $(\CC[q^{\pm\frac{1}{2}}]\llbracket X\rrbracket [X^{-1}]\llbracket \textup{NE}(\bP)\rrbracket )_{>0}$, respectively.
\end{definition}

\begin{proposition}
\label{prop:Y}
If $p(X,Y)$ follows the Convention \ref{convention:fano}, the equation $p(X,Y)=0$ has a unique solution $Y=Y(X)$ in $1+X\CC[X]+(\mathbb{C}[q^{\pm\frac{1}{2}}]\llbracket X\rrbracket [X^{-1}]\llbracket \textup{NE}(\bP)\rrbracket )_{>0}$.
\end{proposition}

\begin{proof}
By Convention \ref{convention:fano}, $f(X,Y)$ contains the monomials $X$ and $Y$. Then the equation $p(X,Y)=0$ modulo $z$ becomes $1-X-Y=0$, since all other terms have degree at least one with respect to the grading $\deg(z^\beta)=\beta\cdot E$. We write as an ansatz $Y(X) = 1-X+\sum_{\beta\in\textup{NE}(\bP)} Y_\beta(X)z^\beta$. Then we can solve for the terms $Y_\beta(X)$ case-by-case in $\beta$ with increasing degree $\deg(z^\beta)=\beta\cdot E$, such that each $Y_\beta(X)$ is an element of $\mathbb{C}[q^{\pm\frac{1}{2}}]\llbracket X\rrbracket [X^{-1}]$ of degree at least one.
\end{proof}

See Section \ref{sec:apers} for the computation of $Y(X)$ for $(\bP^2,E)$. We make no claims about the existence of other solutions.

\begin{proposition}
\label{prop:Ypsi}
Let $Y(X)$ be the solution to $p(X,Y)=0$ from
Proposition \ref{prop:Y}, and let $\Psi(X)$ be the wavefunction defined by $p\cdot \Psi=0$. Then
\begin{equation}
\label{eq:Y}
Y(X) = \frac{\Psi(qX)}{\Psi(X)},
\end{equation}
\end{proposition}

\begin{proof}
Let $Y(X)$ be a solution to $p(X,Y(X))=0$ as in Proposition \ref{prop:Y}. Then $p(X,Y(X))\Psi(X)=0$. Hence, the action of $p$ on $\Psi(X)$ is tautologically given by multiplication with $p(X,Y(X))$ and the action of $Y$ on $\Psi(X)$ is given by multiplication with $Y(X)$. Since $Y\cdot \Psi(X)=\Psi(qX)$, this shows $Y(X)=\Psi(qX)/\Psi(X)$.
\end{proof}

\begin{lemma}
\label{lem:welldef}
Let $Y(X)$ be an element of $1+X\CC[X]+(\mathbb{C}[q^{\pm\frac{1}{2}}]\llbracket X\rrbracket [X^{-1}]\llbracket \textup{NE}(\bP)\rrbracket )_{>0}$. Then $\log Y(X)\in X\CC[X]+(\mathbb{C}[q^{\pm\frac{1}{2}}]\llbracket X\rrbracket [X^{-1}]\llbracket \textup{NE}(\bP)\rrbracket )_{>0}$ and $\const_X(\log Y(X)) \in (\mathbb{C}[q^{\pm\frac{1}{2}}]\llbracket \textup{NE}(\bP)\rrbracket )_{>0}$ are well-defined.
\end{lemma}

\begin{proof}
We write $Y(X)=1+Y_1(X)+Y_2(X)$ with $Y_1(X)\in X\mathbb{C}[X]$ and $Y_2(X)\in (\mathbb{C}[q^{\pm\frac{1}{2}}]\llbracket X\rrbracket [X^{-1}]\llbracket \textup{NE}(\bP)\rrbracket )_{>0}$. Then 
\[ \log Y(X) = \sum_{k>0} \frac{(-1)^k}{k} (Y_1(X)+Y_2(X))^k. \]
For all $k,l>0$ we have $Y_1(X)^k\in X\mathbb{C}[X]$, $Y_1(X)^kY_2(X)^l\in (\mathbb{C}[q^{\pm\frac{1}{2}}]\llbracket X\rrbracket [X^{-1}]\llbracket \textup{NE}(\bP)\rrbracket )_{>0}$ and $Y_2(X)^k\in (\mathbb{C}[q^{\pm\frac{1}{2}}]\llbracket X\rrbracket [X^{-1}]\llbracket \textup{NE}(\bP)\rrbracket )_{>0}$. Taking the formal limit $k,l\rightarrow \infty$ changes $X\mathbb{C}[X]$ to $X\mathbb{C}\llbracket X\rrbracket $ and doesn't affect $(\mathbb{C}[q^{\pm\frac{1}{2}}]\llbracket X\rrbracket [X^{-1}]\llbracket \textup{NE}(\bP)\rrbracket )_{>0}$.
\end{proof}

\begin{corollary}
\label{cor:ua}
We have
\[ \ua(z,q) = \const_X(\log Y(X)). \]
\end{corollary}

\begin{proof}
The right hand side is well-defined by Lemma \ref{lem:welldef}. Then the statement follows directly from Definition \ref{defi:reseq} and Proposition \ref{prop:Ypsi}.
\end{proof}

\subsection{Computing the quantum A-Period}
\label{sec:apers}

Here we explain the method of computing the quantum A-period \`a la \cite{ACDKV}, using the example $(\bP^2,E)$.\footnote{We
thank Daniel Krefl for explaining the calculation in Section 7.3 of \cite{ACDKV}.}
We have
\begin{equation}
\label{eq:pereq2}
1 - X - Y(X) - \frac{s}{X Y(q^{-1}X)} = 0,
\end{equation}
where we have put $s = q^{\frac{1}{2}} z$.  We will later specialize
$s$ in our solution to $q^{\frac{1}{2}}z$.

We write as an ansatz $Y(X) = \sum_{k\geq 0}Y_k(X)s^k,$ then solve order-by-order in $s$.  Taking the equation modulo $s$, we get $Y_0 = 1 - X.$ 
Taking \eqref{eq:pereq2} modulo $s^2$ and inserting $Y_0$ gives the equation
$$1 - X - (1 - X) - Y_1 s - \frac{s}{X(1-q^{-1}X + ...)} = 0,$$
so $Y_1 = -\frac{1}{X(1-q^{-1}X)}.$

Similarly, taking \eqref{eq:pereq2} modulo $s^3$ and inserting $Y_0$ and $Y_1$ gives the equation
$$\frac{1}{X(1-q^{-1}X)} s - Y_2 s^2 - \frac{s}{X(1-q^{-1}X - \frac{1}{q^{-1}X(1-q^{-2}X)} s + ...)} = 0,$$
We rewrite the last term as
\begin{eqnarray*}
&& \frac{s}{X(1-q^{-1}X - \frac{1}{q^{-1}X(1-q^{-2}X)} s + ...)} \\
&=& \frac{s}{X(1-q^{-1}X)} \cdot \frac{1}{1 - \frac{1}{(1-q^{-1}X)q^{-1}X(1-q^{-2}X)} s + ...)} \\
&=& \frac{s}{X(1-q^{-1}X)} \cdot \left(1 + \frac{1}{(1-q^{-1}X)q^{-1}X(1-q^{-2}X)} s + ...\right)
\end{eqnarray*}
The first term cancels the first term in the equation above. To cancel the second term we need to have
$$Y_2 = -\frac{1}{X^2q^{-1}(1-q^{-1}X)^2(1-q^{-2}X)}.$$

It is clear that the period equation \eqref{eq:pereq2}
can be written in terms
of the coefficients $Y_n(X).$ 
First, some notation.
For a power series $B = \sum B_n s^n,$ recall we write $[s^n]B$ for the coefficient $B_n$,
and let us define the truncation $\hat{B}_{\leq n} := \sum_{k\leq n} B_k s^k.$

Then we can write the difference equation as the recursion $Y_0 = 1-X$ and
$$Y_{n+1} = [s^n]\frac{-1}{XY_{\leq n}(q^{-1}X)}$$
for $n\geq 0.$

Note that $Y(X) \in 1 + X\bC[X] + s\bC[q^{\pm 1},X^{-1}]\llbracket X\rrbracket \llbracket s\rrbracket $, hence $\log(Y(X)) \in X\bC\llbracket X\rrbracket  + s\bC[q^{\pm 1},X^{-1}]\llbracket X\rrbracket \llbracket s\rrbracket $.  In particular, $\log(Y(X))$
has a well-defined residue and constant term with respect to $X$.
Let's compute:
$$Y = 1 - X + Y_1 s + Y_2 s^2 + ... = (1-X)\left(1 + \frac{Y_1}{1-X} s + \frac{Y_2}{(1-X)}s^2 + ...\right),$$
so
$$\log(Y) = \log(1-X) + \log\left(1 + \frac{Y_1}{1-X} s + \frac{Y_2(1-X)-Y_1^2/2}{(1-X)^2} s^2 + ...\right).$$

Next write
$$\ua(z,q) = \const(\log Y(X)) = \sum {R}_n(q) s^n = \sum {R}_n(q) q^{\frac{n}{2}}z^n =: \sum \ua_n(q) z^n,$$
where we have recalled $s = q^{\frac{1}{2}}z$ and defined the coefficients $\ua_n(q)$ of $\ua(z,q).$ 

Then $\ua_0(q) = 0.$  Next 
$Y_1 = -\frac{1}{X(1-q^{-1}X)}$ gives
$R_1 = -q^{-1} - 1,$ and $\ua_1(q) = -q^{\frac{1}{2}}-q^{-\frac{1}{2}},$ as in Example \ref{ex:ptf}.
From the above form of $Y_2,$ we also find
$$\ua_2(q) = -q^2 - \frac{7}{2}q - 6 - \frac{7}{2}q^{-1} - q^{-2}.$$

Using this method, it is simple to find
the quantum period to any desired order.
To $O(z^3)$ we find
$$
\begin{array}{l}
\ua(z,q) = \left({-\qh - \qmh}\right) z +
\left(-q^2 - \frac{7}{2} q - 6 - \frac{7}{2} q^{-1} - q^{-2}\right) z^{2} + \\[5pt]
\ \left(-q^{\frac{9}{2}} - 3 q^{\frac{7}{2}} - 12 q^{\frac{5}{2}} - \frac{88}{3} q^{\frac{3}{2}} - 48 \qh - 48 \qmh - \frac{88}{3} q^{-\frac{3}{2}} - 12 q^{-\frac{5}{2}} - 3 q^{-\frac{7}{2}} - {q^{-\frac{9}{2}}}\right) z^{3} + \ldots
\end{array}
$$
As remarked in Example \ref{ex:introex}, this $q$-refines the classical period. Moreover, it agrees with $a(z,q)$ in Example \ref{ex:ptf}. By Theorem \ref{thm:per} below, this is not a coincidence. 
\begin{remark}
Tropical methods have also been made to work recently by Singh and Zeitlin for different quantum integrable systems, \cite{SZ}.   
\end{remark}

\subsection{Equality of definitions for the quantum A-period}

Recall from Section \ref{S:mircurve} that $p(X,Y)=1-f(X,Y)$ gets quantized to an operator which acts on a function $\Phi(X)$ by $X\cdot\Phi(X)=X\Phi(X)$ and $Y\cdot\Phi(X)=\Phi(qX)$, and the wavefunction $\Psi(X)$ is defined by $p\cdot\Psi=0$. We have defined
\[ a(z,q) = \const_{X,Y}(\log p(X,Y)) = - \sum_{n>0} \frac{1}{n} \const_{X,Y}(f^n) \]
and by Corollary \ref{cor:ua},
\[ \ua(z,q) = \const_X(\log Y(X)), \quad Y(X) = \frac{\Psi(qX)}{\Psi(X)}, \]
so that the action of $Y$ on $\Psi(X)$ is given by multiplication with $Y(X)$. Note that for simplicity we have hidden the dependencies on $t$, $z$ and $q$ here.

\begin{theorem}
\label{thm:per}
We have $a(z,q) = \ua(z,q)$.
\end{theorem}

\begin{proof}
Substituting $X=xt$ and $Y=yt$, we get $f(X,Y,z) = tf(x,y,z)|_{z^\beta=t^{\beta\cdot E}z^\beta}$. E.g., for $\bP^2$ we have $f = X+Y+zX^{-1}Y^{-1} = t(x+y+t^3zx^{-1}y^{-1})$. Applying $D$ as in Section \ref{sec:notation}, the claim is equivalent to 
$$D\const_{X,Y}(\log p(X,Y)) = 1 - \const_{x,y}\left(\frac{1}{1-tf(x,y)}\right) = D\const_X(\log Y).$$
Note that $g(t,y):=(1-tf)^{-1} \in (R[y]\llbracket y^{-1}\rrbracket)\llbracket t\rrbracket$ for $R=\CC[\textup{NE}(X)][x]\llbracket x^{-1}\rrbracket$. By Convention \ref{convention:fano}, we can assume that $f(X,Y)$ does not contain any terms of $Y$-order $>1$. In other words, $\ord_Y(f)=1$, with Definition \ref{def:ord}. By Lemma \ref{lem:ord1}, $y(x,f)$ is a Laurent series of order $1$ in $f$. Then we can use Lemma \ref{lem:subst} to get
\begin{eqnarray*} 
\const_{x,y}\left(\frac{1}{1-tf(x,y)}\right) 
&=& \const_{x,f}\left(\frac{1}{1-tf}\frac{f\partial_f y(x,f)}{y(x,f)}\right) \\
&=& \const_{x,f}\left(\left(1+tf+t^2f^2+\ldots\right)\frac{f\partial_fy(x,f)}{y(x,f)}\right) \\
&=& \sum_{k=0}^\infty \const_x\left([f^{-k}]\frac{f\partial_f y(x,f)}{y(x,f)}\right) t^k.
\end{eqnarray*}
Here $[f^{-k}]$ is the operator that takes the coefficient of $f^{-k}$ in the Laurent series right of it.
In the last equality we have used that $y(x,f)$ does not depend on $t$. Now recall $p=1-tf(x,y)=0$. Hence we have $f=1/t$ and $\partial_f y(x,f) = -t^2 \partial_t y(x,t)$. Inserting this into the equation above gives
\begin{eqnarray*} 
\const_{x,y}\left(\frac{1}{1-tf(x,y)}\right) 
&=& \sum_{k=0}^\infty \const_x\left([t^{k}]\frac{f\partial_f y(x,f)}{y(x,f)}\right) t^k \\
&=& \const_x\left(\frac{f\partial_f y(x,f)}{y(x,f)}\right) \\
&=& \const_x\left(-\frac{t\partial_t y(x,t)}{y(x,t)}\right) \\
&=& -\const_x\left(t\tfrac{d}{dt}\log y\right),
\end{eqnarray*}
Note that taking the logarithmic derivative $t\tfrac{d}{dt}$ has the same effect as $D$, since $t$ captures the total $z$-degree. Using $\log Y = \log y + \log t$ we get
\[ D\const_X(\log Y) = 1 + \const_x(D\log y) = 1-\const_{x,y}\left(\frac{1}{1-tf(x,y)}\right), \]
as claimed.
\end{proof}

\subsection{Natural \texorpdfstring{$q$}{q}-refinements of mirror maps}
\label{S:mirmap}

We recall that in \cite{GRZ},
the open mirror map was defined
by $U = xe^{a(z)},$ relating
complex and symplectic open moduli $x$ and
$U$ via $a(z)$, the A-period of the mirror curve, which
depended on the closed complex modulus, $z$.
Moreover $z$ and the symplectic modulus $Q$ were related by $Q = ze^{-a(z)},$
and the function $M(Q) = e^{-a(z)}$ was given
enumerative interpretation.  Here we define
some straightforward $q$-analogues.

\begin{lemma} 
\label{lemma-completed-iso}
Let $A$ be a Noetherian algebra over $A_0$ with an ideal $I$ so that the composition $A_0\to A\to A/I$ is an isomorphism.
The formal completion of $A$ in the ideal $I$ is denoted $\widehat{A}$. 
Suppose $\phi\colon \widehat{A} \to \widehat{A}$ is an $A_0$-algebra homomorphism satisfying $\phi(I)\subseteq I$ and inducing an isomorphism of $\phi\colon \widehat{A}/I^2\to \widehat{A}/I^2$ then $\phi$ is an isomorphism.
\end{lemma}
\begin{proof} 
By assumption, $\phi$ respects the filtration by powers of $I$.
We first show by induction over $k$ that $\phi$ induces an isomorphism on $I^k/I^{k+1}$.
The base case $I^0/I=A_0$ is given. The multiplication map $m\colon I/I^2 \otimes I^n/I^{n+1}  \to  I^{n+1}/I^{n+2}$ is easily seen to be surjective. The induction hypothesis gives an isomorphism on the source of $m$ which implies it on the target, inducing the induction step.

By another induction over $k$, we now show that $\phi$ induces an isomorphism on $A/I^k$ which is straightforward when applying the five lemma to the $\phi$-induced endomorphism of the exact sequence $0\to I^{k-1}/I^k\to A/I^k \to A/I^{k-1}\to 0$. 

Since $\widehat{A}=\liminv A/I^k$, the assertion follows.
\end{proof}

\begin{proposition}
\label{prop:autom} Mapping monomials by $\phi\colon Q^\beta \mapsto z^\beta e^{-(\beta\cdot E)a(-z,q)}$ defines an automorphism of $\CC[q^{\pm\frac12}]\llbracket \textup{NE}(\bP)\rrbracket.$
\end{proposition}

\begin{proof} It is straightforward to check that 
$\phi(Q^\beta)\cdot \phi(Q^{\beta'})=\phi(Q^{\beta+\beta'})$, so $\phi$ gives a well-defined endomorphism. 
If $I$ denotes the ideal generated by all monomials $Q^\beta$ in the ring 
$A=\CC[q^{\pm\frac12}][\textup{NE}(\bP)]$ then $A$, $I$, $\widehat{A}=\CC[q^{\pm\frac12}]\llbracket \textup{NE}(\bP)\rrbracket$ and $\phi$ satisfy the assumptions of Lemma~\ref{lemma-completed-iso} which gives the assertion.
\end{proof}

\begin{definition}
\label{def:mirmap}
Put $U = xe^{a(-z,q)},$ a relation between
two formal parameters
$x$ and $U$ via the quantum period of Definition \ref{def:aper}.  
We define $Q^\beta(z,q) = z^\beta e^{-(\beta\cdot E)a(-z,q)}$ and by Proposition \ref{prop:autom} obtain the inverse expressions $z^\beta(Q,q)$.
We define the function
\begin{equation}
    \label{eq:Mdef}
    M(Q,q) := e^{-a(-z(Q,q),q)} = \frac{U}{x} = \left(\frac{z^\beta(Q,q)}{Q^\beta}\right)^{-1/(\beta\cdot E)}
\end{equation}
where the last equality holds for every non-zero class $\beta\in\NE(\PP)$.
We call the series $M(Q,q)$ a $q$-refined open mirror map.  We will give it an enumerative interpretation in
Theorem \ref{thm:agrees}.
\end{definition}

\begin{example}
\label{ex:qM}
In the case of $\bP^2$, we have
$$
\begin{array}{l}
M(Q,q) = 1 - \left(\qh + \qmh\right)Q + \left(q^2+q+1+q^{-1}+q^{-2}\right)Q^2 \\[5pt]
\ \ - \left( q^{\frac{9}{2}} + 3q^{\frac{7}{2}} + 4q^{\frac{5}{2}}+ 4q^{\frac{3}{2}}+ 4q^{\frac{1}{2}}+ 4q^{-\frac{1}{2}}+ 4q^{-\frac{3}{2}}+ 4q^{-\frac{5}{2}}+ 3q^{-\frac{7}{2}}+ q^{-\frac{9}{2}}\right)Q^3 + \cdots,
\end{array}
$$
as computed in Example \ref{ex:gwex}.
\end{example}

\section{Quantum periods, broken lines and descendant invariants}
\label{S:3}

In this section, we give a combinatorial construction of the quantum period (Definition \ref{def:aper}) via broken lines in a consistent scattering diagram. In Section \ref{S:agrees} we will use this description to prove that the primitive theta function, defined using broken lines, equals the open mirror map, also after $q$-refinement. Moreover, we discuss a relationship with descendant Gromov-Witten invariants with $\lambda$-class insertions.

\subsection{Toric degenerations and scattering diagrams}

\label{S:tordeg}

In this subsection and the next we summarize the definition of scattering diagrams and broken lines with closed moduli $z^\beta\in\CC[\textup{NE}(\bP)]$ and non-commutative variables $X,Y$ ($q$-refined). For more details see \cite{Bou}\cite{CPS}\cite{GRZ}.

Let $\mathbb{P}$ be a smooth projective Fano surface and let ${E}$ be a smooth anticanonical divisor, so that $(\mathbb{P},{E})$ is a smooth log Calabi-Yau pair. As in Section \ref{S:fano}, if $\bP$ is not toric, consider a $\mathbb{Q}$-Gorenstein degeneration of $\mathbb{P}$ to a Gorenstein toric Fano variety $\bP_{\Sigma'}$ or a smooth toric semi-Fano variety $\bP_\Sigma$. Let $\Delta^\star$ be the reflexive polygon given by the convex hull of the ray generators of $\Sigma'$ (or $\Sigma$) and let $\Delta$ be the polar dual polytope. Then $\Delta$ is the toric polytope of $\bP_{\Sigma'}$. The lattice points of $\Delta$ define an embedding of the projective toric variety $\bP_{\Sigma'}$ into $\bP^{|\Delta\cap\mathbb{Z}^2|-1}$. The polytope $\Delta$ is reflexive. In particular it contains a unique interior lattice point. 

Consider the central subdivision of $\Delta$ by adding, for each vertex, an edge that connects the vertex with the interior point. We can smooth the boundary of $\Delta$ by adding affine singularities on these new edges. This induces a toric degeneration of $(\mathbb{P},{E})$, i.e., a degeneration into a union of toric pieces such that the family is strictly semi-stable away from a codimension $2$ subset. The irreducible components of the central fiber of the toric degeneration correspond to the maximal cells of the central subdivision of $\Delta$. In other words, the central subdivision of $\Delta$ is an affine manifold with singularities that is the \emph{intersection complex} of the toric degeneration.

By dualizing we get the \emph{dual intersection complex} $B$. This is an affine manifold with singularities that has a natural polyhedral decomposition (dual to the central subdivision of $\Delta$). It has a unique bounded maximal cell (dual to the interior lattice point of $\Delta$) that is isomorphic to $\Delta^\star$. The other maximal cells are unbounded. The affine singularities are such that all unbounded edges of $B$ are parallel, so that there is a unique unbounded direction $m_{\text{out}}$. We write $\textup{Sing}(B)$ for the singular locus of $B$, which is a finite set of points, and $B_0:=B\setminus\textup{Sing}(B)$.

\begin{example}
\label{expl:P2int}
For $(\mathbb{P}^2,E)$, with $E\subset\mathbb{P}^2$ an elliptic curve, a toric degeneration into a union of three $\mathbb{P}(1,1,3)$'s is given by
\begin{eqnarray*}
\mathcal{X} &=& \{XYZ=t^3(W+X^3+Y^3+Z^3)\} \subset \mathbb{P}(1,1,1,3) \times \mathbb{A}_t^1 \rightarrow \mathbb{A}_t^1 \\
\mathcal{D} &=& \{W=0\} \subset \mathcal{X}
\end{eqnarray*}
Figure \ref{fig:P2int} shows the corresponding intersection complex (= subdivision of the moment polytope) and the dual intersection complex (sometimes called tropicalization). 
\end{example}

\begin{figure}[h!]
\centering
\begin{tikzpicture}[scale=1.6]
\coordinate[fill,circle,inner sep=1pt] (0) at (0,0);
\coordinate[fill,circle,inner sep=1pt] (1) at (-1,-1);
\coordinate[fill,circle,inner sep=1pt] (2) at (2,-1);
\coordinate[fill,circle,inner sep=1pt] (3) at (-1,2);
\coordinate (1a) at (-0.5,-0.5);
\coordinate (2a) at (1,-0.5);
\coordinate (3a) at (-0.5,1);
\draw (1) -- (2) -- (3) -- (1);
\draw (0) -- (1a);
\draw (0) -- (2a);
\draw (0) -- (3a);
\draw[gray,dashed] (1a) node[opacity=1,rotate=45]{$\times$} -- (1);
\draw[gray,dashed] (2a) node[opacity=1,rotate=63.43]{$\times$} -- (2);
\draw[gray,dashed] (3a) node[opacity=1,rotate=26.57]{$\times$} -- (3);
\coordinate[fill,circle,inner sep=1pt] (a) at (0,-1);
\coordinate[fill,circle,inner sep=1pt] (b) at (1,-1);
\coordinate[fill,circle,inner sep=1pt] (c) at (-1,0);
\coordinate[fill,circle,inner sep=1pt] (d) at (-1,1);
\coordinate[fill,circle,inner sep=1pt] (e) at (1,0);
\coordinate[fill,circle,inner sep=1pt] (f) at (0,1);
\draw[<->] (2,0.5) -- (2.5,0.5);
\end{tikzpicture}
\begin{tikzpicture}[scale=1.1]
\draw (1,0) -- (0,1) -- (-1,-1) -- (1,0);
\draw (1,0) -- (2.2,0);
\draw (0,1) -- (0,2.2);
\draw (-1,-1) -- (-2.2,-2.2);
\draw[gray,dashed,fill,fill opacity=0.2] (0.5,2.2) -- (0.5,0.5) node[opacity=1,rotate=45]{$\times$} -- (2.2,0.5);
\draw[gray,dashed,fill,fill opacity=0.2] (-2.2,-1.7) -- (-0.5,0) node[opacity=1,rotate=63.43]{$\times$} -- (-0.5,2.2);
\draw[gray,dashed,fill,fill opacity=0.2] (2.2,-0.5) -- (0,-0.5) node[opacity=1,rotate=26.57]{$\times$} -- (-1.7,-2.2);
\end{tikzpicture}
\caption{The intersection complex (left) and dual intersection complex (right) of $(\mathbb{P}^2,E)$.}
\label{fig:P2int}
\end{figure}
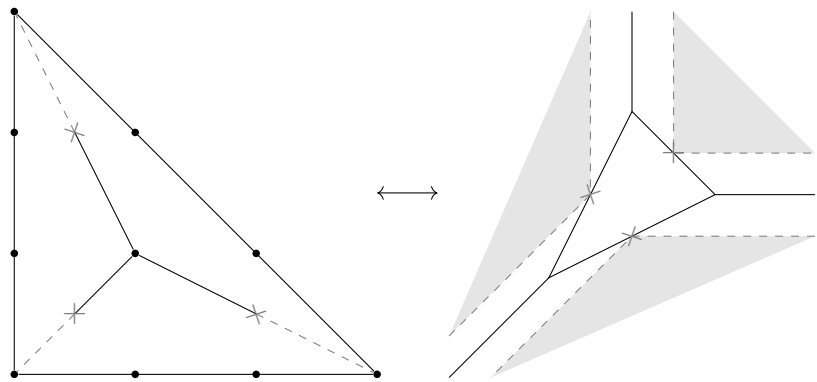

A \emph{multi-valued piecewise linear function} on an open subset $U\subset B_0 = B\setminus\textup{Sing}(B)$ is given by a collection of piecewise linear functions on an open covering $\{U_i\}$ of $U$ which differ only by affine linear functions on overlaps. Choose a cyclic labeling $\rho_1,\ldots,\rho_{r+2}$ of the rays of $\Sigma$, hence of the boundary lattice points $m_1,\ldots,m_{r+2}$ of $\Delta^\star$. For each $k=1,\ldots,r$ define a multi-valued piecewise linear function $\varphi_k$ on $B_0$ as follows. For each maximal cell $\sigma$ of the polyhedral decomposition of $B$ (there are $r+3$ such cells, $1$ bounded and $r+2$ unbounded), let $V_\sigma$ be the interior of $\sigma$. For each $i=1,\ldots,r+2$, let $U_i \subset B_0$ be a neighborhood of $m_i$ such that its closure $\overline{U}_i$ contains both affine singularities adjacent to $m_i$. Then $\{U_i\}\cup\{V_\sigma\}$ gives an open covering of $B_0$. Define $\varphi_k|_{V_\sigma}=0$ for all $\sigma$ and $\varphi_k|_{U_i}=0$ for all $i\neq k$. Define $\varphi_k|_{U_k}$ by the conditions that it is $0$ on $\Delta^\star$ and has slope $1$ along the ray $\rho_k$. Note that the bend locus of $\varphi_k$ is given by $\rho_k$ and the two line segments between $m_k$ and its adjacent affine singularities.

The dual intersection complex $B$ defines a scattering diagram (see \cite{Bou}) with initial rays emanating from the affine singularities. The rays in the scattering diagram carry functions in $\CC[\textup{NE}(\bP)]\llbracket x^m\rrbracket $, where $m$ is the tangent direction of the ray. The scattering algorithm gives a unique consistent scattering diagram.

For $q$-refined scattering (see \cite{Bou}), the commutative variables $x,y$ are replaced by non-commutative variables $X,Y$ with commutation law $XY=q^{-1}YX$.

\begin{example}
\label{expl:P2scat}
For $(\mathbb{P}^2,E)$, the multi-valued piecewise linear function $\varphi$ and the scattering diagram are shown in Figure \ref{fig:P2scat}.
\end{example}

\begin{figure}[h!]
\centering
\begin{tikzpicture}[scale=1.5]
\draw (1,0) -- (0,1) -- (-1,-1) -- (1,0);
\draw (1,0) -- (2.2,0);
\draw (0,1) -- (0,2.2);
\draw (-1,-1) -- (-2.2,-2.2);
\draw[gray,dashed,fill,fill opacity=0.2] (0.5,2.2) -- (0.5,0.5) node[opacity=1,rotate=45]{$\times$} -- (2.2,0.5);
\draw[gray,dashed,fill,fill opacity=0.2] (-2.2,-1.7) -- (-0.5,0) node[opacity=1,rotate=63.43]{$\times$} -- (-0.5,2.2);
\draw[gray,dashed,fill,fill opacity=0.2] (2.2,-0.5) -- (0,-0.5) node[opacity=1,rotate=26.57]{$\times$} -- (-1.7,-2.2);
\draw (-.9,-.9) circle (9pt);
\draw[<-,blue] (-.8,-.8) to[bend right=20] (-.5,-.5);
\draw[<-,blue] (-.8,-1.1) to[bend right=20] (-.3,-1);
\draw[<-,blue] (-1.1,-.9) to[bend left=20] (-1.4,-.3);
\draw[blue] (-.5,-.5) node[above]{\tiny$\ \ \ \ \varphi=0$};
\draw[blue] (-.3,-1) node[right]{\tiny$\varphi=x-2y-1$};
\draw[blue] (-1.4,-.3) node[above]{\tiny$\varphi=-2x+y-1$};
\draw (1,0) circle (9pt);
\draw (0,1) circle (9pt);
\draw[blue] (1,0) node{\tiny$\varphi=0$};
\draw[blue] (0,1) node{\tiny$\varphi=0$};
\draw[blue] (0,0) node{\tiny$\varphi=0$};
\draw[blue] (1.8,0) node{\tiny$\varphi=0$};
\draw[blue] (0,1.8) node{\tiny$\varphi=0$};
\draw[blue] (-1.8,-1.8) node{\tiny$\varphi=0$};
\end{tikzpicture}
\hspace{1cm}
\begin{tikzpicture}[scale=1.1]
\clip (-3,-3) rectangle (3,3);
\draw[dashed] (0.5, 0.5) -- (1.5, 0.5);
\draw[dashed] (0.5, 0.5) -- (0.5, 1.5);
\draw[dashed] (-0.5, 0.0) -- (-0.5, 1.5);
\draw[dashed] (-0.5, 0.0) -- (-1.5, -1.0);
\draw[dashed] (0.0, -0.5) -- (-1.0, -1.5);
\draw[dashed] (0.0, -0.5) -- (1.5, -0.5);
\draw[->] (-6.5, -6.0) -- (-17.5, -16.0);
\draw[->] (-6.0, -6.5) -- (-16.0, -17.5);
\draw[->] (-6.0, -6.5) -- (-16.0, -17.5);
\draw[->] (-6.0, -6.5) -- (-16.0, -17.5);
\draw[->] (-4.0, -4.0) -- (-17.5, -15.812);
\draw[->] (-4.0, -4.0) -- (-15.812, -17.5);
\draw[-] (0.0, -0.5) -- (-2.0, -1.5);
\draw[->] (-0.5, 1.5) -- (2.7, 17.5);
\draw[-] (-0.5, 0.0) -- (-1.5, -2.0);
\draw[->] (1.5, -0.5) -- (17.5, 2.7);
\draw[->] (-1.0, -1.0) -- (-17.5, -17.5);
\draw[->] (-2.0, -1.5) -- (-17.5, -17.0);
\draw[->] (-1.5, -2.0) -- (-17.0, -17.5);
\draw[->] (-4.0, -4.0) -- (-17.5, -17.5);
\draw[->] (-1.5, -2.0) -- (-17.0, -17.5);
\draw[->] (-4.0, -4.0) -- (-17.5, -17.5);
\draw[->] (-1.5, -2.0) -- (-17.0, -17.5);
\draw[-] (0.5, 0.5) -- (-0.5, 1.5);
\draw[->] (-2.0, -1.5) -- (-14.8, -17.5);
\draw[->] (0.0, 4.0) -- (-1.929, 17.5);
\draw[->] (-0.5, 6.0) -- (-1.65, 17.5);
\draw[->] (0.0, 1.0) -- (0.0, 17.5);
\draw[->] (-0.5, 1.5) -- (-0.5, 17.5);
\draw[->] (0.5, 2.0) -- (0.5, 17.5);
\draw[->] (0.0, 4.0) -- (0.0, 17.5);
\draw[->] (0.0, 4.0) -- (0.0, 17.5);
\draw[-] (0.5, 0.5) -- (1.5, -0.5);
\draw[->] (-1.5, -2.0) -- (-17.5, -14.8);
\draw[->] (4.0, 0.0) -- (17.5, 0.0);
\draw[->] (1.0, -0.0) -- (17.5, -0.0);
\draw[->] (1.5, -0.5) -- (17.5, -0.5);
\draw[->] (2.0, 0.5) -- (17.5, 0.5);
\draw[->] (4.0, 0.0) -- (17.5, 0.0);
\draw[-] (-0.5, 0.0) -- (0.5, 2.0);
\draw[->] (2.0, 0.5) -- (17.5, -3.375);
\draw[->] (0.0, 4.0) -- (1.688, 17.5);
\draw[->] (0.5, 6.5) -- (1.5, 17.5);
\draw[-] (0.0, -0.5) -- (2.0, 0.5);
\draw[->] (0.5, 2.0) -- (-3.375, 17.5);
\draw[->] (4.0, 0.0) -- (17.5, -1.929);
\draw[->] (4.0, 0.0) -- (17.5, 1.688);
\draw[->] (6.0, -0.5) -- (17.5, -1.65);
\draw[->] (6.5, 0.5) -- (17.5, 1.5);
\draw[gray,dashed,fill,fill opacity=0.2] (0.5,3) -- (0.5,0.5) node[opacity=1,rotate=45]{$\times$} -- (3,0.5);
\draw[gray,dashed,fill,fill opacity=0.2] (-3,-2.5) -- (-0.5,0) node[opacity=1,rotate=63.43]{$\times$} -- (-0.5,3);
\draw[gray,dashed,fill,fill opacity=0.2] (3,-0.5) -- (0,-0.5) node[opacity=1,rotate=26.57]{$\times$} -- (-2.5,-3);
\end{tikzpicture}
\caption{The multi-valued piecewise linear function (left) and the scattering diagram (right) for $(\bP^2,E)$.}
\label{fig:P2scat}
\end{figure}
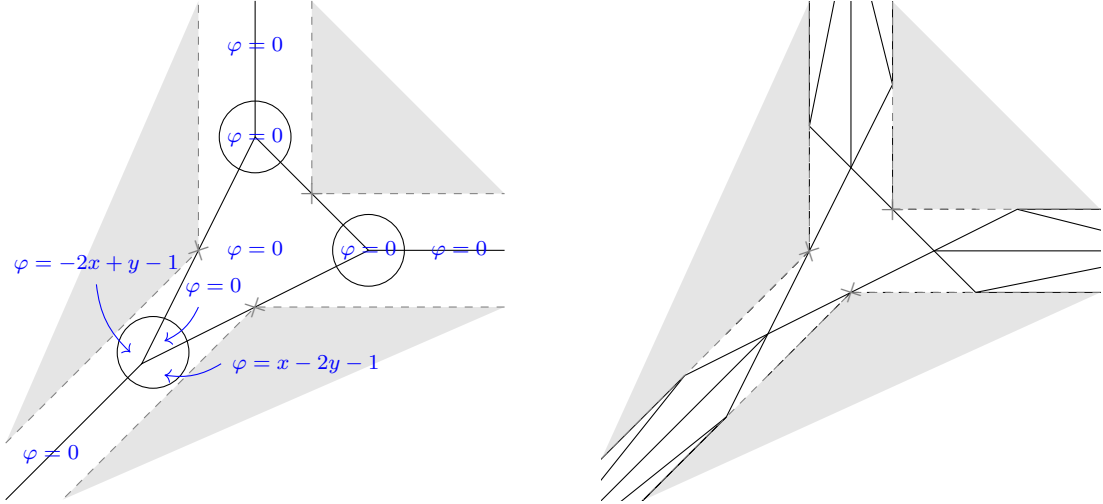

\subsection{Theta functions and quantum periods}

\label{S:theta}

We recall the definitions of broken lines and theta functions from \cite{CPS}\cite{GHS}, to which we refer for further details.

A broken line is a piecewise affine linear embedding $\mathfrak{b}:(-\infty,0]\rightarrow B_0=B\setminus\textup{Sing}(B)$ with $\mathfrak{b}(0)=P$. The domains of affine linearity carry functions that are Laurent monomials in a local coordinate $x^m$ with coefficients in $\CC[\textup{NE}(\bP)]$, where $m$ is the tangent direction of (the graph of) $\mathfrak{b}$. A broken line can ``break'' (i.e., change its domain of affine linearity) only if it intersects a ray of the scattering diagram. In this case, the monomials of the adjacent domains of affine linearity are related by a wall crossing morphism. 

When a broken line $\mathfrak{b}$  crosses the bend locus of $\varphi_k$ such that $\varphi_k$ changes slope along $\mathfrak{b}$ by $\kappa\in\ZZ$, then $\mathfrak{b}$ picks up a factor of $z^{\kappa\beta_k}$. Moreover, the initial rays starting in an affine singularity contained in $\overline{U}_k$ but emanating in the opposite direction carry a factor of $z^{\beta_k}$. If a broken line $\mathfrak{b}$ breaks at such a ray or at any ray that is produced from it later in the scattering procedure, then $\mathfrak{b}$ picks up a factor of $z^{\beta_k}$.

The consistent scattering diagram supports theta functions $(\vartheta_1)_P$ defined as a sum over broken lines with primitive incoming monomial $x^{m_{\text{out}}}$ and ending in a prescribed point $P\in B_0$. Note that the incoming monomial has $z$-degree zero. Hence, it can be seen as a tangent direction lying on the graph of $\varphi_1\times\cdots\times\varphi_k$. The theta functions $(\vartheta_1)_P$ only depend on the chamber of the scattering diagram in which $P$ lies, and are independent of the point within a chamber. We write $(\vartheta_1)_0$ for the theta function inside the bounded maximal cell of $B$, which is a chamber of the scattering diagram by convexity of $\Delta^\star$. We write $(\vartheta_1)_\infty$ for the theta function in an unbounded chamber of the consistent scattering diagram. In this case, the endpoint $P$ will be ``infinitely'' far away (we take a formal limit here) from $\Delta^\star$, hence the notation.

For $q$-refined theta functions, the commutative variables $x,y$ are replaced by non-commutative variables $X,Y$ with commutation law $XY=q^{-1}YX$.

\begin{example}
\label{expl:P2theta}
For $(\bP^2,E)$, Figure \ref{fig:P2theta} shows the broken lines that give
\[ \vartheta_1(X,Y,z)_0 = X+Y+zq^{\frac{1}{2}}X^{-1}Y^{-1}. \]
\end{example}

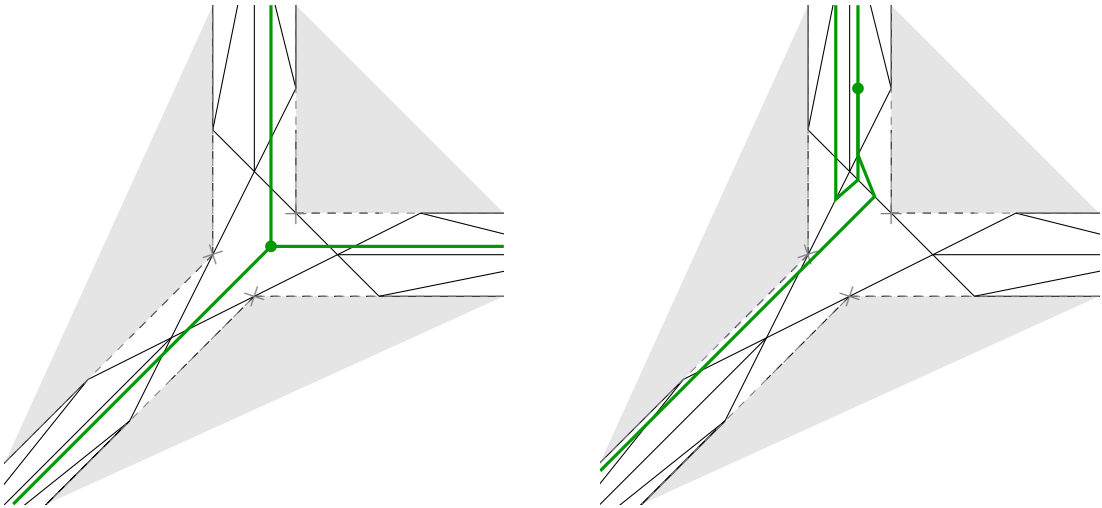
\begin{figure}[h!]
\centering
\begin{tikzpicture}[scale=1.1]
\clip (-3,-3) rectangle (3,3);
\draw[dashed] (0.5, 0.5) -- (1.5, 0.5);
\draw[dashed] (0.5, 0.5) -- (0.5, 1.5);
\draw[dashed] (-0.5, 0.0) -- (-0.5, 1.5);
\draw[dashed] (-0.5, 0.0) -- (-1.5, -1.0);
\draw[dashed] (0.0, -0.5) -- (-1.0, -1.5);
\draw[dashed] (0.0, -0.5) -- (1.5, -0.5);
\draw[->] (-6.5, -6.0) -- (-17.5, -16.0);
\draw[->] (-6.0, -6.5) -- (-16.0, -17.5);
\draw[->] (-6.0, -6.5) -- (-16.0, -17.5);
\draw[->] (-6.0, -6.5) -- (-16.0, -17.5);
\draw[->] (-4.0, -4.0) -- (-17.5, -15.812);
\draw[->] (-4.0, -4.0) -- (-15.812, -17.5);
\draw[-] (0.0, -0.5) -- (-2.0, -1.5);
\draw[->] (-0.5, 1.5) -- (2.7, 17.5);
\draw[-] (-0.5, 0.0) -- (-1.5, -2.0);
\draw[->] (1.5, -0.5) -- (17.5, 2.7);
\draw[->] (-1.0, -1.0) -- (-17.5, -17.5);
\draw[->] (-2.0, -1.5) -- (-17.5, -17.0);
\draw[->] (-1.5, -2.0) -- (-17.0, -17.5);
\draw[->] (-4.0, -4.0) -- (-17.5, -17.5);
\draw[->] (-1.5, -2.0) -- (-17.0, -17.5);
\draw[->] (-4.0, -4.0) -- (-17.5, -17.5);
\draw[->] (-1.5, -2.0) -- (-17.0, -17.5);
\draw[-] (0.5, 0.5) -- (-0.5, 1.5);
\draw[->] (-2.0, -1.5) -- (-14.8, -17.5);
\draw[->] (0.0, 4.0) -- (-1.929, 17.5);
\draw[->] (-0.5, 6.0) -- (-1.65, 17.5);
\draw[->] (0.0, 1.0) -- (0.0, 17.5);
\draw[->] (-0.5, 1.5) -- (-0.5, 17.5);
\draw[->] (0.5, 2.0) -- (0.5, 17.5);
\draw[->] (0.0, 4.0) -- (0.0, 17.5);
\draw[->] (0.0, 4.0) -- (0.0, 17.5);
\draw[-] (0.5, 0.5) -- (1.5, -0.5);
\draw[->] (-1.5, -2.0) -- (-17.5, -14.8);
\draw[->] (4.0, 0.0) -- (17.5, 0.0);
\draw[->] (1.0, -0.0) -- (17.5, -0.0);
\draw[->] (1.5, -0.5) -- (17.5, -0.5);
\draw[->] (2.0, 0.5) -- (17.5, 0.5);
\draw[->] (4.0, 0.0) -- (17.5, 0.0);
\draw[-] (-0.5, 0.0) -- (0.5, 2.0);
\draw[->] (2.0, 0.5) -- (17.5, -3.375);
\draw[->] (0.0, 4.0) -- (1.688, 17.5);
\draw[->] (0.5, 6.5) -- (1.5, 17.5);
\draw[-] (0.0, -0.5) -- (2.0, 0.5);
\draw[->] (0.5, 2.0) -- (-3.375, 17.5);
\draw[->] (4.0, 0.0) -- (17.5, -1.929);
\draw[->] (4.0, 0.0) -- (17.5, 1.688);
\draw[->] (6.0, -0.5) -- (17.5, -1.65);
\draw[->] (6.5, 0.5) -- (17.5, 1.5);
\draw[gray,dashed,fill,fill opacity=0.2] (0.5,3) -- (0.5,0.5) node[opacity=1,rotate=45]{$\times$} -- (3,0.5);
\draw[gray,dashed,fill,fill opacity=0.2] (-3,-2.5) -- (-0.5,0) node[opacity=1,rotate=63.43]{$\times$} -- (-0.5,3);
\draw[gray,dashed,fill,fill opacity=0.2] (3,-0.5) -- (0,-0.5) node[opacity=1,rotate=26.57]{$\times$} -- (-2.5,-3);
\fill[black!40!green] (0.2,0.1) circle (2pt);
\draw[black!40!green,line width=1.2pt] (3,0.1) -- (0.2,0.1);
\draw[black!40!green,line width=1.2pt] (0.2,3) -- (0.2,0.1);
\draw[black!40!green,line width=1.2pt] (-2.9,-3) -- (0.2,0.1);
\end{tikzpicture}
\hspace{1cm}
\begin{tikzpicture}[scale=1.1]
\clip (-3,-3) rectangle (3,3);
\draw[dashed] (0.5, 0.5) -- (1.5, 0.5);
\draw[dashed] (0.5, 0.5) -- (0.5, 1.5);
\draw[dashed] (-0.5, 0.0) -- (-0.5, 1.5);
\draw[dashed] (-0.5, 0.0) -- (-1.5, -1.0);
\draw[dashed] (0.0, -0.5) -- (-1.0, -1.5);
\draw[dashed] (0.0, -0.5) -- (1.5, -0.5);
\draw[->] (-6.5, -6.0) -- (-17.5, -16.0);
\draw[->] (-6.0, -6.5) -- (-16.0, -17.5);
\draw[->] (-6.0, -6.5) -- (-16.0, -17.5);
\draw[->] (-6.0, -6.5) -- (-16.0, -17.5);
\draw[->] (-4.0, -4.0) -- (-17.5, -15.812);
\draw[->] (-4.0, -4.0) -- (-15.812, -17.5);
\draw[-] (0.0, -0.5) -- (-2.0, -1.5);
\draw[->] (-0.5, 1.5) -- (2.7, 17.5);
\draw[-] (-0.5, 0.0) -- (-1.5, -2.0);
\draw[->] (1.5, -0.5) -- (17.5, 2.7);
\draw[->] (-1.0, -1.0) -- (-17.5, -17.5);
\draw[->] (-2.0, -1.5) -- (-17.5, -17.0);
\draw[->] (-1.5, -2.0) -- (-17.0, -17.5);
\draw[->] (-4.0, -4.0) -- (-17.5, -17.5);
\draw[->] (-1.5, -2.0) -- (-17.0, -17.5);
\draw[->] (-4.0, -4.0) -- (-17.5, -17.5);
\draw[->] (-1.5, -2.0) -- (-17.0, -17.5);
\draw[-] (0.5, 0.5) -- (-0.5, 1.5);
\draw[->] (-2.0, -1.5) -- (-14.8, -17.5);
\draw[->] (0.0, 4.0) -- (-1.929, 17.5);
\draw[->] (-0.5, 6.0) -- (-1.65, 17.5);
\draw[->] (0.0, 1.0) -- (0.0, 17.5);
\draw[->] (-0.5, 1.5) -- (-0.5, 17.5);
\draw[->] (0.5, 2.0) -- (0.5, 17.5);
\draw[->] (0.0, 4.0) -- (0.0, 17.5);
\draw[->] (0.0, 4.0) -- (0.0, 17.5);
\draw[-] (0.5, 0.5) -- (1.5, -0.5);
\draw[->] (-1.5, -2.0) -- (-17.5, -14.8);
\draw[->] (4.0, 0.0) -- (17.5, 0.0);
\draw[->] (1.0, -0.0) -- (17.5, -0.0);
\draw[->] (1.5, -0.5) -- (17.5, -0.5);
\draw[->] (2.0, 0.5) -- (17.5, 0.5);
\draw[->] (4.0, 0.0) -- (17.5, 0.0);
\draw[-] (-0.5, 0.0) -- (0.5, 2.0);
\draw[->] (2.0, 0.5) -- (17.5, -3.375);
\draw[->] (0.0, 4.0) -- (1.688, 17.5);
\draw[->] (0.5, 6.5) -- (1.5, 17.5);
\draw[-] (0.0, -0.5) -- (2.0, 0.5);
\draw[->] (0.5, 2.0) -- (-3.375, 17.5);
\draw[->] (4.0, 0.0) -- (17.5, -1.929);
\draw[->] (4.0, 0.0) -- (17.5, 1.688);
\draw[->] (6.0, -0.5) -- (17.5, -1.65);
\draw[->] (6.5, 0.5) -- (17.5, 1.5);
\draw[gray,dashed,fill,fill opacity=0.2] (0.5,3) -- (0.5,0.5) node[opacity=1,rotate=45]{$\times$} -- (3,0.5);
\draw[gray,dashed,fill,fill opacity=0.2] (-3,-2.5) -- (-0.5,0) node[opacity=1,rotate=63.43]{$\times$} -- (-0.5,3);
\draw[gray,dashed,fill,fill opacity=0.2] (3,-0.5) -- (0,-0.5) node[opacity=1,rotate=26.57]{$\times$} -- (-2.5,-3);
\fill[black!40!green] (.1,2) circle (2pt);
\draw[black!40!green,line width=1.2pt] (.1,2) -- (.1,3);
\draw[black!40!green,line width=1.2pt] (.1,2) -- (.1,.9) -- (-.1666,.6666) -- (-.1666,3);
\draw[black!40!green,line width=1.2pt] (.1,2) -- (.1,1.2) -- (.3,.7) -- (-3,-2.6);
\end{tikzpicture}
\caption{Broken lines in the scattering diagram defining $\vartheta_1$ in the central chamber (left) and at infinity up to order $2$ (right).}
\label{fig:P2theta}
\end{figure}

By \cite{Man}, Proposition 2.15, the structure constants of $q$-refined theta multiplication do not depend on the chamber. In particular $\const(\vartheta_1(q)^n)$ does not depend on the chamber. Hence, the following definition makes sense.

\begin{definition}
\label{def:classper}
The \emph{$q$-refined classical period} of $\vartheta_1(q)$ is
\[ \pi(z,q) = \sum_{n>0}\const_{X,Y}(\vartheta_1(q)^n), \]
independent of the chamber in which $\vartheta_1(q)$ is defined.
\end{definition}

\begin{example}
\label{expl:P2classper}
For $(\bP^2,E)$ we have
\[ \pi(z,q) = (3q^{\frac{1}{2}}+3q^{-\frac{1}{2}})z + (6q^2+21q+36+21q^{-1}+6q^{-2})z^2 + \ldots. \]
\end{example}

\begin{lemma}
\label{lem:thetaper}
In the central chamber of the scattering diagram, we have
\[ \vartheta_1(X,Y,z)_0 = f(X,Y,z), \]
where $f(X,Y,z)$ is the Laurent polynomial from Definition \ref{def:f}, with variables $x,y$ replaced by non-commutative variables $X,Y$ with commutation law $XY=q^{-1}YX$.
\end{lemma}

\begin{proof}
Choose the endpoint $P$ close to the origin. Then for each boundary lattice point $p_k$ of $\Delta^\star$ we get broken line with image $P+\mathbb{R}_{\geq 0}p_k$ that does not break at any ray. In other words, this broken line is actually a straight line. If $p_k$ is lies in the interior of an edge of $\Delta^\star$, then we have more broken lines with image $P+\mathbb{R}_{\geq 0}p_k$, obtained as follows. Let $P_k$ be the point where $P+\mathbb{R}_{\geq 0}p_k$ intersects the boundary of $\Delta^\star$. Let $m_k$ and $n_k$ be the lattice distances of $p_k$ from the two vertices of the edge. Choose $0<l\leq\min(m,n)$ and choose $l$ distinct affine singularities on each side of $P_i$. Then we can break at the rays emanating from these singularities. This breaking all happens simultaneously at the same point $P_k\in B$. Since we have chosen the same number $l$ on both sides, the breaking does not change the image of the broken line. The choice of affine singularities corresponds to the choice of $0<i_1<\ldots<i_l<m_k$ and $0<j_1<\ldots<j_l<n_k$ in the definition of $f$ (Definition \ref{def:f}). The $z$-dependence in the definition of $f$ is exactly the one we get from breaking at these rays. Hence $(\vartheta_1)_P$ close to the origin agrees with $f$. By \cite{Man}, Theorem 2.14, $(\vartheta_1)_P$ is constant in a chamber. Hence we have $(\vartheta_1)_0=f$ in the whole central chamber.
\end{proof}

Recall from Section \ref{sec:notation} that we defined an operation $D$ on $\CC[\textup{NE}(\bP)]$ by $Dz^\beta=(\beta\cdot D)z^\beta$.

\begin{proposition}
\label{prop:thetaper}
We have
\[ \pi(z,q) = -Da(z,q), \]
\end{proposition}

\begin{proof}
By Definition \ref{def:aper}, we have
\[ -a(z,q) = \sum_{n>0} \frac{1}{n}\const_{X,Y}(f^n). \]
By the definition of $f$, every monomial in $\const_{X,Y}(f^n)$ takes the form $c(q)z^\beta$ with $\beta\cdot E=n$. By the definition of $D$ we have $Dz^\beta = (\beta\cdot E)z^\beta$. Hence, $D\const_{X,Y}(f^n)=n\const_{X,Y}(f^n)$. Hence, applying $D$ to the equation above gives 
\[ -Da(z,q)=\sum_{n>0}\const_{X,Y}(f^n), \]
which by Definition \ref{def:classper} and Lemma \ref{lem:thetaper} agrees with $\pi(z,q)$.
\end{proof}

It was shown in \cite{GRZ}, Proposition 3.5, that the broken line expansion of the theta function at infinity, denoted  $(\vartheta_1)_\infty$, is a Laurent series in a single variable $y=x^{m_{\text{out}}}$ corresponding to the unique asymptotic direction $m_{\text{out}}$. Consequently, the monomials in $(\vartheta_1)_\infty$ are commutative. The series $(\vartheta_1)_\infty$ has the following enumerative interpretation.

\begin{proposition}
\label{prop:infty}
Let $q=e^{i\hbar}$. Then
\[ \vartheta_1(y,z,q)_\infty = y + \sum_{\beta\in \textup{NE}(\bP)} \sum_{g\geq 0} \frac{1}{\beta\cdot E-1}R_{g,(1,\beta\cdot E-1)}(\mathbb{P}(\log E),\beta) \hbar^{2g} z^\beta y^{-(\beta\cdot E-1)}, \]
where $R_{g,(1,\beta\cdot E-1)}(\mathbb{P}(\log E),\beta)$ is the $2$-marked logarithmic Gromov-Witten invariant of genus $g$ with $\lambda_g$-insertion. Moreover, $\vartheta_1(y,z,q)_\infty$ is a formal Laurent series in $y$ over the ring $\CC[\q^{\pm\frac{1}{2}}][\textup{NE}(\bP)]$.
\end{proposition}

\begin{proof}
By \cite{GRZ}, Theorem 4.14, we have
\[ \vartheta_1(y,z,q)_\infty = y + \sum_{\beta\in \textup{NE}(\bP)} \sum_{g\geq 0} R^{\textup{trop}}_{g,(\beta\cdot E-1,1)}(\mathbb{P}(\log E),\beta) \hbar^{2g} z^\beta y^{-(\beta\cdot E-1)}, \]
where $R^{\textup{trop}}_{g,(\beta\cdot E-1,1)}(\mathbb{P}(\log E),\beta)$ is the tropical analogue of $R_{g,(\beta\cdot E-1,1)}(\mathbb{P}(\log E),\beta)$, defined as a count (with $q$-refined multiplicity) of rational tropical curves on the dual intersection complex $B$ of $(\bP,E)$ with an unbounded leg of weight $\beta\cdot E-1$ through a prescribed point, an unbounded leg of weight $1$ and no other unbounded legs. By \cite{Gra2}, Theorem 6.2, we have 
\[ R^{\textup{trop}}_{g,(\beta\cdot E-1,1)}(\mathbb{P}(\log E),\beta)=(\beta\cdot E-1)R_{g,(\beta\cdot E-1,1)}(\mathbb{P}(\log E),\beta), \]
and by \cite{GRZ}, Theorem 5.4 we have 
\[ (\beta\cdot E-1)R_{g,(\beta\cdot E-1,1)}(\mathbb{P}(\log E),\beta)=\frac{1}{\beta\cdot E-1}R_{g,(1,\beta\cdot E-1)}(\mathbb{P}(\log E),\beta). \]
The three displayed equations together yield the claimed equality. The $y^{-k}$-coefficient of $\vartheta_1(y,z,q)_\infty$ is a sum over finitely many $\beta\in\textup{NE}(\bP)$, those with $\beta\cdot E=k+1$. For each $\beta\in\textup{NE}(\bP)$, the coefficient $R^{\textup{trop}}_{g,(\beta\cdot E-1,1)}(\mathbb{P}(\log E),\beta)$ is a Laurent polynomial in $q^{\pm\frac{1}{2}}$ by the definition of $q$-refined tropical multiplicity (see \cite{GRZ}, Definition 4.6). Hence, $\vartheta_1(y,z,q)_\infty$ is a formal Laurent series over the ring $\CC[\q^{\pm\frac{1}{2}}][\textup{NE}(\bP)]$.
\end{proof}

\begin{example}
\label{expl:P2thetainf}
For $(\bP^2,E)$, Figure \ref{fig:P2theta} shows the broken lines giving the first few terms of
$$
\begin{array}{l}
\vartheta_1(y,z,q)_\infty  = y + \left(\qh + \qmh\right)zy^{-2} + \left(q^2+q+1+q^{-1}+q^{-2}\right)z^2y^{-5} \\[5pt]
\ \ + \left( q^{\frac{9}{2}} + 3q^{\frac{7}{2}} + 4q^{\frac{5}{2}}+ 4q^{\frac{3}{2}}+ 4q^{\frac{1}{2}}+ 4q^{-\frac{1}{2}}+ 4q^{-\frac{3}{2}}+ 4q^{-\frac{5}{2}}+ 3q^{-\frac{7}{2}}+ q^{-\frac{9}{2}}\right)z^3y^{-8} + \cdots.
\end{array}
$$
\end{example}

\subsection{Quantum descendant invariants}

The mirror symmetry proposal of \cite{CCGGK} states that a Laurent polynomial $f$ is mirror dual to $\bP$ if the classical period of $f$ agrees with the regularized GW period\footnote{This term was referred to as a \emph{quantum period} in \cite{CCGGK}, a terminology we need to avoid here.} of $\bP$, which is defined as a generating function of descendant invariants. We give a $q$-refinement of this statement and prove it using tropical geometry.

\begin{definition}
\label{defi:desc}
Let $q=e^{i\hbar}$.
The \emph{$q$-refined regularized GW period} of $\bP$ is
\[ \widehat{G}(z,q) = 1 + \sum_{\beta\in \textup{NE}(\bP)}\sum_{g\geq 0} (\beta\cdot E)! D_{g,1}(\bP,\beta)\hbar^{2g}z^\beta, \]
with descendant Gromov-Witten invariants
\[ D_{g,1}(\bP,\beta) = \int_{\mathcal{M}_{g,1}(\bP,\beta)} (-1)^g\lambda_g\psi^{\beta\cdot{E}-2}\textup{ev}^\star[\textup{pt}], \]
where $\lambda_g=c_g(\mathbb{E})$ are the top Chern classes of the Hodge bundle.
\end{definition}

\begin{remark}
\label{rmk:notation-table}
   This table summarizes our notation.
\begin{table}[h]
\centering
 \begin{tabular}{| c | c |} 
 \hline
 B-side & A-side \\ [0.5ex] 
 \hline\hline
 classical period  & GW period \\ 
 \hline
 quantum period & $q$-refined GW period\\
 \hline
 \end{tabular}
 \end{table}
\end{remark}

\begin{proposition}
\label{prop:descendant}
We have
\[ \pi(z,q)=\widehat{G}(z,q). \]
\end{proposition}

\begin{proof}
By the multiplication rule of theta functions, $\const(\vartheta_1(q)^m)$ is a sum over $m$-tuples of broken lines that are balanced in the endpoint. In the central chamber the broken lines are simply straight lines with ending monomials $z^{m_\rho}$. So $\const(\vartheta_1(q)^m)$ is the count of balanced $m$-tuples of broken lines. But such an $m$-tuple is nothing but a tropical curve with a fixed point lying on a vertex $V$ with valency $m$. The tropical curve has automorphism group of order $m!$ and is counted with a $q$-refined multiplicity defined in \cite{BS}, see \cite{KSU}, \S1.2. By \cite{KSU}, Theorem A, with $n=1$ and $k_1=|\Delta^\circ|-2$, the tropical count equals the descendant invariants $D_{g,1}(\mathbb{P},\beta)$, such that
\[ \const(\vartheta_1(q)^m) = \sum_{\beta: \beta\cdot{E}=m}\sum_{g\geq 0} m! D_{g,1}(\mathbb{P},\beta)\hbar^{2g}z^\beta. \]
\end{proof}

\begin{remark}
Proposition \ref{prop:descendant} is proved in genus zero in \cite{Joh} and for non-smooth divisors (Looijenga pairs) in \cite{KSU2}.
Propositions \ref{prop:thetaper} and \ref{prop:descendant} together yield a relation between $2$-marked invariants $R^{\textup{trop}}_{g,(\beta\cdot E-1,1)}(\mathbb{P}(\log E),\beta)$ and descendant invariants $D_{g,1}(\bP,\beta)$. In genus zero, this relation was also deduced in \cite{Joh}.
\end{remark}

\begin{theorem}
\label{thm:descendant}
We have
\[ -Da(z,q)=\widehat{G}(z,q). \]
\end{theorem}

\begin{proof}
The statement follows from Propositions \ref{prop:thetaper} and \ref{prop:descendant}.
\end{proof}

\begin{remark}
The non-$q$-refined version of Theorem \ref{thm:descendant} was proved in \cite{Giv} using
\[ \widehat{G}(z,q=1) = \sum_{\beta\in \textup{NE}(\bP)} \frac{(-K_{\bP}\cdot\beta)!}{(D_1\cdot\beta)!\cdots(D_r\cdot\beta)!}z^\beta. \]
\end{remark}

\section{The \texorpdfstring{$q$}{q}-refined open mirror map is the quantum theta function}
\label{S:agrees}
We prove the identification of the quantum theta function asymptotic broken line expansion and the $q$-refined open mirror map in this section. As a consequence, we obtain an enumerative interpretation of the $q$-refined
open mirror map. 

\begin{proposition}
\label{prop:agrees}
Let $A=\bigoplus_{k\ge 0}A_k$ be a finitely generated graded algebra over $A_0$ and let $\widehat{A}$ 
denote the formal completion of $A$ in the ideal $\bigoplus_{k\ge 1}A_k$. 
Let $\theta(y)\in A\lfor y^{-1}\rfor[y]$ be a formal Laurent series of the form $\theta(y)=y+\sum_{k\geq 0}a_ky^{-k}$ with $a_k\in A_{k+1}$. 
\begin{enumerate}
\item The series $\eta(w)=w\exp\left({-\sum_{k>0}\frac{1}{k}[y^0]\theta^kw^{-k}}\right)$ gives a well-defined element of $A\lfor w^{-1}\rfor[w]$ of the form $\eta(w)=w+\sum_{k\geq 0} b_kw^{-k}$ that is a compositional inverse of $\theta(y)$, i.e., it satisfies
$\theta(\eta(w))=w.$
\item The series $\widetilde{y}=-\exp\left({-\sum_{k>0}\frac{(-1)^k}{k}[y^0]\theta^k}\right)$ gives a well-defined element of $\widehat{A}$ that satisfies
$\theta(\tilde{y})=-1.$
Moreover, $\tilde{y}$ uniquely determines the coefficients $a_k$, hence it determines a function $\theta(y)$ of the given form.
\end{enumerate}
\end{proposition}

\begin{proof}
Define $f(y)=\theta(y^{-1})$ and let $A_r$ be the degree $r$ graded piece of $A$. Then $f(y)$ is a formal Laurent series in $y$ of the form $f(y)=y^{-1}+\sum_{k\geq 0}f_{k}y^k$ such that $f_{k}\in A_{k+1}$. Recall from Section \ref{sec:notation} that we write $[y^k]f$ for the coefficient of $y^k$ in $f$. The power series $h(y):=yf(y)=1+\sum_{k>0} f_{k-1}y^k\in A\lfor y\rfor$ has the convenient property 
$[y^k]h^l\in A_k$ for every $l\ge 0$. 
In particular, $[y^k](yf)^k\in A_k$, so $[y^0]f^k\in A_k$ and
\begin{equation}
[y^{-1}]f^k\in A_{k-1}.
\label{grading-of-f}
\end{equation}
 We apply Proposition~\ref{lem:3} to $h$ to have
\[ \exp\left(\sum_{k>0}\frac{1}{k}[y^k]h^kw^k\right) = \sum_{k>0}\frac{1}{k}[y^{k-1}]h^kw^{k-1}, \]
with a formal variable $w$. We multiply both sides with $w$, insert $h=yf$ and define
\[ g(w) := w\exp\left(\sum_{k>0}\frac{1}{k}[y^0]f^kw^k\right) = \sum_{k>0}\frac{1}{k}[y^{-1}]f^kw^k. \]
Comparing the very last expression with Lemma~\ref{lem:inversion} shows\footnote{While Lemma~\ref{lem:inversion} is only stated for complex coefficient, we can embed our coefficient ring $A\lfor y\rfor [y^{-1}]$ into $\CC$ before using the lemma.} that $g(w^{-1})$ is the compositional inverse of $f(y)$. 
That is, we have $f(g(w^{-1}))=w$, or equivalently $\theta(g(w^{-1})^{-1})=w$. From the definitions of $g(w)$ and $\eta(w)$ we see that $\eta(w)=g(w^{-1})^{-1}$, proving (1).

We write $g(w)=\sum_{k>0}g_kw^k$ and have $g_k=[w^{k}]g=\frac{1}{k}[y^{-1}]f^k$. 
By \eqref{grading-of-f}, $g_k\in A_{k-1}$. This allows us to specialize $w$ to any fixed value $c$ and the resulting series $g(c)$ gives a well-defined element in the formal completion $\widehat A$.
We set $w=-1$. Then $\theta(g(w^{-1})^{-1})=w$ becomes $\theta(g(-1)^{-1})=-1$. 
It is the substitution of an element of $\widehat A$ for $y$ in an element of $A\lfor y^{-1}\rfor [y]$.
We recognize that $\widetilde{y} = g(-1)^{-1}$, giving $\theta(\widetilde{y})=-1$.

By Lemma~\ref{lem:inversion}, the compositional inverse $g(w)$ uniquely determines $\theta(y)$. We have argued above that $g_k\in A_{k-1}$. Hence, the $w$-exponent in $g$ only captures the degree of $g_k$ in the graded ring $A$. In turn, we can recover $g(w)$ from any fixed value of $w$. In particular, we can recover $g(w)$ from $\tilde{y}=\frac{1}{g(-1)}$ and this determines $\theta(y)$ under the requirement that $\theta(y)$ is of the given form. This proves (2), completing the proof.
\end{proof}

Using the quantum A-period $a(z,q)$ from Definition \ref{def:aper},
in Definition \ref{def:mirmap} we defined a natural $q$-refinement of the open mirror map by
\begin{equation}
    M(Q,q) = e^{-a(-z(Q),q)}, \quad Q^\beta(z,q) = z^\beta e^{-(\beta\cdot{E})a(-z,q)}. 
\label{eq-o-c-mm}
\end{equation}

\begin{theorem}
\label{thm:agrees}
The insertion of $Q^\beta(z,y)=z^\beta\cdot(-1/y)^{\beta\cdot{E}}$ into $M(Q,q)$ gives the identity
\[ \vartheta_1(y,z,q)_\infty = yM(Q,q). \]
\end{theorem}

\begin{proof}
The ring $A:=\CC[q^{\pm\frac{1}{2}}][\textup{NE}(\bP)]$ is graded by $\deg(z^\beta)=\beta\cdot E$. By Proposition \ref{prop:infty}, $\vartheta_1(y,z,q)_\infty$ is a formal Laurent series over $A$ of the form $\vartheta_1(y,z,q)_\infty=y+\sum_{k\geq  0}a_ky^{-k}$ with $a_k\in A_{k+1}$. Then Proposition \ref{prop:agrees}, (2) shows that the insertion of 
\[ y(z,q) := -\exp\left({-\sum_{k>0}\frac{(-1)^k}{k}[y^0]\vartheta_1^k}\right)\in \CC[q^{\pm\frac{1}{2}}]\lfor\textup{NE}(\bP)\rfor  \]
gives $\vartheta_1(y(z,q),z,q)_\infty=-1$.
Comparing $y(z,q)$ with Proposition~\ref{prop:thetaper} shows
\[ y(z,q) = -\exp({a(-z,q)}). \]
Comparing with $M(Q,q)=\exp\big({-a(-z(Q,q),q)}\big)$ from \eqref{eq-o-c-mm}, yields 
\begin{equation}
\label{eq-y}
y(z,q) = \frac{-1}{M(Q(z,q),q)}
\end{equation}
and therefore
\begin{equation}
\label{eq-minus-one}
    \vartheta_1(y(z,q),z,q)_\infty=-1=y(z,q)M(Q(z,q),q).
\end{equation}
Inserting \eqref{eq-y} into $Q^\beta(z,y)$ from \eqref{eq-o-c-mm} yields
\begin{equation*}
    Q^\beta(z,y(z,q)) = z^\beta M(Q(z,q),q)^{\beta\cdot E}=z^\beta \exp\big({-a(-z(Q(z,q),q),q)}\big)^{\beta\cdot E}=Q^\beta(z,q).
\end{equation*}
We see from Proposition \ref{prop:infty} and the definition of $Q^\beta(z,y(z,q))$ in \eqref{eq-o-c-mm} that \eqref{eq-minus-one} takes the form $\theta_1(\widetilde{y})=\theta_2(\widetilde{y})$ for $\widetilde{y}=y(z,q)$. 
The final part of the statement of Proposition~\ref{prop:agrees} now gives $\theta_1(y)=\theta_2(y)$ as elements of $A\lfor y^{-1}\rfor[y]$ which gives the assertion.
\end{proof}

\begin{corollary}
\label{cor:agrees}
Put $q=e^{i\hbar}$.  Then
\[ M(Q,q) = 1+\sum_\beta\sum_{g\geq 0} \frac{(-1)^{\beta\cdot E}}{\beta\cdot E-1}R_{g,(1,\beta\cdot E-1)}(\mathbb{P}(\log E),\beta)\hbar^{2g}Q^\beta. \]
\end{corollary}

\begin{proof}
The statement follows from Proposition \ref{prop:infty} and Theorem \ref{thm:agrees}.
\end{proof}

\subsection{An equation satisfied by the full log potential}

\begin{definition}
\label{defi:fulllog}
We define the \emph{full log potential} as
\[ W = \sum_{n>0} \vartheta_n U^n = \sum_{n>0}\sum_{m>0}\sum_{\beta : \beta\cdot{E}=n+m}\sum_{g\geq 0} mR_{g,(m,n)}(\bP(\log E),\beta)\hbar^{2g}z^\beta y^{-m} U^n. \]
\end{definition}

By \cite{GRZ}, Theorem 4.14, $W$ is a generating series of $2$-marked logarithmic Gromov-Witten invariants. However, it is not a generating series for open Gromov-Witten invariants, even in genus $0$, because the open-log correspondence receives correction terms in both, higher genus and higher winding (resp. order of the fixed point). For genus $0$ open Gromov-Witten invariants, the full potential can be explicitly written in terms of the winding $1$ invariants, \cite{LM}, (2.23). The aim of this section is to find an equation satisfied by the full log potential $W$. Since the coefficients of $\vartheta_n$ depend on $\vartheta_1$ in a difficult way, involving $2$-marked log invariants $R_{g,(p,q)}(\bP(\log E),\beta)$ that are not known to have a simple generating function, this equation will be less elegant than the Lerche-Mayr equation. 
For a formal Laurent series $f(y)$, we use $\res_{y=0}f(y)$ to refer to its \emph{formal residue}, the $-1$st coefficient.

\begin{proposition}
\label{prop:full}
The full log potential $W$ satisfies the equation
\[ 0 = \res_{y'=0}W(y',U) + \left(\frac{1}{U}-\vartheta_1(y) + \res_{y'=0}\left(D_U^{-1}W(y',U)\right)\right)W(y,U). \]
\end{proposition}

\begin{proof}
By \cite{GRZ}, Proposition 5.2, the multiplication rule in the ring of theta functions gives
\begin{equation}
\label{eq:thetarule}
\vartheta_n \cdot \vartheta_1 = \vartheta_{n+1} + \sum_{m=0}^n \alpha_{n,1}^m \vartheta_m + \alpha_{1,n}^0,
\end{equation}
with
\[ \alpha_{n,1}^m = \sum_{\beta : \beta\cdot{E}=n-m+1}\sum_{g\geq 0}(n-m)R_{g,(n-m,1)}(\bP(\log E),\beta)\hbar^{2g}z^\beta, \]
and
\[ \alpha_{1,n}^0 = \sum_{\beta : \beta\cdot{E} = n+1}\sum_{g\geq 0} R_{g,(1,n)}(\bP(\log E),\beta)\hbar^{2g}z^\beta. \]
By inserting Proposition \ref{prop:infty} into equation \eqref{eq:thetarule}, we obtain a system of equations for the 2-marked invariants $R_{g,(m,n)}(\bP(\log E),\beta)$. This system of equations determines all invariants $R_{g,(m,n)}(\bP(\log E),\beta)$ from the invariants $R_{g,(n,1)}(\bP(\log E),\beta)$, hence all theta functions $\vartheta_n(q)$ from the primitive one $\vartheta_1(q)$. To phrase the system of equations as an equation for the generating function $W$, we sum equation \eqref{eq:thetarule} in the sum $\sum_{n>0} (-) U^n$. The left hand side becomes $W\vartheta_1$ and the first term of the right hand side is $\frac{1}{U}W$. The last term on the right hand side is 
\begin{eqnarray*}
\res_{y=0}W &=& \res_{y=0}\sum_{n\geq 0}\sum_{m>0}\sum_{\beta : \beta\cdot{E}=n+m}\sum_{g\geq 0} mR_{g,(m,n)}(\bP(\log E),\beta)\hbar^{2g}z^\beta y^{-m}U^n \\
&=& \sum_{n\geq 0}\sum_{\beta : \beta\cdot{E}=n+1}\sum_{g\geq 0} R_{g,(1,n)}(\bP(\log E),\beta)\hbar^{2g}z^\beta U^n. 
\end{eqnarray*}
Similarly, the middle term on the right hand side is (where $D_U^{-1}$ is log integration, i.e., the inverse to $D_U=U\frac{d}{dU}$)
\begin{align*}
&\res_{y'=0}\left(D_U^{-1}W(y',U)\right)W(y,U) = \res_{y'=0}\left(\sum_{m'\geq 0}\frac{1}{m'}\left((y')^{m'}\right.\right. \\
& \qquad \qquad \left.\left. +\sum_{k>0}\sum_{\beta : \beta\cdot{E}=m'+k}\sum_{g\geq 0}R_{g,(k,m')}(\bP(\log E),\beta)\hbar^{2g}z^\beta (y')^{-k}\right)U^{m'}\right)W(y,U) \\
&= \left(\sum_{m'\geq 0}\frac{1}{m'}\sum_{\beta : \beta\cdot{E}=m'+1}\sum_{g\geq 0}R_{g,(1,m')}(\bP(\log E),\beta)\hbar^{2g}z^\beta U^{m'}\right)\left(\sum_{m\geq 0}\vartheta_m U^m\right) \\
&= \sum_{n\geq 0}\sum_{m\geq 0}\sum_{\beta : \beta\cdot{E}=n-m+1}\sum_{g\geq 0}\frac{1}{n-m}R_{g,(1,n-m)}(\bP(\log E),\beta)\hbar^{2g}z^\beta \vartheta_m U^n \\
&= \sum_{n\geq 0}\sum_{m\geq 0}\sum_{\beta: \beta\cdot{E}=n-m+1}\sum_{g\geq 0} (n-m) R_{g,(n-m,1)}(\bP(\log E),\beta)\hbar^{2g}z^\beta \vartheta_m U^n.
\end{align*}
Putting everything together, we obtain the claimed equation.
\end{proof}

\section{Open-log correspondence for smooth log Calabi-Yau pairs}
\label{S:openlog}

In this section, we prove a higher genus open-log correspondence for smooth log Calabi-Yau pairs $(\bP, E)$. We first show a correspondence between higher genus two-pointed log invariants and local invariants in a similar approach as in \cite{GRZ}. Then assuming $\bP$ is toric, we use the Topological Vertex \cite{AKMV} to relate the log invariants of $(\bP, E)$ to open invariants of $K_{\bP}$ in all-genus.

In order to state our results, we first establish some notation. Let $w \in \mathbb{Z}_{\geq 0}$, and $\vec{k} = (k_1, \ldots, k_n) \in \mathbb{Z}_{\geq 0}^n$ be a vector of non-negative integers such that $\sum_{j=1}^n jk_j = w$. Let $|\vec{k}|$ be the number of non-zero entries of $\vec{k}$. We denote $O_g(K_{\bP}/L, \beta + w\beta_0, \vec{k})$ to be the genus-$g$, winding-$w$, open Gromov-Witten invariant of an outer Aganagic-Vafa brane $L \subset K_{\bP}$ in framing-0, in curve class $\beta+w\beta_0$ with $\beta_0 \in H_2(K_{\bP}, L)$, and in winding profile $\vec{k}$ if it has $k_j$ boundary components of winding $j$. The invariants $O_g(K_{\bP}/L, \beta+w\beta_0, \vec{k})$ can be defined using stable relative maps, and we refer to \cite{LLLZ} \cite{FL} for a more detailed definition. We write $n_g^{open}(K_{\bP}/L, \beta + w\beta_0, \vec{k})$ to be the corresponding open-BPS invariant defined from open multiple cover formulas for $O_g(K_{\bP}/L, \beta + w\beta_0, \vec{k})$ \cite{MV}.
As Gromov-Witten invariants and BPS invariants are uniquely determined by each other via the multiple-cover formulas, the $n^{open}_g(K_\bP/L, \beta+w\beta_0, \vec{k})$ can be defined by the $O_g(K_\bP/L, \beta+w\beta_0, \vec{k})$.

Let $\pi:\hat{\mathbb{P}} \rightarrow \mathbb{P}$ be the blow up of $\mathbb{P}$ at a point with exceptional curve $C$. In Section \ref{sec:GRZ}, we first review the genus-0 open-log correspondence for smooth log Calabi-Yau pairs $(\bP, E)$ established in \cite{GRZ}. In Section \ref{sec:log_GV},  we relate higher genus two-pointed log invariants $R_{g, (1,p)}(\mathbb{P}(\log E), \beta)$ with Gopakumar-Vafa invariants of local $\widehat{\bP}$ in Theorem \ref{thm:tropical_local}. In Section \ref{sec:winding-w}, additionally assuming $\bP$ is toric and for winding $w > 0$, we use the Topological Vertex \cite{AKMV} to relate Gopakumar-Vafa invariants of local $\widehat{\bP}$ in curve class $\pi^*\beta - wC$ to certain LMOV invariants of $L \subset K_{\bP}$ in representation $(w)$ that is a Young Tableau with a single row of $w$ boxes in Theorem \ref{thm:winding}. LMOV invariants are related to the open-BPS invariants $n_g^{open}$ by Equation \ref{eq:change_of_basis}. We then specialize Theorem \ref{thm:winding} to $w = 1$ in Corollary \ref{cor:winding-1} that generalizes \cite{LLW}, Theorem 1.1 to all-genus, and to $w = 2$ in Corollary \ref{cor:winding-2}. We conjecture a higher-winding open-closed relation involving LMOV invariants in Conjecture \ref{conj:GV_LMOV}.

In Section \ref{sec:ol_correspondence}, we prove an open-log correspondence in all-genus in Theorem \ref{thm:ol_correspondence} which relates two-pointed log invariants $R_{g, (1,p)}(\mathbb{P}(\log E), \beta)$ to winding-1, 1-boundary component, open-BPS invariants $n_g^{open}(K_\bP/L, \beta+\beta_0, (1))$ via a discrepancy quantified by stationary Gromov-Witten invariants of $E$, and $R_{h, (1,p)}(\mathbb{P}(\log E), \beta)$ for $h < g$. We provide a computational check of Theorem \ref{thm:ol_correspondence} in Example \ref{ex:explicitdelta} for $\bP = \bP^2$. In Section \ref{S:verification}, we provide a computational check of Corollary \ref{cor:winding-1} using the Topological Vertex with the relevant open and closed BPS invariants listed in Table \ref{table:open_closed_BPS}. 

\subsection{The \texorpdfstring{$g = 0$}{g=0} open-log correspondence}

\label{sec:GRZ}

We have the following genus-0 open-log correspondence due to \cite{GRZ}. 

\begin{theorem}[\cite{GRZ}, Theorem 1.1]
\label{thm:GRZ}
    Let $\mathbb{P}$ be a toric Fano surface with smooth anticanonical divisor $E$. Let $\beta \in \textup{NE}(\bP)$ be an effective curve class and $\beta_0 \in H_2(K_{\bP}, L)$ be a holomorphic disc class with boundary on $L$. Then, 
    
    \[
    O_0(K_{\bP}/L, \beta+\beta_0, (1)) = (-1)^{\beta\cdot E}(\beta\cdot E-1)R_{0, (1, \beta\cdot E-1)}(\bP(\log E), \beta)
    \]
\end{theorem}
When $\bP = \bP^2$, the above invariants are the coefficients of the open mirror map $M(Q)$ of an outer Aganagic-Vafa brane $L \subset K_{\bP^2}$ in framing-0,

\[
M(Q) = 1 - 2Q + 5Q^2 - 32Q^3 + 286Q^4 - 3038Q^5 + \ldots
\]
Theorem \ref{thm:GRZ} implies that the mirror dual Landau-Ginzburg superpotential is equivalent to the open mirror map $M(Q)$. Its proof uses results from \cite{Gra2} \cite{vGGR} \cite{LLW} and \cite{GRZ}, Corollary 6.6 to relate two-pointed log invariants $R_{0, (1, \beta\cdot E - 1)}$ to open invariants of $K_{\bP}$. 

\subsection{Higher genus log/Gopakumar-Vafa correspondence}

\label{sec:log_GV}

Let $N_g(K_{\widehat{\bP}}, \pi^*\beta-C)$ be the genus-$g$, local Gromov-Witten invariant of $K_{\widehat{\bP}}$ in class $\pi^*\beta-C$ \cite{CKYZ}, and $n_g(K_{\widehat{\bP}}, \pi^*\beta-C)$ the genus-$g$, Gopakumar-Vafa invariant of $K_{\widehat{\bP}}$ in curve class $\pi^*\beta-C$ defined by multiple cover formulas \cite{GV1} \cite{GV2}. Let $Q^{\beta}$ be a formal variable tracking the curve class $\beta \in \textup{NE}(\bP)$. Using the higher genus analogues of arguments used to prove Theorem \ref{thm:GRZ}, we have the following theorem.

\begin{theorem}
    \label{thm:tropical_local}
Let $(\mathbb{P}, E)$ be a smooth log Calabi-Yau pair. Then,

\begin{multline*}
    \sum_{\substack{g \geq 0, \\ \beta \in \textup{NE}(\bP)}}\frac{1}{(\beta\cdot E-1)}R_{g, (1, \beta\cdot E-1)}(\bP(\log E), \beta) \hbar^{2g}Q^{\beta} =\\
    \sum_{\substack{g \geq 0, \\ \beta \in \textup{NE}(\bP)}}\left[(-1)^{\beta\cdot E}n_g(K_{\widehat{\bP}}, \pi^*\beta-C)
    \left(2\sin \frac{\hbar}{2}\right)^{2g-2}Q^{\beta}\right] - \Delta  
\end{multline*}
where $\Delta$, defined in Equation \ref{eq:delta_ol}, is a discrepancy expressed by the stationary Gromov-Witten theory of $E$ and two pointed log invariants of $\bP(\log E)$. 
\end{theorem}

\begin{proof}[Proof of Theorem \ref{thm:tropical_local}]

From here on, we simplify the notation by indexing sums $\displaystyle{\sum_{g,\beta}}$ where there is no confusion. We take the generating series of two-pointed log invariants $R_{g, (1, \beta\cdot E - 1)}(X(\log E), \beta)$ with one prescribed contact point with $E$ of order 1 and one non-prescribed contact point with $E$ of order $(\beta\cdot E - 1)$,

\[
\sum_{g,\beta}\frac{1}{(\beta\cdot E-1)}R_{g, (1, \beta\cdot E-1)}(\bP(\log E), \beta)\hbar^{2g}Q^{\beta}
\]
By \cite{GRZ}, Corollary 6.6, the above generating series becomes, 

\begin{equation}
\label{eq:tropical_local3}
    \begin{split}
        \sum_{g,\beta}\frac{1}{(\beta\cdot E-1)}&\left[R_g(\widehat{\bP}(\log \pi^*E - C), \pi^*\beta - C) \right.\\
        &- \left. \sum_{i=0}^{g-1}R_{i,(1, \beta\cdot E - 1)}(\bP(\log E), \beta)N(g-i, 1)\right]\hbar^{2g}Q^{\beta} \\
    \end{split} 
\end{equation}
where $R_g(\widehat{\bP}(\log \pi^* E - C),\pi^*\beta-C)$ are genus-$g$, maximal tangency invariants of the log Calabi-Yau pair $(\widehat{\bP}, \pi^* E - C)$ with $\lambda_g$-insertion (defined in \cite{GRZ}, Section 4), and $N(g,p)$ is the genus-$g$, local relative invariant of $\bP^1(\log \infty)$ with tangency order $p$ computed in \cite{BP}, Theorem 5.1 and is given by the $\hbar^{2g}$-coefficient of $\frac{(-1)^{p+1}}{p}\frac{i\hbar}{\q^{p/2}-\q^{-p/2}}$.

For reasons we shall see, define $\Delta$ to be the term,

\begin{equation}
\label{eq:delta_ol}
\begin{split}
    \Delta := \sum_{g,\beta}\Biggl[&(-1)^{\beta\cdot E}\biggl[\sum_{n \geq 0}\sum_{\substack{g = h+g_1+\ldots+g_n, \\ \mathbf{a} = (a_1,\ldots, a_n) \in \mathbb{Z}^n_{\geq 0}, \\ \pi^*\beta-C = d_E [E] + \beta_1 + \ldots + \beta_n, \\ d_E \geq 0, \beta_j \cdot D > 0}}\frac{(-1)^{g-1 + (E\cdot E)d_E}(E\cdot E)^m}{m!|Aut(\mathbf{a}, g)|}  \\
    &N_{h, (\mathbf{a}, 1^m)}(E, d_E)\prod_{j=1}^n ((-1)^{\beta_j \cdot E}(\beta_j \cdot E) R_{g_j, (\beta_j\cdot E)}(\widehat{\bP}(\log \pi^* E - C), \beta_j)\biggr]\\
    &- \frac{1}{(\beta\cdot E - 1)}\sum_{i=0}^{g - 1} R_{i,(1, \beta\cdot E - 1)}(\bP(\log E), \beta)N(g-i, 1)\Biggr]\hbar^{2g}Q^{\beta}
\end{split}
\end{equation}
where $N_{h, (\textbf{a}, 1^m)}(E, d_E)$ are genus-$h$, stationary invariants of $E$ defined as,
\begin{equation}
\label{eq:stationary_invariants_E}
    N_{h, (\textbf{a}, 1^m)}(E, d_E) := \int_{[\barM_{h,n+m}(E, d_E)]^{vir}}\prod_{i=1}^n ev_i^*[pt] \psi_i^{a_i} \prod_{j=1}^m ev_j^*[pt] \psi_j 
\end{equation}
where $\textbf{a} \in \mathbb{Z}^n_{\geq 0}$, $\psi_j \in H^2(\barM_{h,n+m}(E, d_E))$ is the $\psi-$class at the $j$-th marked point, and $[pt] \in H^2(E)$ is the Poincare-dual of a point.

For $g \geq 0$ and $\beta \in \textup{NE}(\bP)$, we define $\Delta(g,\beta)$ as the $\hbar^{2g}Q^{\beta}$-coefficient of $\triangle$ by,

\[
\Delta := \sum_{g, \beta} \Delta(g, \beta)\hbar^{2g}Q^{\beta}
\]
Applying the $g \geq 0$ log-local principle (combining \cite{BFGW}, Propositions 3.1 and 3.4) to $R_{g, (\beta\cdot E-1)}(\widehat{\bP}(\log \pi^*E - C), \pi^*\beta - C)$ in (\ref{eq:tropical_local3}), it becomes,

\begin{equation}
\label{eq:tropical_local5}
    \begin{split}
        \sum_{g,\beta}\left[(-1)^{\beta\cdot E}N_g(K_{\widehat{\bP}}, \pi^*\beta - C)\hbar^{2g}Q^{\beta}\right]  - \Delta
    \end{split}
\end{equation}
As $\pi^*\beta - C$ is a primitive curve class, the Gopakumar-Vafa formula for $N_g(K_{\widehat{\bP}}, \pi^*\beta-C)$ takes the form,

\begin{equation}
\label{eq:GV}
    \sum_{g \geq 0}N_g(K_{\widehat{\bP}}, \pi^*\beta-C) \hbar^{2g-2} = \sum_{g \geq 0}n_g(K_{\widehat{\bP}}, \pi^*\beta-C) \left(2\sin \frac{\hbar}{2}\right)^{2g-2}
\end{equation}
By combining (\ref{eq:tropical_local3}), (\ref{eq:delta_ol}), (\ref{eq:tropical_local5}), and (\ref{eq:GV}), we have,

\begin{equation}
    \begin{split}
        &\sum_{g,\beta}\frac{1}{(\beta\cdot E-1)}R_{g, (1, \beta\cdot E-1)}(\bP(\log E), \beta)\hbar^{2g}Q^{\beta}=\\ &\hskip.3in\sum_{g,\beta}\left[(-1)^{\beta\cdot E}
        n_g(K_{\widehat{\bP}}, \pi^*\beta-C) \left(2\sin \frac{\hbar}{2}\right)^{2g-2}Q^{\beta} \right] - \Delta \\
    \end{split}
\end{equation}

\end{proof}
Theorem \ref{thm:tropical_local} relates the generating function of genus-$g$, two-pointed log invariants $R_{g, (\beta\cdot E-1, 1)}(\bP(\log E), \beta)$ to the generating function of genus-$g$, Gopakumar-Vafa invariants $n_g(K_{\widehat{\bP}}, \pi^*\beta-C)$ in curve class $\pi^*\beta-C$. We give an explicit expression for the discrepancy $\Delta$ in Example \ref{ex:explicitdelta}. 

\subsection{Open/closed correspondence}
\label{sec:winding-w}

For this section, suppose that $\bP$ is toric and $\pi:\widehat{\bP} \rightarrow \bP$ is a toric blow up at a point. Then $K_{\bP}, K_{\widehat{\bP}}$ are toric Calabi-Yau 3-folds. For $w > 0$, we use the Topological Vertex \cite{AKMV} to derive an equality relating the genus-$g$, degree-$d$ LMOV invariant $N^{\mathrm{LMOV}}_{g, (w)}(K_{\bP}/L, \beta)$ of an outer AV-brane $L \subset K_{\bP}$ with Young Tableau given by a single row $(w)$ of length $w$ boxes to the closed invariants $N_g(K_{\widehat{\bP}}, \pi^*\beta - wC)$. By Equation \ref{eq:change_of_basis} and open multiple cover formulas \cite{MV}, this relates $O_g(K_{\bP}/L, \beta+w\beta_0, \vec{k})$, where $\vec{k}$ satisfies $\sum_j jk_j = w$, to $N_g(K_{\widehat{\bP}}, \pi^*\beta - wC)$. The work of \cite{LLLZ} \cite{LLZ1} \cite{LLZ2} \cite{MOOP} show that calculations with the Topological Vertex \cite{AKMV} compute the open Gromov-Witten invariants $O_g(K_{\bP}/L, \beta+w\beta_0, \vec{k})$ as defined with stable relative maps. We refer to \cite{Kon} \cite{LLLZ} \cite{AKMV} for more definitions and details on the Topological Vertex. 

The toric diagram of a local toric Fano surface $\bP$ is given by a convex $k$-gon with a non-compact edge emanating from each vertex. We refer to Figure \ref{fig:flop} for the relevant toric diagrams. Let $\{v_i\}_{i=1}^k$ be the torus-fixed points, and $\{\ell_i\}_{i=1}^k$ be the compact 1-dimensional torus orbits of the toric diagram of local $\bP$. Let $t_i$ be the K\"ahler parameter of $\ell_i$, and define the variable $Q^{\ell_i} := e^{-t_i}$. For each $\ell_i$, we assign a variable $\lambda_i$ that runs over all partitions or Young Tableaus of arbitrary length. We refer to a partition $\lambda$ and its $U(N$)-representation interchangeably. By the Topological Vertex, the partition function of local $\bP$ is given by,

\[
Z_{\bP}(Q^{\ell_i}) = \sum_{\lambda_1,\ldots, \lambda_k} (-1)^{\sum_i |\lambda_i|}q^{\sum_i \kappa_{\lambda_i}}\prod_{i=1}^k Q^{|\lambda_i|\ell_i}\left(\prod_{j=1}^{k-1} C_{\lambda^T_j \lambda_{j+1}\varnothing}\right)C_{\lambda_{k}^T \lambda_1 \varnothing}
\]
where the vertex functions $C_{\lambda_1 \lambda_2 \lambda_3}$ are defined in \cite{Kon}, Definition 2.2. We will write a partition $\lambda$ as a finite sequence of non-increasing integers $\lambda_1 \geq \lambda_2 \geq \ldots$.

The toric diagram of local $\widehat{\bP}$ has an additional compact 1-dimensional torus orbit in the polygon corresponding to the exceptional curve $C \cong \bP^1$, with K\"ahler parameter $t_{k+1}$ and $Q^C := e^{-t_{k+1}}$. We label the partition of $C$ as $\lambda_{k+1}$. The partition function of local $\widehat{\bP}$ is given by, 

\[
Z_{\widehat{\bP}}(Q^{\ell_i}, Q^C) = \sum_{\lambda_1,\ldots, \lambda_{k+1}}(-1)^{\sum_i |\lambda_i|}q^{\sum_i \kappa_{\lambda_i}}Q^{|\lambda_{k+1}|C}\prod_{i=1}^k Q^{|\lambda_i|\ell_i}\left(\prod_{j=1}^{k} C_{\lambda^T_j \lambda_{j+1}\varnothing}\right)C_{\lambda_{k+1}^T \lambda_1 \varnothing}
\]
We define $Z_{\widehat{\bP}}(Q^{\ell_i}, Q^{-C})$ by replacing $Q^C$ with $Q^{-C}$ in $Z_{\widehat{\bP}}(Q^{\ell_i}, Q^C)$.

As the normal bundle of $C$ in $K_{\widehat{\bP}}$ is $\cO_{\bP^1}(-1) \oplus \cO_{\bP^1}(-1)$, we denote $W$ to be the space obtained by flopping along $C \subset K_{\widehat{\bP}}$. We write $E$ to be the corresponding exceptional curve of $W$, with K\"ahler parameter $t'_{k+1}$ and $Q^E := e^{-t'_{k+1}}$. We label the partition of $E$ to be $\lambda'_{k+1}$. The partition function of $W$ is given by,

\begin{align*}
   \hspace{-2mm} Z_W(Q^{\ell_i}, Q^E) = \hspace{-3mm} \sum_{\lambda_1,\ldots, \lambda_k, \lambda'_{k+1}} \hspace{-5mm} (-1)^{\sum_i |\lambda_i|}q^{\sum_i \kappa_{\lambda_i}}Q^{|\lambda'_{k+1}|E}\prod_{i=1}^k Q^{|\lambda_i|\ell_i}\left(\prod_{j=1}^{k-1} C_{\lambda^T_j \lambda_{j+1}\varnothing}\right)
   C_{\lambda^T_k \lambda_1 \lambda'_{k+1}}C_{\lambda^{'T}_{k+1}\varnothing\varnothing}     
\end{align*}

\begin{figure}[h!]
    \centering

\tikzset{every picture/.style={line width=0.75pt}} 

\begin{tikzpicture}[x=0.90pt,y=0.75pt,yscale=-1,xscale=1]

\draw    (89.28,80.9) -- (112.39,110.45) ;
\draw    (112.39,110.45) -- (141.44,110.83) ;
\draw [shift={(132.92,110.72)}, rotate = 180.76] [color={rgb, 255:red, 0; green, 0; blue, 0 }  ][line width=0.75]    (10.93,-3.29) .. controls (6.95,-1.4) and (3.31,-0.3) .. (0,0) .. controls (3.31,0.3) and (6.95,1.4) .. (10.93,3.29)   ;
\draw    (89.28,138.46) -- (112.39,110.45) ;
\draw [shift={(104.65,119.83)}, rotate = 129.52] [color={rgb, 255:red, 0; green, 0; blue, 0 }  ][line width=0.75]    (10.93,-4.9) .. controls (6.95,-2.3) and (3.31,-0.67) .. (0,0) .. controls (3.31,0.67) and (6.95,2.3) .. (10.93,4.9)   ;
\draw    (161.25,139.23) -- (141.44,110.83) ;
\draw [shift={(155.35,130.77)}, rotate = 235.1] [color={rgb, 255:red, 0; green, 0; blue, 0 }  ][line width=0.75]    (10.93,-4.9) .. controls (6.95,-2.3) and (3.31,-0.67) .. (0,0) .. controls (3.31,0.67) and (6.95,2.3) .. (10.93,4.9)   ;
\draw    (141.44,110.83) -- (163.23,80.9) ;
\draw    (161.25,176.84) -- (161.25,139.23) ;
\draw    (89.28,176.07) -- (89.28,138.46) ;
\draw    (268.13,43.12) -- (280.35,58.74) -- (291.24,72.67) ;
\draw    (268.13,137.9) -- (291.24,109.89) ;
\draw [shift={(283.5,119.27)}, rotate = 129.52] [color={rgb, 255:red, 0; green, 0; blue, 0 }  ][line width=0.75]    (10.93,-4.9) .. controls (6.95,-2.3) and (3.31,-0.67) .. (0,0) .. controls (3.31,0.67) and (6.95,2.3) .. (10.93,4.9)   ;
\draw    (311.05,138.29) -- (291.24,109.89) ;
\draw [shift={(305.15,129.83)}, rotate = 235.1] [color={rgb, 255:red, 0; green, 0; blue, 0 }  ][line width=0.75]    (10.93,-4.9) .. controls (6.95,-2.3) and (3.31,-0.67) .. (0,0) .. controls (3.31,0.67) and (6.95,2.3) .. (10.93,4.9)   ;
\draw    (291.24,72.67) -- (313.03,42.73) ;
\draw    (311.05,175.9) -- (311.05,138.29) ;
\draw    (268.13,175.51) -- (268.13,137.9) ;
\draw    (291.24,109.89) -- (291.24,72.67) ;
\draw [shift={(291.24,85.28)}, rotate = 90] [color={rgb, 255:red, 0; green, 0; blue, 0 }  ][line width=0.75]    (10.93,-4.9) .. controls (6.95,-2.3) and (3.31,-0.67) .. (0,0) .. controls (3.31,0.67) and (6.95,2.3) .. (10.93,4.9)   ;
\draw  [dash pattern={on 0.84pt off 2.51pt}]  (89.28,176.07) -- (111,210.5) ;
\draw  [dash pattern={on 0.84pt off 2.51pt}]  (111,210.5) -- (140.49,210.33) ;
\draw  [dash pattern={on 0.84pt off 2.51pt}]  (142.43,211.29) -- (161.25,176.84) ;
\draw  [dash pattern={on 0.84pt off 2.51pt}]  (268.13,175.51) -- (281.33,210.82) ;
\draw  [dash pattern={on 0.84pt off 2.51pt}]  (281.33,210.82) -- (300.81,210.82) ;
\draw  [dash pattern={on 0.84pt off 2.51pt}]  (300.81,210.82) -- (311.05,175.9) ;
\draw    (338.45,106.72) -- (311.05,138.29) ;
\draw    (268.13,137.9) -- (241.88,106.72) ;
\draw    (89.28,138.46) -- (63.03,107.28) ;
\draw    (188.99,104.98) -- (161.25,139.23) ;
\draw    (419.29,137.57) -- (442.4,109.56) ;
\draw [shift={(434.67,118.94)}, rotate = 129.52] [color={rgb, 255:red, 0; green, 0; blue, 0 }  ][line width=0.75]    (10.93,-4.9) .. controls (6.95,-2.3) and (3.31,-0.67) .. (0,0) .. controls (3.31,0.67) and (6.95,2.3) .. (10.93,4.9)   ;
\draw    (462.21,137.95) -- (442.4,109.56) ;
\draw [shift={(456.31,129.5)}, rotate = 235.1] [color={rgb, 255:red, 0; green, 0; blue, 0 }  ][line width=0.75]    (10.93,-4.9) .. controls (6.95,-2.3) and (3.31,-0.67) .. (0,0) .. controls (3.31,0.67) and (6.95,2.3) .. (10.93,4.9)   ;
\draw    (462.21,175.56) -- (462.21,137.95) ;
\draw    (419.29,175.18) -- (419.29,137.57) ;
\draw    (442.4,109.56) -- (442.4,72.33) ;
\draw  [dash pattern={on 0.84pt off 2.51pt}]  (419.29,175.18) -- (432.5,210.48) ;
\draw  [dash pattern={on 0.84pt off 2.51pt}]  (432.5,210.48) -- (451.98,210.48) ;
\draw  [dash pattern={on 0.84pt off 2.51pt}]  (451.98,210.48) -- (462.21,175.56) ;
\draw    (489.62,106.39) -- (462.21,137.95) ;
\draw    (419.29,137.57) -- (393.05,106.39) ;

\draw (118.6,88.57) node [anchor=north west][inner sep=0.75pt]    {$C$};
\draw (297.2,73.94) node [anchor=north west][inner sep=0.75pt]    {$E$};
\draw (293.24,91.18) node [anchor=north west][inner sep=0.75pt]    {$v_{k}$};
\draw (316.41,128.84) node [anchor=north west][inner sep=0.75pt]    {$v_{1}$};
\draw (84.37,103.4) node [anchor=north west][inner sep=0.75pt]    {$v_{k}$};
\draw (146.76,100.55) node [anchor=north west][inner sep=0.75pt]    {$v_{k+1}$};
\draw (164.64,128.83) node [anchor=north west][inner sep=0.75pt]    {$v_{1}$};
\draw (110.44,113.05) node [anchor=north west][inner sep=0.75pt]    {$\lambda _{k+1}$};
\draw (255.5,57.74) node [anchor=north west][inner sep=0.75pt]    {$v_{k+1}$};
\draw (254.16,79.61) node [anchor=north west][inner sep=0.75pt]    {$\lambda '_{k+1}$};
\draw (303.55,105.87) node [anchor=north west][inner sep=0.75pt]    {$\lambda _{1}$};
\draw (138.87,129.13) node [anchor=north west][inner sep=0.75pt]    {$\lambda _{1}$};
\draw (95.29,129.6) node [anchor=north west][inner sep=0.75pt]    {$\lambda _{k}$};
\draw (233.53,127.74) node [anchor=north west][inner sep=0.75pt]    {$v_{k-1}$};
\draw (257.96,102.23) node [anchor=north west][inner sep=0.75pt]    {$\lambda _{k}$};
\draw (53.36,130.21) node [anchor=north west][inner sep=0.75pt]    {$v_{k-1}$};
\draw (116.53,224.1) node [anchor=north west][inner sep=0.75pt]    {$\widehat{\mathbb{P}}$};
\draw (282.17,230.07) node [anchor=north west][inner sep=0.75pt]    {$W$};
\draw (435.33,229.9) node [anchor=north west][inner sep=0.75pt]    {$\mathbb{P}$};
\draw (444.4,94.34) node [anchor=north west][inner sep=0.75pt]    {$v_{k}$};
\draw (467.58,128.51) node [anchor=north west][inner sep=0.75pt]    {$v_{1}$};
\draw (454.21,111.54) node [anchor=north west][inner sep=0.75pt]    {$\lambda _{1}$};
\draw (384.69,127.41) node [anchor=north west][inner sep=0.75pt]    {$v_{k-1}$};
\draw (410.13,108.4) node [anchor=north west][inner sep=0.75pt]    {$\lambda _{k}$};
\end{tikzpicture}

    \caption{Toric diagrams of $\widehat{\bP}, W$, and $\bP$.}
    \label{fig:flop}
\end{figure}
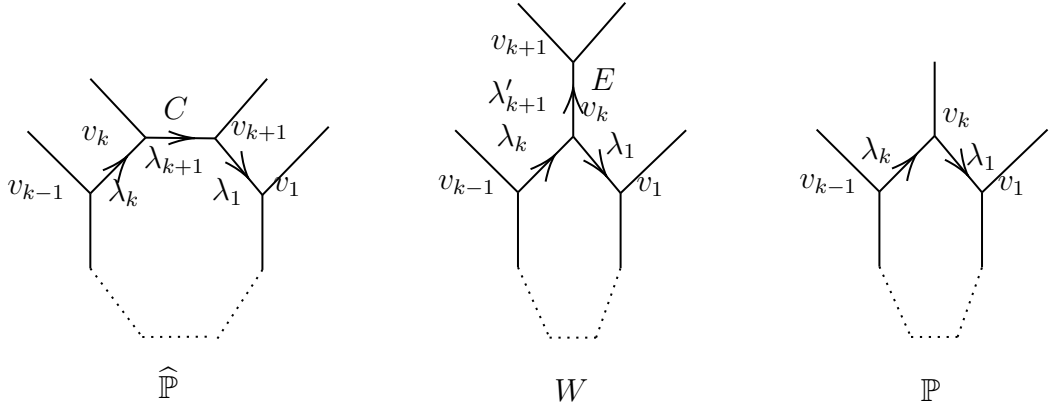
Let $V$ be an $U(N)$-matrix with $N$ suitably large, representing the monodromy of a flat connection on $L$. Let $Tr_{\nu} V$ be the trace of $V$ in a $U(N)$-representation given by a partition $\nu$. The open free energy of an outer AV-brane $L \subset K_{\bP}$ is given by,

\[
Z_{\bP}(Q^{\ell_i}, V) = \sum_{\lambda_1,\ldots, \lambda_k, \nu}(-1)^{\sum_i |\lambda_i|}q^{\sum_i \kappa_{\lambda_i}}\prod_{i=1}^k Q^{|\lambda_i|\ell_i}\left(\prod_{j=1}^{k-1} C_{\lambda^T_j \lambda_{j+1}\varnothing}\right) C_{\lambda^T_k\lambda_1 \nu} Tr_{\nu}V
\]
For a fixed partition $\nu$ attached to an outer AV-brane, define,

\[
Z_{\bP}(Q^{\ell_i}, V)|_{\nu} := \sum_{\lambda_1,\ldots, \lambda_k}(-1)^{\sum_i |\lambda_i|}q^{\sum_i \kappa_{\lambda_i}}\prod_{i=1}^k Q^{|\lambda_i|\ell_i}\left(\prod_{j=1}^{k-1} C_{\lambda^T_j \lambda_{j+1}\varnothing}\right) C_{\lambda^T_k\lambda_1 \nu} Tr_{\nu}V
\]
We write $Z_{\bP}(Q^{\ell_i}, V)  := \sum_\nu Z_{\bP}|_\nu Tr_\nu V$. Define the normalized open partition function $Z^{open}_{\bP}(V) := \frac{Z_{\bP}(Q^{\ell_i}, V)}{Z_{\bP}(Q^{\ell_i})}$, which is given by,

\[
Z^{open}_{\bP}(V) = \mathrm{Exp}\left(\sum_\nu f_{\nu\varnothing \varnothing}(q, Q) Tr_\nu V\right)
\]
where $\mathrm{Exp}$ denotes the plethystic exponential.  \footnote{We mention a minor typo in \cite{AKMV}, Equation 7.23. The subscript of $\widehat{f}$ should be $R'_i$ not $R_i$.} The $f_{\nu\varnothing \varnothing}(q, Q)$ are generating series related to LMOV invariants in representation $\nu$ of an outer AV-brane, and hence the $Z_{\nu}$ are expressible by LMOV invariants. Writing $Z^{open}_{\bP}(V) = \sum_\nu Z_\nu Tr_\nu V$, we have $Z_{\bP}|_\nu = Z_\nu Z_{\bP}$.

Denote $N^{\mathrm{LMOV}}_{g, \lambda}(K_{\bP}/L, \beta)$ to be the genus-$g$, LMOV invariant of an outer AV-brane $L \subset K_{\bP}$ in degree $\beta$ and representation $\nu$. We write its generating series as,

\begin{equation}
\label{eq:def_fhat}
  \widehat{f}_{\nu \varnothing \varnothing}(q, Q) = \sum_{\substack{g \geq 0, \\ \beta \in NE(\bP)}} N^{\mathrm{LMOV}}_{g, \nu}(K_{\bP}/L, \beta) (q^{\frac{1}{2}} - q^{\frac{-1}{2}})^{2g}Q^{\beta}  
\end{equation}
The $\widehat{f}_{\nu \varnothing \varnothing}$ are related to $f_{\nu \varnothing \varnothing}$ by \cite{AKMV}, Equation 7.23. The LMOV invariants are related to the open-BPS invariants by the linear transformation,

\begin{equation}
\label{eq:change_of_basis}
  n_g^{open}(K_{\bP}/L, \beta+\beta_0, \vec{k}) = \sum_{\nu} \chi_\nu(C(\vec{k}))N^{\mathrm{LMOV}}_{g, \nu}(K_{\bP}/L, \beta)  
\end{equation}
where $\chi_{\nu}(C(\vec{k}))$ is the character $\chi_{\nu}$ of the $U(N)$-representation $\nu$ evaluated on the conjugacy class $C(\vec{k})$ of the symmetric group $S_{w}$ determined by $\vec{k}$. The matrix $(\chi_\nu(C(\vec{k})))_{\nu, \vec{k}}$ is invertible over $\mathbb{Q}$ but not $\mathbb{Z}$. This suggests the LMOV invariants $N^{\mathrm{LMOV}}_{g, \nu}$ are in some sense more fundamental than the open-BPS invariants $n_g^{open}$. For more details on $Z^{open}_{\bP}(V)$, LMOV invariants and their relationship with open-BPS invariants, we refer to \cite{AKMV}, Section 7.3.

Recall that $C \subset K_{\widehat{\bP}}$ is the exceptional curve, $W$ is the flop of $K_{\widehat{\bP}}$ along $C$, and $E$ is the resulting exceptional curve of $W$.
\begin{theorem}
\label{thm:winding}
Let $\pi: \widehat{\bP} \rightarrow \bP$ be the toric blow up at a point. Under the identification $Q^{-C} = Q^E$, we have the equality of generating series relating closed invariants $N_g(K_{\widehat{\bP}}, \pi^*\beta - wC)$ to $\mathrm{LMOV}$ invariants $Z_{\nu}$ of an outer AV-brane $L \subset K_{\bP}$,

\[
    \sum_{\substack{g \geq 0, \\ \beta \in NE(\bP), \\ w \in \mathbb{Z}_{>0}}}N_g(K_{\widehat{\bP}}, \pi^*\beta - wC)\hbar^{2g-2}Q^{\pi^*\beta - wC} = \sum_{k \geq 1}\frac{(-1)^{k+1}}{k} \left(\sum_{\nu \neq \emptyset} (-1)^{|\nu|}Z_\nu C_{\nu^T \varnothing\varnothing}Q^{|\nu|E}\right)^k 
\]
Hence, for $w \in \mathbb{Z}_{> 0}$ the coefficient of $Q^{-wC} = Q^{wE}$ is given by,

\[
\sum_{j=1}^w \frac{(-1)^{j+1+w}}{j} \sum_{\substack{\ell_a \in \mathbb{Z}_{\geq 1}, \\ (\ell_1, \ldots, \ell_j), \\ \sum_{a=1}^j \ell_a = w}}\prod_{\substack{|\nu_a| = \ell_a, \\ 1 \leq a \leq j}} Z_{\nu_a}C_{\nu_a^T \varnothing \varnothing}
\]

\end{theorem}

\begin{proof}
By flop invariance of the Topological Vertex \cite{KM}, we have $Q^{-C} = Q^E$, and the equality of closed partition functions,

\[
Z_{\widehat{\bP}}(Q^{\ell_i}, Q^{-C}) = Z_W(Q^{\ell_i}, Q^E)
\]
The gluing or degeneration formula in the Topological Vertex (\cite{LLLZ}, Proposition 5.8) gives,

\begin{align*}
Z_W(Q^{\ell_i}, Q^E) &= \sum_\nu (-1)^{|\nu|} Z_{\bP}|_\nu C_{\nu^T \varnothing \varnothing}Q^{|\nu|E} \\
&= \sum_\nu (-1)^{|\nu|} Z_{\bP} Z_\nu C_{\nu^T \varnothing \varnothing}Q^{|\nu|E} 
\end{align*}
which implies,

\begin{equation}
\label{eq:partition_fns1}
\frac{Z_{\widehat{\bP}}(Q^{\ell_i}, Q^{-C})}{Z_\bP(Q^{\ell_i})} = \sum_\nu (-1)^{|\nu|} Z_\nu C_{\nu^T \varnothing \varnothing}Q^{|\nu|E} = 1 + \sum_{\nu\neq 0} (-1)^{|\nu|}Z_\nu C_{\nu^T \varnothing \varnothing}Q^{|\nu|E}
\end{equation}
Taking the log of Equation \ref{eq:partition_fns1} to obtain connected invariants, we obtain the first expression. Since we divide by $Z_\bP$, we have $w > 0$ in the left generating series in Theorem \ref{thm:winding}. The second expression of the claim follows from a combinatorial analysis.
\end{proof}

We provide some examples of Theorem \ref{thm:winding}. For explicit values of LMOV invariants and Gopakumar-Vafa invariants for $\bP = \bP^2$, we refer to tables in Appendix \ref{sec:LMOV_GV}.

\begin{example}[Theorem \ref{thm:winding} in winding-1]
\label{ex:w1}

Setting $w = 1$ in Theorem \ref{thm:winding}, we have,

\begin{align*}
  Q^{-C}\sum_{\substack{g \geq 0, \\ \beta \in NE(\bP)}}N_g(K_{\widehat{\bP}}, \pi^*\beta - C)Q^{\pi^*\beta}\hbar^{2g-2} &= -Z_{(1)}C_{(1)\varnothing \varnothing}Q^E \\
  &= -(q^{\frac{1}{2}} - q^{\frac{-1}{2}})^{-2}\widehat{f}_{(1)}Q^E
\end{align*}
The Gopakumar-Vafa formula \cite{GV1} \cite{GV2} for the curve class $\pi^*\beta - C$ is,

\[
\sum_{\substack{g \geq 0, \\ \beta \in NE(\bP)}}N_g(K_{\widehat{\bP}}, \pi^*\beta - C)Q^{\pi^*\beta - C}\hbar^{2g-2} = \sum_{\substack{g \geq 0, \\ \beta \in NE(\bP)}}(-1)^{g-1}n_g(K_{\widehat{\bP}}, \pi^*\beta - C) Q^{\pi^*\beta - C}(q^{\frac{1}{2}} - q^{\frac{-1}{2}})^{2g-2}
\]
where $q = e^{i \hbar}$. Hence multiplying by $-(q^{\frac{1}{2}} - q^{\frac{-1}{2}})^{2}$ and cancelling $Q^{-C}$ with $Q^E$, we have

\begin{equation}
\label{eq:winding-1_eq}
   \sum_{\substack{g \geq 0, \\ \beta \in NE(\bP)}}(-1)^{g}n_g(K_{\widehat{\bP}}, \pi^*\beta - C) Q^{\pi^*\beta}(q^{\frac{1}{2}} - q^{\frac{-1}{2}})^{2g} = \widehat{f}_{(1)} 
\end{equation}
under the variable change $z = (q^{\frac{1}{2}} - q^{\frac{-1}{2}})^{2}$. As the winding-1, open-BPS invariants are equal to the LMOV invariants in representation $\square$ by Equation \ref{eq:change_of_basis}, we have,

\begin{corollary}
  \label{cor:winding-1}
    For $w = 1$, Example \ref{ex:w1} of Theorem \ref{thm:winding} implies,
    \begin{align*}
      n_g(K_{\widehat{\bP}}, \pi^*\beta-C) &= (-1)^{g}N^{\mathrm{LMOV}}_{g, (1)}(K_{\bP}/L, \beta) \\
      &= (-1)^g n_g^{open}(K_\bP/L, \beta+\beta_0, (1)) 
    \end{align*}
\end{corollary}

\end{example}

\begin{example}[Theorem \ref{thm:winding} in winding-2]
\label{ex:w2}
Setting $w = 2$ in Theorem \ref{thm:winding}, we have 

\[
Q^{-2C}\sum_{\substack{g \geq 0, \\ \beta \in NE(\bP)}}N_g(K_{\widehat{\bP}}, \pi^*\beta - 2C)Q^{\pi^*\beta} \hbar^{2g-2}= \left( Z_{(2)} C_{(1,1)} + Z_{(1,1)} C_{(2)} - \frac{1}{2}(Z_{(1)} C_{(1)})^2\right)Q^{2E} 
\]
The relevant quantities are described in \cite{AKMV} and are,
\begin{align*}
  C_{(2)} &= \frac{q^2}{(q-1)(q^2-1)} \\
  C_{(1,1)} &= \frac{q}{(q-1)(q^2-1)} \\
  C_{(1)} &= (q^{\frac{1}{2}} - q^{\frac{-1}{2}})^{-1} \\
  Z_{(2)} &= f_{(2)}(q, Q) + \frac{1}{2}f_{(1)}(q^2, Q^2) + \frac{1}{2}f_{(1)}(q, Q)^2  \\
  Z_{(1,1)} &= f_{(1,1)}(q, Q) - \frac{1}{2}f_{(1)}(q^2, Q^2) + \frac{1}{2}f_{(1)}(q, Q)^2 \\
  f_{(2)} &= (q^{\frac{1}{2}} - q^{\frac{-1}{2}})^{-1}(q^{\frac{-1}{2}}\widehat{f}_{(2)} - q^{\frac{1}{2}}\widehat{f}_{(1,1)}) \\
  f_{(1,1)} &= (q^{\frac{1}{2}} - q^{\frac{-1}{2}})^{-1}(-q^{\frac{1}{2}}\widehat{f}_{(2)} + q^{\frac{-1}{2}}\widehat{f}_{(1,1)}) \\
  Z_{(1)} &= f_{(1)} = (q^{\frac{1}{2}} - q^{\frac{-1}{2}})^{-1}\widehat{f}_{(1)}
\end{align*}
where $\widehat{f}$ are defined in Equation \ref{eq:def_fhat}. We have $C_{(1,1)} + C_{(2)} = (q^{1/2} - q^{-1/2})^{-2}, C_{(1,1)} - C_{(2)} = \frac{-1}{2(q-q^{-1})}$. The Gopakumar-Vafa formula for curve class $\pi^*\beta - 2C$ is,

\begin{equation}
\label{eq:winding-2_GV}
 \begin{split}
    Q^{\pi^*\beta - 2C}\sum_{\substack{g \geq 0, \\ \beta \in NE(\bP)}}N_g(K_{\widehat{\bP}}, \pi^*\beta - 2C) \hbar^{2g-2} &= \sum_{\substack{g \geq 0, \\ \beta \in NE(\bP)}}\left[(-1)^{g-1}n_g(K_{\widehat{\bP}}, \pi^*\beta - 2C)\left(q^{\frac{1}{2}} - q^{\frac{-1}{2}}\right)^{2g-2} \right. \\
&+ \left. (-1)^{g-1}n_g(K_{\widehat{\bP}}, \frac{\pi^*\beta}{2} - C)\frac{1}{2}(q-q^{-1})^{2g-2}\right]Q^{\pi^*\beta - 2C}
\end{split}   
\end{equation}
Plugging in expressions, we have,

\begin{equation}
\label{eq:winding2_vertex}
    \begin{split}
        Z_{(2)} C_{(1,1)} + Z_{(1,1)} C_{(2)} - \frac{1}{2}(Z_{(1)} C_{(1)})^2 &= f_{(2)} \frac{q}{(q-1)(q^2-1)} + f_{(1,1)}\frac{q^2}{(q-1)(q^2-1)} \\
 &- \frac{1}{2(q-q^{-1})}f_{(1)}(q^2, Q^2) \\
 &= (q^{\frac{1}{2}} - q^{\frac{-1}{2}})^{-1}\left(q^{\frac{-1}{2}}\widehat{f}_{(2)} - q^{\frac{1}{2}}\widehat{f}_{(1,1)}\right)\frac{q}{(q-1)(q^2-1)} \\
 &+ (q^{\frac{1}{2}} - q^{\frac{-1}{2}})^{-1}\left(-q^{\frac{1}{2}}\widehat{f}_{(2)} + q^{\frac{-1}{2}}\widehat{f}_{(1,1)}\right)\frac{q^2}{(q-1)(q^2-1)} \\
 &- \frac{1}{2(q-q^{-1})^2}\widehat{f}_{(1)}(q^2, Q^2) \\
 &= \frac{-1}{(q^{\frac{1}{2}} - q^{\frac{-1}{2}})^2}\widehat{f}_{(2)}(q, Q) - \frac{1}{2(q-q^{-1})^2}\widehat{f}_{(1)}(q^2, Q^2)  
    \end{split}
\end{equation}
Setting Equation \ref{eq:winding-2_GV} equal to Equation \ref{eq:winding2_vertex} implies, on the level of individual invariants,

\begin{corollary}
\label{cor:winding-2}
For $w = 2$, Example \ref{ex:w2} of Theorem \ref{thm:winding} implies,

\begin{align*}
n_g(K_{\widehat{\bP}}, \pi^*\beta - 2C) &= (-1)^g N^{\mathrm{LMOV}}_{g, (2)}(K_\bP/L, \beta) 
\end{align*} 
\end{corollary}
\end{example}
For $\bP = \bP^2, \widehat{\bP} = \FF_1$, Corollaries \ref{cor:winding-1}, \ref{cor:winding-2} agree with values of Gopakumar-Vafa invariants of $K_{\FF_1}$ in \cite{HKR}, Appendix B and LMOV invariants of $L \subset K_{\bP^2}$ in \cite{AKMV}, \S 9.4. We conjecture that,

\begin{conjecture} 
\label{conj:GV_LMOV}
    For $w \in \mathbb{Z}_{>0}$, $\beta \in NE(\bP)$ and all $g \geq 0$, 
    
    \[
    n_g(K_{\widehat{\bP}}, \pi^*\beta - wC) = (-1)^g N^{\mathrm{LMOV}}_{g, (w)}(K_\bP/L, \beta)
    \]
    with proof in $w = 1, 2$.
    \begin{proof}
        Refer to Examples \ref{ex:w1}, \ref{ex:w2}.
    \end{proof}
\end{conjecture}

\begin{example}
  In Appendix \ref{sec:LMOV_GV}, we provide tables of LMOV and Gopakumar-Vafa invariants that we computed to verify Conjecture \ref{conj:GV_LMOV} for $\bP = \bP^2$ and $w \leq 4, d \leq 5$ and $g \leq 3$.  For example, the genus-$3$ Gopakumar-Vafa invariant of $K_{\FF_1}$ in curve class $\pi^*5H - 2C$ is $-992$ as listed in Table \ref{subtable:winding2_LMOV}, and the genus-3, degree-5 LMOV invariant of $L \subset K_{\bP}$ in representation $(2)$ is $992$.
\end{example}

\begin{remark}
\label{rem:1}
    In genus-0, we have $O_0(K_{\bP}/L, \beta+\beta_0, (1)) = n_0^{open}(K_{\bP}/L, \beta+\beta_0, (1))$ (see footnote in Example \ref{ex:explicitdelta}) and $N_0(K_{\widehat{\bP}}, \pi^*\beta-C) = n_0(K_{\widehat{\bP}}, \pi^*\beta-C)$ since $\pi^*\beta-C$ is a primitive curve class. By remarks made in \cite{GRZ}, Section 2.2, $O_0(K_{\bP}/L, \beta+\beta_0, (1))$ equals the genus-0, open Gromov-Witten invariant of a moment fiber of $K_{\bP}$.  Thus, Corollary $\ref{cor:winding-1}$ is equivalent to \cite{LLW}, Theorem 1.1 in genus-0.
\end{remark}

\begin{remark}
    For toric manifolds (both Fano and non-Fano), the genus-0, open Gromov-Witten invariant of a moment fiber was defined using symplectic geometry by \cite{FOOO}. Both the open invariant of an outer AV-brane $L \subset K_{\bP}$ and of a moment fiber in $K_{\bP}$ agree in genus-0, by remarks in \cite{GRZ}, Section 2.2. Hence by Corollary \ref{cor:winding-1} and \cite{LLW}, Theorem 1.1, genus-0 open Gromov-Witten invariants defined with either stable relative maps in algebraic geometry \cite{LLLZ} \cite{FL} or with symplectic geometry \cite{FOOO} agree, as they both equal $N_0(K_{\widehat{\bP}}, \pi^*\beta- C)$.
\end{remark}

\begin{remark}
    Corollary \ref{cor:winding-1} is also used to establish a correspondence between the winding-1, open-BPS invariants $n_g^{open}(K_{\bP}/L, \beta+\beta_0, (1))$ of an outer AV-brane $L \subset K_{\bP}$ and certain higher genus closed invariants of the projectivized canonical bundle $\bP(K_{\bP} \oplus \cO_{\bP})$ \cite{Zho}. 
\end{remark}

\subsection{Higher genus open-log correspondence}

\label{sec:ol_correspondence}

From results in Sections \ref{sec:log_GV} and \ref{sec:winding-w}, we have a higher genus open-log correspondence for smooth log Calabi-Yau pairs $(\bP, E)$ that relates the generating function of two-pointed log invariants $R_{g, (1, \beta\cdot E-1)}(\bP(\log E), \beta)$ to the genus-$g$, winding-1, 1-boundary component, open-BPS invariants $n^{open}_g(K_\bP/L, \beta+\beta_0, (1))$ of an outer AV-brane $L \subset K_{\bP}$ in framing-0,

\begin{theorem}[Open-log correspondence for smooth log Calabi-Yau pairs]
\label{thm:ol_correspondence}
Let $\bP$ be a toric Fano surface with smooth anticanonical divisor $E$. Let $L$ be an outer Aganagic-Vafa brane in $K_{\bP}$. Then, we have, 

\begin{multline*}
\sum_{\substack{g \geq 0, \\ \beta \in \textup{NE}(\bP)}}\frac{1}{(\beta\cdot E-1)}R_{g, (1, \beta\cdot E-1)}(\bP(\log E), \beta) \hbar^{2g}Q^{\beta} =\\  \sum_{\substack{g \geq 0, \\ \beta \in \textup{NE}(\bP)}}\left[(-1)^{\beta\cdot E + g}
 n_g^{open}(K_\bP/L, \beta+\beta_0, (1))
  \left(2\sin \frac{\hbar}{2}\right)^{2g-2}Q^{\beta}\right]  - \Delta  
\end{multline*}
where the discrepancy term $\Delta$ is as in Theorem \ref{thm:tropical_local}.
\end{theorem}

\begin{proof}
    Apply Corollary \ref{cor:winding-1} to Theorem \ref{thm:tropical_local}.
\end{proof}
Thus, by Theorem \ref{thm:ol_correspondence}, we can express log invariants $R_{g, (1, \beta\cdot E - 1)}(\bP(\log E), \beta)$ by open-BPS invariants $n_g^{open}(K_{\bP}/L, \beta+\beta_0, (1))$, or equivalently the open Gromov-Witten invariants $O_g(K_{\bP}/L, \beta + w\beta_0, \vec{k})$, and vice versa.

\begin{example}
\label{ex:explicitdelta}
    Let $\bP = \bP^2$. We write $O_1(K_{\bP^2}/L, \beta+\beta_0, (1)) = -n_1^{open}(K_{\bP^2}/L, \beta+\beta_0, (1)) + \frac{1}{24}n_0^{open}(K_{\bP^2}/L, \beta+\beta_0, (1))$ by the open multiple cover formulas \cite{MV} \footnote{\label{footnote:sign}In order to match with open Gromov-Witten invariants computed in \cite{AKMV} \cite{LLW} \cite{GZ}, we introduce a global sign to the open multiple cover formula in \cite{MV}, Equation 2.10.}. Specializing Theorem \ref{thm:ol_correspondence} to genus-1 and degree-4, we have,
    
    \begin{multline*}
        O_1(K_{\bP^2}/L, 4H + \beta_0, (1)) = \frac{1}{11}R_{1, (1,11)}(\bP^2(\log E), 4H) +  \frac{11}{24}R_{0, (1,11)}(\bP^2(\log E), 4H)  \\
        -3R_{0, (3)}(\mathbb{F}_1(\log \pi^*E - C), \pi^*H) - 2R_{0, (1)}(\FF_1(\log \pi^*E - C), C)R_{0, (2)}(\FF_1(\log \pi^*E - C), F)
    \end{multline*}
    From Equation \ref{eq:delta_ol}, we can compute that,
    
    \[
    \Delta(1, 4H) = O_1(K_{\bP^2}/L, 4H + \beta_0, (1)) - \frac{1}{11}R_{1, (1,11)}(\bP^2(\log E), 4H) + \frac{1}{24}n_0(K_{\FF_1}, \pi^*4H - C)
    \]
    where $n_0(K_{\FF_1}, \pi^*4H- C) = 11R_{0, (11,1)}(\bP^2(\log E), 4H)$ by the log-local principle \cite{vGGR} and \cite{GRZ}, Corollary 6.6.
    
    The open Gromov-Witten invariant $O_1(K_{\bP^2}/L, 4H+\beta_0, (1))$ was computed by localization to be $\frac{-3313}{12}$ \cite{GZ}. The relevant right hand side invariants $R_{g, (11,1)}(\bP^2(\log E), 4H)$ and $R_{0}(\FF_1(\log \pi^*E - C), \beta)$ can be computed by $q$-refined tropical curve counting in the scattering diagram of $(\bP^2, E)$ (\cite{GRZ}, Appendix B) and/or using the formula of \cite{vGGR}. The first equality above is explicitly, 

    \[
    \frac{-3313}{12} = \frac{1}{11}\cdot (-18513) + \frac{11}{24}\cdot 3146 - 3\cdot 9 - 2\cdot 1 \cdot 4
    \]
\end{example}

\begin{remark}
    Open-log correspondences in all-genus have been proven when the anti-canonical divisor is singular. In \cite{BBvG}, an open-log correspondence is proven for Looijenga pairs satisfying a deformation property, and is extended to quasi-tame Looijenga pairs in \cite{BS2}. In \cite{Sch}, an all genus open-log correspondence is established for two-component Looijenga pairs.
\end{remark}

\begin{remark}
Recall from Theorem \ref{thm:agrees} that $\vartheta_1(q)$ or equivalently $M(Q, q)$ encode the higher-genus, two-pointed log invariants of $(\bP, E)$. We note that, unlike in genus-0, $M(Q, q)$ is not the generating series of higher-genus, 1 boundary component, winding-1 open invariants of an outer Aganagic-Vafa brane $L \subset K_{\bP}$. For $\bP = \bP^2$, we have from Example \ref{ex:qM} that,
\begin{multline*}
M(Q,q) = 1 - (q^{\frac{1}{2}} + q^{\frac{-1}{2}})Q + (q^{-2} + q^{-1} + 1 + q + q^2)Q^2 \\
 - \left( q^{\frac{9}{2}} + 3q^{\frac{7}{2}} + 4q^{\frac{5}{2}}+ 4q^{\frac{3}{2}}+ 4q^{\frac{1}{2}}+ 4q^{-\frac{1}{2}}+ 4q^{-\frac{3}{2}}+ 4q^{-\frac{5}{2}}+ 3q^{-\frac{7}{2}}+ q^{-\frac{9}{2}}\right)Q^3 + \ldots
\end{multline*}
whereas the higher-genus open invariants of $L \subset K_{\bP^2}$ can be computed by the Topological Vertex (\cite{AKMV}, Equation 9.10),

\[
\widehat{f}_{\square \, \varnothing \, \varnothing}(z, Q) = 1 - 2Q + 5Q^2 - (32 + 9z)Q^3 + (286 + 288z + 108z^2 + 14z^3)Q^4 + \ldots
\]
where $z = (q^{\frac{1}{2}} - q^{\frac{-1}{2}})^2$. The coefficient of $z^g Q^d$ in $\widehat{f}_{\square \,\varnothing\, \varnothing}$ is the genus-$g$, degree-$d$, LMOV invariant of $L \subset K_{\bP^2}$ in representation $\square$. By Equation \ref{eq:change_of_basis}, the LMOV invariant in representation $\square$ is equal to the winding-1, open-BPS invariant $n_g^{open}(K_{\bP^2}/L, \beta+\beta_0, (1))$. In the limit as $z \rightarrow 0$ or $q\rightarrow 1$, we recover the genus-0, open mirror map $M(Q, q)$ from $\widehat{f}_{\square \,\varnothing\, \varnothing}$. Hence, we see that $\widehat{f}_{\square \,\varnothing\, \varnothing} \neq M(Q, q)$, with the difference $M(Q, q) - \widehat{f}_{\square \,\varnothing\, \varnothing}$ described in Theorem \ref{thm:ol_correspondence}.
\end{remark}

\subsection{Verification of Corollary \ref{cor:winding-1}}

\label{S:verification}
We provide a computational check of Corollary \ref{cor:winding-1} and Theorem \ref{thm:ol_correspondence} for $\bP = \bP^2$ for curve classes $\beta = dH$ for $d \leq 4$ and all-genus. To perform calculations with the topological
vertex, we use the diagrams in Figure \ref{fig:toric_diagrams}.
The winding-1, open topological string partition function $Z_{\bP^2}(V)$ of an outer $AV$-brane in local $\bP^2$ in framing-0 is given by,

\[
Z_{\bP^2}(V) = \sum_{\lambda_1, \lambda_2, \lambda_3}  (-1)^{\sum_i |\lambda_i|} Q^{\left(\sum_i |\lambda_i|\right)H} q^{\sum_i \kappa_{\lambda_i}} C_{\square \lambda_2 \lambda^T_3}C_{\varnothing \lambda_1 \lambda_2^T}C_{\varnothing\lambda_3 \lambda_1^T}Tr_{\square} V,
\]
where the topological vertex is given by $C_{\lambda_1,\lambda_2,\lambda_3},$ a Laurent series in $q^{\frac{1}{2}}$ with integer coefficients
indexed by three partitions. The closed topological string partition function $Z_{\FF_1}$ of local $\FF_1$ with exceptional curve $C$ is given by,

\[
Z_{\FF_1} = \sum_{\lambda_1, \lambda_2, \lambda_3, \lambda_4} (-1)^{\sum_i |\lambda_i|}Q^{(|\lambda_4|-|\lambda_1|-|\lambda_3|)C}Q^{(|\lambda_1| + |\lambda_2| + |\lambda_3|)H}q^{\sum_i \kappa_{\lambda_i}}C_{\varnothing \lambda_1 \lambda_4^T}C_{\varnothing \lambda_2 \lambda_1^T}C_{\varnothing \lambda_3 \lambda_2^T}C_{\varnothing \lambda_4 \lambda_3^T}
\]

\begin{figure}[ht]
    \centering

\tikzset{every picture/.style={line width=0.75pt}} 

\tikzset{every picture/.style={line width=0.75pt}} 

\begin{tikzpicture}[x=0.75pt,y=0.75pt,yscale=-1,xscale=1]

\draw    (150.33,189.67) -- (111.03,230.82) ;
\draw    (241,189.67) -- (295.1,206.42) ;
\draw    (134.97,55.75) -- (150,101.33) ;
\draw    (150,101.33) -- (150.33,189.67) ;
\draw [shift={(150.19,151.5)}, rotate = 269.78] [color={rgb, 255:red, 0; green, 0; blue, 0 }  ][line width=0.75]    (10.93,-3.29) .. controls (6.95,-1.4) and (3.31,-0.3) .. (0,0) .. controls (3.31,0.3) and (6.95,1.4) .. (10.93,3.29)   ;
\draw    (150.33,189.67) -- (241,189.67) ;
\draw [shift={(201.67,189.67)}, rotate = 180] [color={rgb, 255:red, 0; green, 0; blue, 0 }  ][line width=0.75]    (10.93,-4.9) .. controls (6.95,-2.3) and (3.31,-0.67) .. (0,0) .. controls (3.31,0.67) and (6.95,2.3) .. (10.93,4.9)   ;
\draw    (150,101.33) -- (241,189.67) ;
\draw [shift={(190.48,140.62)}, rotate = 44.15] [color={rgb, 255:red, 0; green, 0; blue, 0 }  ][line width=0.75]    (10.93,-3.29) .. controls (6.95,-1.4) and (3.31,-0.3) .. (0,0) .. controls (3.31,0.3) and (6.95,1.4) .. (10.93,3.29)   ;

\draw (119,72.07) node [anchor=north west][inner sep=0.75pt]    {$\square $};
\draw (120.33,135.07) node [anchor=north west][inner sep=0.75pt]    {$R_{2}$};
\draw (199.33,126.73) node [anchor=north west][inner sep=0.75pt]    {$R_{3}$};
\draw (181.33,205.07) node [anchor=north west][inner sep=0.75pt]    {$R_{1}$};
\end{tikzpicture}
\qquad\quad
\begin{tikzpicture}[x=0.75pt,y=0.75pt,yscale=-1,xscale=1]

\draw    (170.33,209.67) -- (133,247.67) ;
\draw    (320.5,209.17) -- (374.6,225.92) ;
\draw    (131,79.75) -- (170,121.33) ;
\draw    (170,121.33) -- (170.33,209.67) ;
\draw [shift={(170.19,171.5)}, rotate = 269.78] [color={rgb, 255:red, 0; green, 0; blue, 0 }  ][line width=0.75]    (10.93,-3.29) .. controls (6.95,-1.4) and (3.31,-0.3) .. (0,0) .. controls (3.31,0.3) and (6.95,1.4) .. (10.93,3.29)   ;
\draw    (170.33,209.67) -- (320.5,209.17) ;
\draw [shift={(251.42,209.4)}, rotate = 179.81] [color={rgb, 255:red, 0; green, 0; blue, 0 }  ][line width=0.75]    (10.93,-4.9) .. controls (6.95,-2.3) and (3.31,-0.67) .. (0,0) .. controls (3.31,0.67) and (6.95,2.3) .. (10.93,4.9)   ;
\draw    (229.5,120.83) -- (320.5,209.17) ;
\draw [shift={(269.98,160.12)}, rotate = 44.15] [color={rgb, 255:red, 0; green, 0; blue, 0 }  ][line width=0.75]    (10.93,-3.29) .. controls (6.95,-1.4) and (3.31,-0.3) .. (0,0) .. controls (3.31,0.3) and (6.95,1.4) .. (10.93,3.29)   ;
\draw    (170,121.33) -- (229.5,120.83) ;
\draw [shift={(192.75,121.14)}, rotate = 359.52] [color={rgb, 255:red, 0; green, 0; blue, 0 }  ][line width=0.75]    (10.93,-3.29) .. controls (6.95,-1.4) and (3.31,-0.3) .. (0,0) .. controls (3.31,0.3) and (6.95,1.4) .. (10.93,3.29)   ;
\draw    (229.67,78.33) -- (229.5,120.83) ;

\draw (233.33,221.57) node [anchor=north west][inner sep=0.75pt]    {$R_{2}$};
\draw (283.33,147.73) node [anchor=north west][inner sep=0.75pt]    {$R_{3}$};
\draw (141.33,151.57) node [anchor=north west][inner sep=0.75pt]    {$R_{1}$};
\draw (186,93.57) node [anchor=north west][inner sep=0.75pt]    {$R_{4}$};
\end{tikzpicture}
    \caption{The toric diagrams of local $\bP^2$ (left) and local $\FF_1$ (right).  Internal edges have representations $R_i$ attached, where for the open calculation of local $\bP^2$ with winding $1$, we have the fundamental representation $\square$ on the external edge.}
    \label{fig:toric_diagrams}
\end{figure}
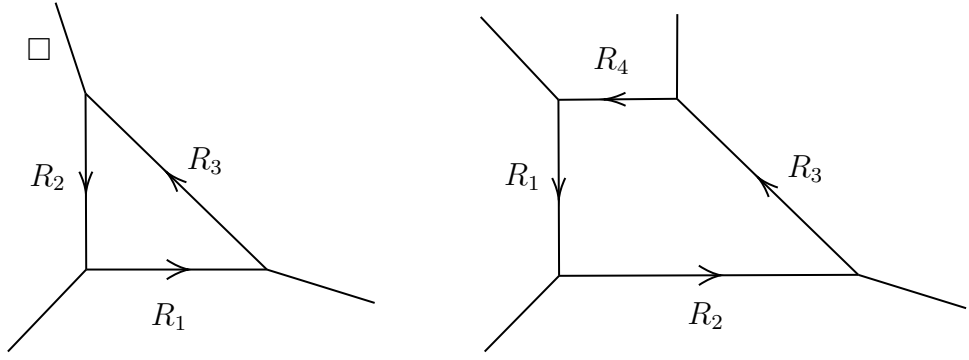

We begin with degree one, $d=1$. 
Then all representation labels are either trivial or $\square$, simplifying the expressions for the vertex functions $C_{\lambda\mu\nu}$ (see \cite{AKMV}, pg. 53 for explicit calculations). Extracting integer invariants, we get
\begin{equation}
\label{eq:g=0_closed_BPS}
   n_0(K_{\FF_1}, \pi^*H - C) = -2, \quad n_g(K_{\FF_1}, \pi^*H - C) = 0 \text{ for }g>0 
\end{equation}
and 
\begin{equation}
    \label{eq:g=0_open_BPS}
    n^{open}_0(K_{\bP^2}/L, H+\beta_0, (1)) = 2, \quad n^{open}_g(K_{\bP^2}/L, H+\beta_0, (1)) = 0 \text{ for }g>0
\end{equation}

For degrees $2 \leq d \leq 4$, we 
can use the calculations for the refined
Topological Vertex in \cite[\S 5.6.1]{IKV},
then un-refine them.  These are listed in 
\ref{table:open_closed_BPS}. 

\begin{table}[h]
\centering
 \begin{tabular}{||c c c||} 
 \hline
 $d$ & $\sum_{j_L, j_R} N_C^{(j_L, j_R)}(j_L, j_R)$ & $\widehat{f}_{\square \,\varnothing \, \varnothing}$  \\ [0.5ex] 
 \hline\hline
 2 & $(0,2)$ & $5$ \\ 
 3 & $(\frac{1}{2}, 4) \oplus (0, \frac{7}{2}) \oplus (0, \frac{5}{2})$ & $-(32+9z)$ \\ 
 4 & $\left(\frac{3}{2}, \frac{13}{2}\right) \oplus (1, 6) \oplus (1, 5) \oplus 2\left(\frac{1}{2}, \frac{11}{2}\right) \oplus (0, 6) \oplus 2\left(\frac{1}{2}, \frac{9}{2}\right)$ & \\
 & $\oplus (0, 5) \oplus \left(\frac{1}{2}, \frac{7}{2}\right) \oplus 2(0, 4) \oplus (0, 3) \oplus (0, 2)$  & $286+288z+108z^2+14z^3$  \\
 \hline
 \end{tabular}
     \caption{Open and closed BPS degeneracies.}
    \label{table:open_closed_BPS}
\end{table}

The table lists the Gopakumar-Vafa invariants of local $\FF_1$ in curve classes $(d-1)B + dF$, with the $SU(2)_L \times SU(2)_R$ spin content of BPS states supported on those curves.  To un-refine, we want to ignore $j_R,$
meaning each such representation contributes a factor
of $(2j_R+1),$ so $(j_L,j_R)$
generates a term
$(2j_R+1)(\q^{-2j_L} + \q^{-2j_L+2} + \ldots + \q^{2j_L-2} + \q^{2j_L})$ to the curve counting function (see also \cite[Appendix B]{HKR} for some of these BPS invariants). For given $d$, we compute the total contribution from all of the $(j_L, j_R)$. After making a change of variables $z := (\q^{\frac{1}{2}} - \q^{\frac{-1}{2}})^2$, the computed expression matches, up to a convention-dependent overall sign, the coefficient of $Q^d$ in $-\widehat{f}_{\square\,\varnothing\,\varnothing}$ (defined in Section \ref{sec:GRZ}) containing the open-BPS invariants in curve class $dH + \beta_0$.\footnote{In Table \ref{table:open_closed_BPS}, the coefficient of $(0,4)$ is listed in \cite{IKV} as 1 instead of 2 --- a typo, we believe. Indeed, the coefficient should be 2 in order to recover the genus-0 BPS invariant of local $\FF_1$ in class $3B+4F$ in the limit $\q\rightarrow 1$, which is 286. This invariant is computed in \cite{LLW} \cite{CKYZ} \cite{AKMV}.} This completes the verification
of Corollary \ref{cor:winding-1} for $\bP^2$
for degrees $d \leq 4.$

\begin{appendix}

\section{Evidence for Conjecture \ref{conj:GV_LMOV}: GV and LMOV invariants}

\label{sec:LMOV_GV}

To provide evidence for Conjecture \ref{conj:GV_LMOV}, we list
here some tables of Gopakumar-Vafa (GV) and Labastida-Mari\~no-Ooguri-Vafa (LMOV) invariants.  
The GV invariants
$n_g(K_{\FF_1}, \pi^*\beta - wC)$
for $K_{\FF_1}$ in curve class $\pi^*dH-wC$ come from \cite{HKR}.
We abbreviate them $n_g(d,w)$ here, for simplicity.
The LMOV invariants
$N^{\mathrm{LMOV}}_{g, (w)}(K_\bP/L, \beta)$
in representation $(w)$
are obtained from \cite{AKMV}.\footnote{We thank Albrecht Klemm and Jie Gu for calculating $N^{\mathrm{LMOV}}_{g, (w)}(K_{\bP^2}/L, dH)$ for $w = 3, 4$ and $d \leq 5$,
beyond what appears in \cite{AKMV}.}  We abbreviate them
$N^{\mathrm{LMOV}}_{g,(w)}(d)$ here for simplicity.
These are listed for total winding $w\leq 4,$ genus $g\leq 3,$ and degree $d \leq 5.$

\begin{figure}[H]
\centering
\begin{subfigure}[b]{0.45\textwidth}
\centering
 \begin{tabular}{||c c c c c||} 
 \hline
 $d$ & $g = 0$ & $g = 1$ & $g = 2$ & $g = 3$ \\ [0.5ex] 
 \hline\hline
 1 & -2 & 0 & 0 & 0\\ 
 2 & 5 & 0 & 0 & 0 \\
 3 & -32 & 9 & 0 & 0 \\
 4 & 286 & -288 & 108 & -14\\
 5 & -3038 & 6984 & -7506 & 4519\\
 \hline
 \end{tabular}
 \caption{$w=1$:  $n_g(d,1) = (-1)^gN^{\mathrm{LMOV}}_{g,(1)}(d)$}
 \label{subtable:winding1_LMOV}
\end{subfigure}
\hfill
\begin{subfigure}[b]{0.45\textwidth}
\centering
 \begin{tabular}{||c c c c c||} 
 \hline
 $d$ & $g = 0$ & $g = 1$ & $g = 2$ & $g = 3$ \\ [0.5ex] 
 \hline\hline
 1 & 0 & 0 & 0 & 0\\ 
 2 & 0 & 0 & 0 & 0 \\
 3 & 7 & 0 & 0 & 0 \\
 4 & -110 & 68 & -12 & 0\\
 5 & 1651 & -2938 & 2353 & -992\\
 \hline
 \end{tabular}
 \caption{$w=2$:  $n_g(d,2) = (-1)^gN^{\mathrm{LMOV}}_{g,(2)}(d)$}
 \label{subtable:winding2_LMOV}
\end{subfigure}
\hfill

\vspace{0.5cm}

\begin{subfigure}[b]{0.45\textwidth}
\centering
 \begin{tabular}{||c c c c c||} 
 \hline
 $d$ & $g = 0$ & $g = 1$ & $g = 2$ & $g = 3$ \\ [0.5ex] 
 \hline\hline
 1 & 0 & 0 & 0 & 0\\ 
 2 & 0 & 0 & 0 & 0 \\
 3 & 0 & 0 & 0 & 0 \\
 4 & 9 & 0 & 0 & 0\\
 5 & -288 & 300 & -116 & 15\\
 \hline
 \end{tabular}
 \caption{$w=3$:  $n_g(d,3) = (-1)^gN^{\mathrm{LMOV}}_{g,(3)}(d)$} \label{subtable:winding3_LMOV}
\end{subfigure}
\hfill
\begin{subfigure}[b]{0.45\textwidth}
\centering
 \begin{tabular}{||c c c c c||} 
 \hline
 $d$ & $g = 0$ & $g = 1$ & $g = 2$ & $g = 3$ \\ [0.5ex] 
 \hline\hline
 1 & 0 & 0 & 0 & 0\\ 
 2 & 0 & 0 & 0 & 0 \\
 3 & 0 & 0 & 0 & 0 \\
 4 & 0 & 0 & 0 & 0\\
 5 & 11 & 0 & 0 & 0\\
 \hline
 \end{tabular}
 \caption{$w=4$:  $n_g(d,4) = (-1)^gN^{\mathrm{LMOV}}_{g,(4)}(d)$}
 \label{subtable:winding4_LMOV}
\end{subfigure}

\caption{GV invariants $n_g(K_{\FF_1}, \pi^*dH - wC)$ equal $(-1)^g$ times LMOV invariants $N^{\mathrm{LMOV}}_{g, (w)}(K_{\bP^2}/L, dH)$ for various genera and degrees, for representations/total winding $w = 1, 2, 3, 4$. Tables shown are GV invariants.}

\label{fig:LMOV_GV_tables}
\end{figure}

\section{Combinatorial results}
\label{sec:combinatorics}

\subsection{Laurent series inversion and substitution}
Recall that we write $[x^m]f(x)$ to refer to the coefficient of $x^m$ in a series $f(x)$.

Let $f(x)$ be a formal Laurent series with complex coefficients of the form
$f(x)=\sum_{k\geq -1}f_kx^k$ with $f_{-1}\neq 0$. A compositional inverse of $f$ is a series $g(z)$ that solves $f(g(z))=z$. 

\begin{lemma}[Lagrange inversion for Laurent series]
\label{lem:inversion}
As an element $g\in\CC\lfor z^{-1}\rfor$, there is a unique compositional inverse $g$ of $f$ given by the expression
\[ g(z) = \sum_{k> 0}\frac{z^{-k}}{k}[x^{-1}]f^k. \]
Moreover, $f(x)$ of the given form is uniquely determined by $g(z)$.
\end{lemma}

\begin{proof}
The multiplicative inverse of $f$ can be written as
\[ \frac{1}{f(x)}=\frac{x}{h(x)}, \quad h(x)=xf(x)=\sum_{k\geq -1}f_kx^{k+1} \]
Note that $h(x)$ is a formal power series with $h(x)\neq 0$. The Lagrange inversion formula states that there exists a unique power series $\tilde{g}(z)$ which is the compositional inverse of $1/f(x)$. From this one can derive the Lagrange-B\"urmann formula for formal power series, which gives an explicit expression for the coefficients of $\tilde{g}(z)$. This is \cite[Theorem 7.5]{burmann}, which gives
\[ \tilde{g}(z) = \sum_{k>0}\frac{z^k}{k}[x^{k-1}]h^k = \sum_{k>0}\frac{z^k}{k}[x^{-1}]f^k. \]
The equation $1/f(\tilde{g}(z))=z$ is equivalent to $f(\tilde{g}(z))=z^{-1}$ and to $f(\tilde{g}(z^{-1}))=z$. So the inverse of $f(x)$ is $g(z)=\tilde{g}(z^{-1})$ as claimed. Conversely, for any compositional inverse $g(z)$ of $f(x)$ we have that $\tilde{g}(z)=g(z^{-1})$ is a compositional inverse of $1/f(x)$. If $g(z)$ is a formal power series in $z^{-1}$, then $\tilde{g}(z)$ is a formal power series in $z$. This is unique by the Lagrange inversion formula for formal power series. Hence, $g(z)$ is unique as an element of $\CC\lfor z^{-1}\rfor$. The Lagrange inversion formula for formal power series also states that the power series $1/f(z)$ is uniquely determined by $\tilde{g}(z)$. Hence, $f(x)$ of the given form is uniquely determined by $g(z)$.
\end{proof}

\begin{definition}
\label{def:ord}
The ring $\CC\llbracket x^{-1}\rrbracket[x]$ is a field. An element $f\in\CC\llbracket x^{-1}\rrbracket[x]$ takes the form
\[ f(x)=\sum_{k=-\infty}^{\ord(f)} f_kx^k, \quad f_{\ord(f)} \neq 0, \]
and we call $\ord(f)$ the \emph{order} of $f$.
\end{definition}

\begin{lemma}
\label{lem:ord1}
Let 
\[ f(x)=\sum_{k=-\infty}^1 f_kx^k \]
be a Laurent series with $\ord(f)=1$. There exists a unique and explicitly computable formal Laurent series $g(z)$ of order $\ord(g)=1$ that is the compositional inverse of $f$.
\end{lemma}

\begin{proof}
By Lemma \ref{lem:inversion} the series 
\[ \tilde{f}(x) := f(x^{-1})=\sum_{k=-1}^\infty f_{-k}x^k \]
has inverse
\[ \tilde{g}(z) = \sum_{k=1}^\infty \frac{z^{-k}}{k}[x^{-1}]\tilde{f}^k = \sum_{k=1}^\infty \frac{z^{-k}}{k}[x^1]f^k = f_1z^{-1}\left(1+\sum_{k=1}^\infty \frac{z^{-k}}{(k+1)f_1}[x^1]f^{k+1}\right). \]
This means $\tilde{f}(\tilde{g}(z))=f(\tilde{g}(z)^{-1})=z$, hence $f(g(z))=z$ for
\[ g(z)=\tilde{g}(z)^{-1} = f_1^{-1}z\left(1+\sum_{k=1}^\infty \frac{z^{-k}}{(k+1)f_1}[x^1]f^{k+1}\right)^{-1}, \]
and this is a Laurent series with $\ord(g)=1$.
\end{proof}

\begin{lemma}
\label{lemA4}
If $f\in\CC\llbracket x^{-1}\rrbracket[x]$ and $f\neq 0$, then $f'/f\in\CC\llbracket x^{-1}\rrbracket \subset \CC\llbracket x^{-1}\rrbracket[x]$ with $\ord(f'/f)=-1$ and
\[ [x^{-1}](f'/f) = \ord(f). \]
\end{lemma}

\begin{proof}
Write $n=\ord(f)$ and 
\[ f(x) = \sum_{k=-\infty}^{n} f_kx^k. \]
Then
\[ f'(x) = \sum_{k=-\infty}^{n} kf_kx^{k-1} \]
and
\[ \frac{f'(x)}{f(x)} = \frac{nf_nx^{n-1}\left(1+\sum_{k=1}^\infty \frac{n-k}{n}\frac{f_{n-k}}{f_n}x^{-k}\right)}{f_nx^n\left(1+\sum_{k=1}^\infty \frac{f_{n-k}}{f_n}x^{-k}\right)} = nx^{-1}\left(1 + \sum_{k=1}^\infty a_k x^{-1}\right) \]
for some $a_k\in\CC$. Hence, $\ord(f'/f)=-1$ and $[x^{-1}](f'/f)=n$ as claimed.
\end{proof}

\begin{definition}
\label{defi:consta}
For $f\in \CC\llbracket x^{-1}\rrbracket[x]$, we write $\res_x(f)=[x^{-1}]f$ and $\const_x(f)=[x^0]f$. Moreover, for $f\in (\CC\llbracket x^{-1}\rrbracket[x])\llbracket t\rrbracket$,
we write $f = \sum_{k=0}^\infty f_k(x) t^k$ and define
\[ \res_x(f) = \sum_{t=0}^\infty \res_x(f_k)t^k, \qquad \const_x(f) = \sum_{t=0}^\infty \const_x(f_k)t^k.  \]
\end{definition}
There is a notion of formal derivation $f'=\partial_xf$ for $f \in \CC\llbracket x^{-1}\rrbracket[x]$ and then we define $\partial_x$ by $\CC$-linear extension also for $f \in(\CC\llbracket x^{-1}\rrbracket[x])\llbracket t \rrbracket$.

\begin{lemma}
\label{lem:subst}
Let $f \in (\CC\llbracket x^{-1}\rrbracket[x])\llbracket t \rrbracket$ and $x \in \CC\llbracket u^{-1}\rrbracket[u]\setminus\{0\}$. Then
\[ \res_u\Big(f(x(u))x'(u)\Big) = \ord(x)\res_x\left(f\right), \]
and
\[ \const_u\left(f(x(u))\frac{u\partial_ux(u)}{x(u)}\right) = \ord(x)\const_x(f). \]
\end{lemma}

\begin{proof}
Note that $\im(\partial_x)=\ker(\res_x)$, so we can write $f(x)=\res_x(f)x^{-1}+\partial_x \widetilde{f}(x)$. Insertion yields
\[ f(x(u))x'(u) = \left(\res_x(f)x(u)^{-1} + \partial_x \widetilde{f}(x(u))\right)x'(u) = \res_x(f)x(u)^{-1}x'(u) + \partial_u \widetilde{f}(x(u)). \]
Lemma~\ref{lemA4} gives $\res_u(x'(u)/x(u))=\ord(x)$ and $\res_u(\partial_u \widetilde{f}(x(u))=0$, since $\im(\partial_u)=\ker(\res_u)$. Hence, $\res_u\left(f(x(u))x'(u)\right) = \ord(x)\res_x\left(f(x)\right)$ as claimed. The second statement follows from the first statement applied to $f(x)/x$.
\end{proof}

\subsection{Bell polynomial identities}
In the following we consider polynomials and power series over a $\CC$-algebra $A$ that is an integral domain.

\begin{definition}
For a sequence of complex numbers $a=(a_1,a_2,\ldots)$ and $k\ge 0$, let $\widehat{B}_{n,k}(a)$ be the coefficient of $x^n$ in $(a_1x+a_2x^2+\ldots)^k$,
\[ \left(\sum_{i=1}^\infty a_ix^i\right)^k = \sum_{n=1}^\infty \widehat{B}_{n,k}(a) x^n. \]
For $n<k$ we have $\widehat{B}_{n,k}(a)=0$ and we define $\widehat{B}_{n,k}(a):=0$ whenever $k<0$.
\end{definition}

\begin{proposition}
\label{prop:bell}
$\widehat{B}_{n,k}(a)$ is equal to the partial ordinary Bell polynomial
\[ \widehat{B}_{n,k}(a) = \widehat{B}_{n,k}(a_1,\ldots,a_{n-k+1}) = \sum_{\substack{(j_1,\ldots,j_{n-k+1}) \in \mathbb{Z}_{\geq 0}^{n-k+1} \\ \sum j_i=k \\ \sum ij_i=n}}  \frac{k!}{j_1!\cdots j_{n-k+1}!} a_1^{j_1}\cdots a_{n-k+1}^{j_{n-k+1}}. \]
\end{proposition}

\begin{proof}
E.T.\,Bell introduced the Bell polynomials for this property, see \cite{Bell}.
\end{proof}

Recall that we write $[x^m]g(x)$ to refer to the coefficient of $x^m$ in $g(x)$.

\begin{lemma}
\label{lem:0}
For a power series $f=\sum_{k\geq 0}a_kx^k$ and $n,m\geq 0$, we have
\[ [x^m]f^n = \sum_{k=0}^m \binom{n}{k}a_0^{n-k}\widehat{B}_{m,k}(a). \]
\end{lemma}

\begin{proof}
For $a_0=0$ the claimed equation reduces to $[x^m]f^n=\widehat{B}_{m,n}(a)$, which is Proposition \ref{prop:bell}.
In general, we can write $f(x)=a_0+g(x)$, where $g(x)=\sum_{k>0}a_kx^k$ has no constant term. Then
\[ [x^m]f^n = \sum_{k=0}^m \binom{n}{k}a_0^{n-k}[x^m]g^k. \]
It remains to observe that $[x^m]g^k=\widehat{B}_{m,k}(a)$ since $g$ has no constant term and otherwise the same coefficients as $f$. This completes the proof.
\end{proof}

\begin{lemma} 
\label{lem:1}
Let $n,k \in \mathbb{Z}_{\geq 0}$, $\tau\in\mathbb{C}$, and $a=(a_1,a_2,\ldots)$. Let $\alpha(l,m)$ be a polynomial in $l$ and $m$ of degree at most one taking non-zero values for $l,m\in\{0,...,n\}$. We have
\[ \binom{\tau}{k}\widehat{B}_{n,k}(a) = \sum_{m=0}^n\sum_{l=0}^m \frac{\alpha(0,0)}{\alpha(l,m)}\binom{\tau-\alpha(l,m)}{k-l}\binom{\alpha(l,m)}{l}\widehat{B}_{m,l}(a)\widehat{B}_{n-m,k-l}(a). \]
\end{lemma}

\begin{proof}
The exponential Bell polynomials $B_{n,k}(a)$ are by definition related to $\widehat{B}_{n,k}(a)$ by
\begin{equation}
\widehat{B}_{n,k}(a_1,\ldots,a_{n-k+1}) = \frac{k!}{n!}B_{n,k}(1!a_1,\ldots,(n-k+1)!a_{n-k+1}). 
\label{Bell-types}
\end{equation}
The second equation of \cite{BGW}, Corollary 10, is
\[ \binom{\tau}{k}B_{n,k}(x) = \sum_{l=0}^k\sum_{m=l}^n \frac{\alpha(0,0)}{\alpha(l,m)}\binom{\tau-\alpha(l,m)}{k-l}\binom{\alpha(l,m)}{l}\frac{\binom{n}{m}}{\binom{k}{l}}B_{m,l}(x)B_{n-m,k-l}(x). \]
We substitute $x_i=i!a_i$ and use \eqref{Bell-types}
 to find that $\binom{\tau}{k}B_{n,k}(a)$ equals
\[ \sum_{l=0}^k\sum_{m=l}^n \frac{\alpha(0,0)}{\alpha(l,m)}\binom{\tau-\alpha(l,m)}{k-l}\binom{\alpha(l,m)}{l}\underbrace{\frac{\binom{n}{m}}{\binom{k}{l}}\frac{k!}{n!}\frac{m!}{l!}\frac{(n-m)!}{(k-l)!}}_{=1}B_{m,l}(a)B_{n-m,k-l}(a). \]
The binomials and factorials in the last sum that do not show in the asserted identity combine to the factor $1$. It remains to observe that the last summation only differs from the one in the assertion by removing terms where $l>k$ and these terms contain a zero factor $\widehat{B}_{n-m,k-l}(a)$, so we have verified the assertion.
\end{proof}

\begin{lemma}
\label{lem:2}
For a power series $f(x)=1+\sum_{k\geq 1}f_kx^k$ one has, for all $n\geq 0$,
\[ [x^{n-1}]f^n = \sum_{\substack{a>0,b\geq 0 \\ a+b=n}}\frac{1}{a}[x^{a-1}]f^{a}[x^b]f^{b}. \]
\end{lemma}

\begin{proof}
By Lemma \ref{lem:0}, the right hand side of the claim can be written as
\begin{eqnarray*}
&& \sum_{\substack{a>0,b\geq 0 \\ a+b=n}} \frac{1}{a} \left(\sum_{k_1=0}^{a-1} \binom{a}{k_1}\widehat{B}_{a-1,k_1}(f)\right) \left(\sum_{k_2=0}^b \binom{b}{k_2}\widehat{B}_{b,k_2}(f)\right) \\
&=& \sum_{\substack{a>0,b\geq 0 \\ a+b=n}}\sum_{k_1=0}^{a-1}\sum_{k_2=0}^b \frac{1}{a}\binom{a}{k_1}\binom{b}{k_2} \widehat{B}_{a-1,k_1}(f)\widehat{B}_{b,k_2}(f) \\
&\stackrel{\substack{k:=k_1+k_2\\b=n-a}}{=}& \sum_{k=0}^{n-1}\sum_{a=1}^n\sum_{k_1=0}^{a-1} \frac{1}{a}\binom{a}{k_1}\binom{n-a}{k-k_1} \widehat{B}_{a-1,k_1}(f)\widehat{B}_{n-a,k-k_1}(f) \\
&\stackrel{\substack{m:=a-1\\l:=k_1}}{=}& \sum_{k=0}^{n-1}\sum_{m=0}^{n-1}\sum_{l=0}^m \frac{1}{m+1}\binom{m+1}{l}\binom{n-(m+1)}{k-l} \widehat{B}_{m,l}(f)\widehat{B}_{n-m-1,k-l}(f) \\
&\stackrel{\textrm{Lem.\,\ref{lem:1}}}{=}& \sum_{k=0}^{n-1} \binom{n}{k}\widehat{B}_{n-1,k}(f) = [x^{n-1}]f^n.
\end{eqnarray*}
In the last equality we used Lemma~\ref{lem:0} and in the one before we used Lemma~\ref{lem:1} for the parameters $n-1,k,\tau=n$ and $\alpha(l,m)=m+1$.
\end{proof}

\begin{proposition}
\label{lem:3}
For a power series $f(x)=1+\sum_{k\geq 1}f_kx^k$ one has
\[ \textup{exp}\left(\sum_{k>0}\frac{1}{k}[x^k]f^kz^k\right) = \sum_{k>0}\frac{1}{k}[x^{k-1}]f^kz^{k-1}. \]
\end{proposition}

\begin{proof}
These functions have the same value $1$ at $z=0$, so it is enough to show that their derivatives agree\footnote{We thank Ezra Getzler for suggesting this approach.}. Deriving the equation $\textup{exp}(A)=B$ gives $BA'=B'$, thus it suffices to show
\[ \left(\sum_{k>0}\frac{1}{k}[x^{k-1}]f^kz^{k-1}\right)\left(\sum_{k>0}[x^k]f^kz^{k-1}\right) = \sum_{k>0}\frac{k-1}{k}[x^{k-1}]f^kz^{k-2}. \]
By comparing the $z^{n-2}$-coefficients, this translates, for every $n$, into the identity
\[ \sum_{\substack{a,b>0 \\ a+b=n}}\frac{1}{a}[x^{a-1}]f^{a}[x^b]f^{b} = \frac{n-1}{n}[x^{n-1}]f^n. \]
We can add $\frac{1}{n}[x^{n-1}]f^n$ to both sides and incorporate it into the sum as $b=0$ to get the equivalent identity
\[ \sum_{\substack{a>0,b\geq 0 \\ a+b=n}}\frac{1}{a}[x^{a-1}]f^{a}[x^b]f^{b} = [x^{n-1}]f^n. \]
that we recognize as the statement of Lemma \ref{lem:2} which concludes the proof.
\end{proof}

\end{appendix}

\end{document}